\documentclass{amsart}[a4]

\usepackage{amssymb}
\usepackage{colonequals}
\usepackage{comment}
\usepackage{dsfont}
\usepackage{tikz-cd}
\tikzset{labl/.style={anchor=south, rotate=270, inner sep=.5mm}}
\usepackage{eucal}
\usepackage{hyperref}
\hypersetup{
  bookmarksnumbered=true,
  colorlinks=true,
  linkcolor=blue,
  citecolor=blue,
  filecolor=blue,
  menucolor=blue,
  urlcolor=blue,
  bookmarksopen=true,
  bookmarksdepth=2,
  pageanchor=true
}

\newtheorem{theorem}{Theorem}[section]
\newtheorem{corollary}[theorem]{Corollary}
\newtheorem{lemma}[theorem]{Lemma}
\newtheorem{proposition}[theorem]{Proposition}

\theoremstyle{definition}
\newtheorem{definition}[theorem]{Definition}
\newtheorem{example}[theorem]{Example}

\theoremstyle{remark}
\newtheorem{remark}[theorem]{Remark}

\numberwithin{equation}{section}

\newcommand{\ann}{\operatorname{ann}}
\newcommand{\Coh}{\operatorname{Coh}}
\newcommand{\colim}{\operatorname{colim}}
\newcommand{\comp}{\operatorname{comp}}

\newcommand{\End}{\operatorname{End}}
\newcommand{\Ext}{\operatorname{Ext}}

\newcommand{\Hom}{\operatorname{Hom}}
\newcommand{\fHom}{\operatorname{\mathcal{H}\!\!\;\mathit{om}}}
\newcommand{\RHom}{\operatorname{{\mathbf R}Hom}}
\newcommand{\inc}{\operatorname{inc}}

\newcommand{\Inj}{\operatorname{Inj}}
\renewcommand{\Im}{\operatorname{Im}}
\newcommand{\Ker}{\operatorname{Ker}}
\newcommand{\loc}{\operatorname{loc}}
\newcommand{\Loc}{\operatorname{Loc}}
\newcommand{\MaxSpec}{\operatorname{MaxSpec}}
\renewcommand{\mod}{\operatorname{mod}}
\newcommand{\Mod}{\operatorname{Mod}}
\newcommand{\Proj}{\operatorname{Proj}}
\newcommand{\res}{\operatorname{res}}
\newcommand{\Spc}{\operatorname{Spc}}
\newcommand{\Spec}{\operatorname{Spec}}
\newcommand{\stmod}{\operatorname{stmod}}
\newcommand{\StMod}{\operatorname{StMod}}
\newcommand{\supp}{\operatorname{supp}}
\newcommand{\thick}{\operatorname{thick}}
\newcommand{\Thick}{\operatorname{Thick}}

\newcommand{\da}{{\downarrow}}
\newcommand{\iso}{\xrightarrow{\raisebox{-.6ex}[0ex][0ex]{$\scriptstyle{\sim}$}}}
\newcommand{\kos}[2]{{#1}/\!\!/{#2}} 
\newcommand{\leftiso}{\xleftarrow{\raisebox{-.6ex}[0ex][0ex]{$\scriptstyle{\sim}$}}}
\newcommand{\longiso}{\xrightarrow{\
    \raisebox{-.6ex}[0ex][0ex]{$\scriptstyle{\sim}$}\ }}
\newcommand{\lto}{\longrightarrow}
\newcommand{\lotimes}{\otimes^{\mathbf L}}
\newcommand{\ua}{{\uparrow}}

\newcommand{\Ab}{\mathrm{Ab}}
\newcommand{\fin}{\mathrm{fin}}
\newcommand{\id}{\mathrm{id}}
\newcommand{\op}{\mathrm{op}}
\newcommand{\perf}{\mathrm{perf}}

\newcommand{\cat}{\mathcal}
\newcommand{\one}{\mathds 1}

\newcommand{\bfD}{\mathbf D} 
\newcommand{\bfK}{\mathbf K}

\newcommand{\bbZ}{\mathbb Z}

\newcommand{\fa}{\mathfrak{a}}
\newcommand{\fb}{\mathfrak{b}}

\newcommand{\fp}{\mathfrak{p}}
\newcommand{\fq}{\mathfrak{q}}
\newcommand{\fz}{\mathfrak{z}}

\newcommand{\Si}{\Sigma}

\let\Gamma\varGamma
\let\Delta\varDelta
\let\Theta\varTheta
\let\Lambda\varLambda
\let\Xi\varXi
\let\Pi\varPi
\let\Sigma\varSigma
\let\Upsilon\varUpsilon
\let\Phi\varPhi
\let\Psi\varPsi
\let\Omega\varOmega

\begin{document}

\title{Tensor triangular geometry\\
 Notes for an Oberwolfach Seminar}

\author{Henning Krause}
\address{Fakult\"at f\"ur Mathematik\\
Universit\"at Bielefeld\\ D-33501 Bielefeld\\ Germany}
\email{hkrause@math.uni-bielefeld.de}

\thanks{\today}

\maketitle

\setcounter{tocdepth}{1} 
\tableofcontents


\section*{Overview}

These are the notes from lectures I gave at the Oberwolfach Seminar
‘Tensor Triangular Geometry and Interactions’ which was held in
October 2025.\footnote{These notes are complemented by notes of
  Nat\`alia Castellana and Greg Stevenson \cite{Stevenson:2026a}. They will be published
  together in a forthcoming volume of the `Oberwolfach Seminars' book
  series by Birkhäuser Science.}

The aim of these notes is twofold: We develop notions of support for
triangulated categories, and we apply them to classify thick and
localising tensor ideals of categories that arise in modular
representation theory of finite groups.

There are two basic situations: Either the triangulated category
admits a monoidal structure, or it admits an action of a
graded-commutative ring. In order to cover both we introduce the
notion of a triangulated category that is fibred over a space. The space may
be the prime ideal spectrum of a tensor triangulated category or the
prime ideal spectrum of a ring that is acting.

There is some further dichotomy: The triangulated categories we study
are either essentially small, or they admit set-indexed coproducts and
are compactly generated. Both settings are related; so we study the
support for a compactly generated category by introducing first the
support for its subcategory of compact objects.

The ultimate goal is to establish a stratification of a triangulated
category via an appropriate space, and to actually compute it. This
amounts to classifying thick and localising subcategories.  For
instance, we achieve this for derived categories of modular
representations. More specifically, for any finite group the
stratification of its derived category involves the spectrum of the
cohomology ring. This builds on Quillen's stratification of group
cohomology, and one may think of this process as a categorification of
the Quillen stratification. This prototypical example explains the
term `stratification' in the triangulated context.

These notes are not exhaustive, though fairly self-contained. In
particular, only few references are included. The guiding principle
has been to develop all concepts that are need for our applications in
modular representation theory, and to do it in sufficient generality
that the techniques can be easily used in other contexts.

\subsection*{Acknowledgements}
I am grateful to the participants of the Oberwolfach Seminar and also
to my co-organizers Nat\`alia Castellana and Greg Stevenson for
their inspiration and for many valuable comments. Thanks are also due to
the Oberwolfach Institute for their support of the seminar.
This work was supported by the Deutsche Forschungsgemeinschaft
(SFB-TRR 358/1 2023 - 491392403).

\section{The lattice of thick subcategories}

The thick subcategories of a triangulated category form a
lattice. This is a fundamental object of interest and we collect some
basic properties. A useful tool is the category of cohomological
functors from a triangulated category to abelian groups, and any thick
subcategory yields a localisation sequence for the category of
cohomological functors.  There is a notion of commutativity for pairs
of thick subcategories, which is equivalent to having Mayer--Vietoris
type exact sequences for cohomological functors.

References are \cite{Benson/Iyengar/Krause:2015a,Krause:2023a}.

\subsection*{Lattice structure}

Let $(L,\le)$ be a poset. For any subset $U\subseteq L$ we write
\[\bigvee_{x\in U}x\colonequals\sup U\] for the \emph{supremum} (or
\emph{join}) and
\[\bigwedge_{x\in U}x\colonequals\inf U\] for the \emph{infimum} (or
\emph{meet}) of the elements in $U$, provided they exist.

\begin{definition}
 A poset $L$ is a
\emph{lattice} if $x\vee y$ and $x\wedge y$ exists for any pair $x,y$
in $L$. A lattice is \emph{complete} if $\sup U$ and $\inf U$ exist for any
$U\subseteq L$. 
An element
$x\in L$ is \emph{finite} or \emph{compact} if $x\le \sup U$ implies
$x\le\sup V$ for some finite subset $V\subseteq U$.
\end{definition}

For the definition of a complete lattice it suffices to request the
existence of all suprema, since
\[\inf U=\sup\{x\in L\mid x\le u\text{ for all }u\in U\}.\]

Let $\cat T$ be a triangulated category. The thick subcategories of
$\cat T$ are partially ordered via inclusion.  For a class of objects
$\cat X\subseteq\cat T$ let $\thick(\cat X)$ denote the smallest thick
subcategory of $\cat T$ that contains $\cat X$. There is an inductive
procedure to construct the thick closure:
\[\thick(\cat X)=\bigcup_{n\ge 0}\cat X(n)\] where $\cat X(0)\colonequals\cat X$
and $\cat X(n+1)$ is obtained from  $\cat X(n)$ by adding
extensions, (de)suspensions, and direct summands of objects in  $\cat X(n)$.

We write $\Thick(\cat T)$ for the lattice of thick subcategories of
$\cat T$.  This may not be a set, but the operations $\vee$ and
$\wedge$ are well defined. In fact, for each pair of thick
subcategories $\cat U,\cat V$ one has
\[\cat U\wedge\cat V=\cat U\cap\cat V\qquad\text{and}\qquad \cat
  U\vee\cat V=\thick(\cat U\cup\cat V)\,.\] More generally, for each
family $(\cat U_i)_{i\in I}$ in $\Thick(\cat T)$
\[\bigwedge_{i\in I}\cat U_i=\bigcap_{i\in I}\cat
  U_i\qquad\text{and}\qquad \bigvee_{i\in I}\cat U_i=\thick\left(\bigcup_{i\in
    I}\cat U_i\right).\]

\begin{lemma}
  The lattice $\Thick(\cat T)$ is complete. An element
  $\cat U\in\Thick(\cat T)$ is finite if and only if
  $\cat U=\thick(x)$ for some object $x\in \cat U$.\qed
\end{lemma}

For an exact functor $f\colon\cat T\to\cat T'$ there is a pair of maps
\begin{equation*}
\begin{tikzcd}[column sep=large]
\Thick(\cat T) \ar[rr,yshift=2.5,"\Thick(f)"] &&  \ar[ll,yshift=-2.5,"\Thick^-(f)"]
\Thick(\cat T')
\end{tikzcd}
\end{equation*}
which are defined by
\[\Thick(f)(\cat U)\colonequals\thick(f(\cat U))\qquad\text{and}\qquad
  \Thick^-(f)(\cat V)\colonequals f^{-1}(\cat V)\]
for $\cat U\subseteq\cat T$ and $\cat V\subseteq\cat T'$. These maps
form an adjoint pair, because one has
\[\Thick(f)(\cat U)\subseteq\cat V\qquad\iff\qquad \cat
  U\subseteq\Thick^-(f)(\cat V).\]

We may think of a poset as category. The objects are the elements and
there is a unique morphism $x\to y$ if and only if $x\le y$. Then for
any set $U$ one has
\[\bigvee_{x\in U}x =\coprod_{x\in
    U}x\qquad\text{and}\qquad\bigwedge_{x\in U}x =\prod_{x\in U}x.\]
Thus a left adjoint preserves all joins, while a right adjoint
preserves all meets.

\subsection*{Localisation and cohomology}

Let $\cat T$ be a triangulated category with suspension
$\Si\colon\cat T\iso\cat T$. Given objects $x,y$ in $\cat T$ set
\[ \Hom_{\cat T}^*(x,y)\colonequals \bigoplus_{i\in\bbZ}\Hom_{\cat
T}(x,\Si^i y) \quad \text{and}\quad
\End_{\cat T}^{*}(x)\colonequals \Hom_{\cat T}^{*}(x,x)\,.
\] These are graded abelian groups. Composition makes $\End_{\cat
T}^{*}(x)$ a graded ring and $\Hom_{\cat T}^{*}(x,y)$ a
left-$\End_{\cat T}^{*}(y)$ right-$\End_{\cat T}^{*}(x)$ module.

Fix a triangulated subcategory $\cat U\subseteq\cat T$. The composite
of the inclusion $\cat U\to \cat T$ and the quotient functor
$\cat T\to\cat T/\cat U$ induces a long exact sequence which we
describe as follows.

For $x\in\cat T$ consider the \emph{slice category} $\cat U/x$; the
objects are pairs $(u,\phi)$ given by a morphism $\phi\colon u\to x$
with $u\in\cat U$ and a morphisms $(u,\phi)\to (u',\phi')$ is given by
a morphism $\alpha\colon u\to u'$ such that
$\phi'\alpha=\phi$. Analogously, one defines the category
$x\backslash\cat U$ of all morphisms $x\to y$ with cone in $\cat U$.

\begin{lemma}\label{le:loc-long} 
The canonical functor $\cat T\to\cat T/\cat U$ induces for each object
$x\in\cat T$ a functorial exact sequence:
\begin{multline*}
\cdots\lto \Hom^*_{\cat T/\cat U}(-,\Si^{-1}x)\lto \colim_{u\to
  x}\Hom^*_\cat T(-,u)\lto \\ \lto\Hom^*_\cat T(-,x)\lto
\Hom^*_{\cat T/\cat U}(-,x)\lto \cdots
\end{multline*}
\end{lemma}
\begin{proof}
  From the definition of $\cat T/\cat U$ we have
\[
\Hom^*_{\cat T/\cat U}(-,x)=\colim_{x\to y}\Hom^*_\cat T(-,y)
\]
where $x\to y$ runs through all morphisms with cone in $\cat U$, so
the objects in $x\backslash\cat U$.  Given such an exact triangle
$u\to x\to y\to\Si u$ in $\cat T$ yields an exact sequence:
\[
\cdots\to \Hom^*_{\cat T}(-,\Si^{-1}y)\to \Hom^*_\cat T(-,u)\to\Hom^*_\cat T(-,x)\to
\Hom^*_{\cat T}(-,y)\to \cdots
\]
Taking the colimit over the filtered category of all such
triangles yields the desired exact sequence, because there are
forgetful functors to $\cat U/x$ and $x\backslash\cat U$ which induce
isomorphisms when taking colimits.
\end{proof}

Now assume that $\cat T$ is essentially small and let $\Coh\cat T$
denote the category of cohomological functors $\cat T^\op\to\Ab$ into
the category of graded abelian groups. The assignment
\[x\longmapsto H_x\colonequals\Hom^*_\cat T(-,x)\] yields an embedding
$\cat T\to\Coh\cat T$.

\begin{lemma}
The cohomological functors $\cat T^\op\to\Ab$ are preciseley the
filtered colimits of representable functors.
\end{lemma}

\begin{proof}
  Any filtered colimit of cohomological functors is again
  cohomological. On the other hand, any additive functor $F\colon\cat T^\op\to\Ab$
  equals the colimit of representable functors indexed by the category
  $\cat T/F$ of all morphisms $H_x\to F$, so
\[\colim_{H_x\to F}\Hom^*_{\cat T}(-,x)=F.\]
The category $\cat T/F$ is filtered when $F$ is cohomological.
\end{proof}

There are three essential properties of $\Coh\cat T$ that we are going
to use:
\begin{enumerate}
\item the category  $\Coh\cat T$ has filtered colimits,
\item the category $\Coh\cat T$ admits a canonical exact structure,
  and
\item any filtered colimit of exact sequences is again exact.
\end{enumerate}
By definition a sequence $0\to F\to G\to H\to 0$ in $\Coh\cat T$ is
\emph{exact} if $0\to F(x)\to G(x)\to H(x)\to 0$ is exact in $\Ab$ for
all $x\in\cat T$.

An exact functor $f\colon\cat T\to\cat T'$ between essentially small
triangulated categories induces an adjoint pair of functors
\begin{equation*}
\begin{tikzcd}
  \Coh\cat T \ar[rr,yshift=2.5,"f^*"] && \ar[ll,yshift=-2.5,"f_*"]
  \Coh\cat T'
\end{tikzcd}
\end{equation*}
where for $F\in\Coh\cat T$ and $F'\in\Coh\cat T'$ one defines
\[f^*(F)\colonequals\colim_{H_x\to F}\Hom^*_{\cat T'}(-,f(x))\qquad\text{and}\qquad
f_*(F')\colonequals F'\circ f.\]
Using Yoneda's lemma the adjointness is easily checked:
\begin{align*}
  \Hom(f^*(F),F')
  &\cong\lim_{H_x\to F}\Hom(\Hom^*_{\cat                T'}(-,f(x)),F')\\
  &\cong\lim_{H_x\to F} F'(f(x))\\
  &\cong\lim_{H_x\to F}\Hom(\Hom^*_{\cat T}(-,x),f_*(F'))\\
  &\cong\Hom(F,f_*(F')).
\end{align*}
Note that $f^*$ and $f_*$ both are exact and preserve filtered
colimits.

For a pair of exact functors
\[\cat U\xrightarrow{\ i\ }\cat T\xrightarrow{\ q\ }\cat T/\cat U\]
given by the inclusion of a triangulated subcategory and the quotient
functor, set
\[\Gamma_\cat U\colonequals i^*\circ i_*\qquad\text{and}\qquad L_\cat
  U\colonequals q_*\circ q^*\,.\] Both functors are idempotent, so the
adjunctions induce isomorphisms $\Gamma_\cat U^2 \iso \Gamma_\cat U$
and $L_\cat U\iso L^2_\cat U$.  For $x\in\cat T$ one computes:
\begin{align*}
  \Gamma_U(\Hom^*_\cat T(-,x))&=\colim_{u\to x}\Hom^*_\cat
                               T(-,u))\\
  L_U(\Hom^*_\cat T(-,x))&=\Hom^*_{\cat
                           T/\cat U}(-,x)
\end{align*}

The equivalence $\Si\colon \cat T\iso\cat T$ extends to an equivalence
$\Coh\cat T\iso\Coh\cat T$ by taking $F$ to
$\Si F\colonequals F\circ \Si^{-1}$.

\begin{proposition}
  The sequence of functors $\cat U\to\cat T\to\cat T/\cat U$
  induces an exact sequence of functors $\Coh\cat
T\to\Coh\cat T$
\begin{equation}\label{eq:coh-loc-seq}
 \cdots \lto\Si^{n-1} L_\cat U\lto \Si^n\Gamma_\cat U\lto \Si^n\lto
  \Si^nL_\cat U\lto \Si^{n+1}\Gamma_\cat U\lto\cdots
\end{equation}
which is functorial in $\cat U$.
\end{proposition}
\begin{proof}
  The morphism $\Gamma_\cat U\to\id$ is the counit which is given by
  the construction of $\Gamma_\cat U$. Analogously, $\id\to L_\cat U$
  is the unit.  For a representable functor $F=\Hom^*_{\cat T}(-,x)$ the
  sequence identifies with the one from Lemma~\ref{le:loc-long}. It
  remains to observe that any cohomological functor is a filtered
  colimit of representable functors and that each functor in the above
  sequence preserves filtered colimits.

  An inclusion $\cat U\subseteq\cat V$ of triangulated subcategories
  induces morphisms $\Gamma_\cat U\to \Gamma_\cat V$ and
  $L_\cat U\to L_\cat V$ which yield a morphism of exact sequences of
  functors $\Coh\cat T\to\Coh\cat T$.
\end{proof}

We capture the situation in the following diagram: 
\begin{equation*}
  \begin{tikzcd}
\cat U \arrow[tail]{d}\arrow[tail,"i"]{rr} &&\cat T  \arrow[tail]{d} \arrow[twoheadrightarrow,"q"]{rr}&&\cat T/\cat U \arrow[tail]{d}\\
\Coh\cat U \arrow[tail,yshift=0.75ex]{rr}{i^*} &&\Coh\cat T  \arrow[twoheadrightarrow,yshift=-0.75ex]{ll}{i_*}
\arrow[twoheadrightarrow,yshift=0.75ex]{rr}{q^*} &&\Coh(\cat T/\cat U) \arrow[tail,yshift=-0.75ex]{ll}{q_*}
\end{tikzcd}
\end{equation*}
The lower row is a \emph{localisation sequence}, which means the following:
\begin{enumerate}
\item[(L1)] $(i^*,i_*)$ and $(q^*,q_*)$ are adjoint pairs,
\item[(L2)] $i^*$ and $q_*$ are fully faithful (equivalently,
  $\id\iso i_*i^*$ and $q^*q_*\iso\id$),
\item[(L3)] $\Im i^*=\Ker q^*$ (equivalently, $i^*i_*(x)\iso x$ if and only if $q^*(x)=0$).
\end{enumerate}
An analogue of this localisation sequence arises when $i$ and $q$ admit right
adjoints; see \eqref{eq:tria-loc-seq}. However, for
essentially small triangulated categories it is rare that the
inclusion of a thick subcategory admits an adjoint, and the
above localisation sequence is a useful substitute.

\subsection*{Commutativity}

Let $\cat T$ be a triangulated category, and we assume that $\cat T$ is essentially small. 

\begin{definition}
  A pair of thick subcategories
  $\cat U,\cat V$ of $\cat T$ is called \emph{commuting} if the inclusions
  $\cat V\leftarrow \cat U\cap\cat V\to\cat U$ induce isomorphisms
\[\Gamma_{\cat U}\Gamma_{\cat V}\leftiso\Gamma_{\cat U\cap\cat
    V}\iso \Gamma_{\cat V}\Gamma_{\cat U}\,.\] 
A thick subcategory is \emph{central} if it commutes with all other
thick subcategories.
\end{definition}

There is an equivalent and more elementary way to describe
commutativity:  A pair $\cat U,\cat V$ of thick subcategories is
commuting if and only if all morphisms $\cat U\ni u\to v\in\cat V$ and
$\cat V\ni v\to u\in\cat U$ factor through objects in
$\cat U\cap\cat V$. The following lemma provides the explanation.

\begin{lemma}\label{le:U-V-central}
  For a pair $\cat U,\cat V\in\Thick(\cat T)$ the following are equivalent.
\begin{enumerate}
\item Every morphism $\cat U\ni u\to v\in\cat V$ factors through an
  object in $\cat U\wedge\cat V$.
\item The canonical morphism
  $\Gamma_{\cat U\wedge\cat V}\to\Gamma_\cat U\Gamma_\cat V$ is an
  isomorphism.
\end{enumerate}
\end{lemma}
\begin{proof}
  When $\Coh\cat U$ is viewed as a full subcategory of $\Coh\cat T$, then the
  objects in $\Coh\cat U$ are precisely the filtered colimits of objects
  $H_u$ with $u\in\cat U$. From this it follows that $F\in\Coh\cat T$ belongs
  to $\Coh\cat U$ if and only if every morphism $H_x\to F$ with $x\in\cat T$
  factors through $H_u$ for some $u\in\cat U$.

  The above general observation shows that the two conditions are
  equivalent to the property that each object of the form
  $\Gamma_\cat U\Gamma_\cat V F$ belongs to $\Coh\cat U\wedge\cat V$.
\end{proof}

For a pair of commuting subcategories $\cat U,\cat V$ there are further canonical
isomorphisms:
\begin{equation}\label{eq:commuting}
\begin{gathered}
  L_\cat U L_\cat V\iso L_{\cat U\vee\cat V}\leftiso L_\cat V L_\cat U\\
  \Gamma_\cat V L_\cat U\leftiso \Gamma_{\cat V} L_{\cat U\wedge\cat V}\cong L_{\cat U\wedge\cat V} \Gamma_{\cat V}
  \iso L_\cat U\Gamma_\cat V\\
  \Gamma_\cat V L_\cat U\iso \Gamma_{\cat U\vee\cat V} L_\cat U\cong L_\cat U \Gamma_{\cat U\vee\cat V}
  \leftiso L_\cat U\Gamma_\cat V
\end{gathered}
\end{equation}
It is straightforward to deduce these isomorphisms from the long exact sequence
\eqref{eq:coh-loc-seq}, and details are left to the reader.

A pair of commuting subcategories gives rise to a pair of
\emph{Mayer--Vietoris sequences}. It turns out that commutativity is
actually equivalent to having such long exact sequences.

\begin{proposition}\label{pr:MV}
  Let $\cat U,\cat V\in\Thick(\cat T)$ be commuting.  Then there are canonical exact
  sequences
  \[\cdots\lto\Si^{-1}L_{\cat U\vee\cat V}\lto L_{\cat U\wedge\cat V}\lto L_\cat U\oplus
   L_\cat V\lto L_{\cat U\vee\cat V}\lto\Si
     L_{\cat U\wedge\cat V}\lto\cdots 
   \] and
 \[\cdots\lto\Si^{-1}\Gamma_{\cat U\vee\cat V}\lto\Gamma_{\cat U\wedge\cat V}\lto\Gamma_\cat U\oplus\Gamma_\cat V\lto\Gamma_{\cat U\vee\cat V}\lto\Si
     \Gamma_{\cat U\wedge\cat V}\lto\cdots
   \]
   where the morphisms are induced by the inclusions \[\cat U\leftarrow
   \cat U\wedge\cat V\to\cat V\qquad\text{and} \qquad\cat U\to
   \cat U\vee\cat V\leftarrow\cat V.\] 
 \end{proposition}

 \begin{proof}
   The long exact sequence \eqref{eq:coh-loc-seq} given by $\cat U$ and composed
   with $L_{\cat U\wedge\cat V}\to L_\cat V$
   yields  a morphisms of  exact sequences.
\[\begin{tikzcd}[column sep = small]
    \cdots\arrow{r}&\Si^{-1}L_\cat U\arrow{r}\arrow{d}&\Gamma_\cat U
    L_{\cat U\wedge\cat V}\arrow{r}\arrow[d, "\sim" labl]&
    L_{\cat U\wedge\cat V}\arrow{r}\arrow{d}& L_\cat U\arrow{r}\arrow{d}&\Si\Gamma_\cat U
    L_{\cat U\wedge\cat V}\arrow{r}\arrow[d, "\sim" labl]&\cdots\\
    \cdots\arrow{r} &\Si^{-1}L_{\cat U\wedge\cat V}\arrow{r}&\Gamma_{\cat U}L_\cat V\arrow{r}&L_\cat V\arrow{r}&
    L_{\cat U\vee\cat V}\arrow{r}&\Si\Gamma_{\cat U}L_\cat V\arrow{r}&\cdots
  \end{tikzcd}\] Here we use the identification
$L_\cat U L_\cat V= L_{\cat U\vee\cat V}$ and the isomorphisms from
\eqref{eq:commuting}. Inverting them yields the connecting morphism
$L_{\cat U\vee\cat V}\to\Si L_{\cat U\wedge\cat V}$, and then a
standard argument turns the diagram into an exact sequence of the
above form; cf.\ \cite[Lemma~10.7.4]{Dieck:2008a}. The proof for the
second sequence is analogous.
\end{proof}

\begin{corollary}
   Let $\cat U,\cat V\in\Thick(\cat T)$ be commuting.  
Then for each object $x\in\cat T$ there is a canonical exact sequence  
\begin{multline*}
  \cdots\to \Hom^*_{\cat T/(\cat U\wedge\cat V)}(-,x)\to \Hom^*_{\cat
    T/\cat U}(-,x)\oplus
  \Hom^*_{\cat T/\cat V}(-,x)\to\\
  \to\Hom^*_{\cat T/(\cat U\vee\cat V)}(-,x)\to\Hom^*_{\cat T/(\cat
    U\wedge\cat V)}(-,\Sigma x)\to\cdots\,.
\end{multline*}
\end{corollary}
\begin{proof}
Let $F=\Hom^*_\cat T(-,x)$. Then $L_\cat U F=\Hom^*_{\cat T/\cat
  U}(-,x)$. Thus the sequence follows from the Mayer--Vietoris sequence of Proposition~\ref{pr:MV}.
\end{proof}

\subsection*{Distributivity}

Let $\cat T$ be a triangulated category. We consider the lattice
$\Thick(\cat T)$ of thick subcategories and recall the following definitions.

\begin{definition}
A lattice is \emph{distributive} if for all elements
$x,y,z$ there is an equality
\[(x\wedge z)\vee(y\wedge z)=(x\vee y)\wedge z.\] An equivalent
condition is that for all elements $x,y,z$ there is an equality
\[(x\vee z)\wedge(y\vee z)=(x\wedge y)\vee z.\] A \emph{frame} is a
complete lattice where arbitrary joins distribute over finite meets.
\end{definition}

The lattice $\Thick(\cat T)$ is rarely distributive; cf.\ Example~\ref{ex:distrib}. But
distributivity is obtained when one restricts to appropriate classes
of thick subcategories. For instance, these can be tensor ideals or
subcategories given via central localisation. For establishing
distributivity a crucial ingredient is commutativity.

Our goal is to show that the central subcategories form a sublattice
of $\Thick(\cat T)$ which is distributive. We begin with a sequence of
lemmas.

\begin{lemma}\label{le:fin-dist}
  Let $\cat U,\cat V,\cat W\in\Thick(\cat T)$ and suppose $\cat W\subseteq\cat U\wedge\cat V$.
If  $\Gamma_{\cat W}\iso\Gamma_\cat U\Gamma_\cat V$ then $\cat W=\cat U\wedge\cat V$.
\end{lemma}
\begin{proof}
Let $x\in \cat U\wedge\cat V$. Then $\Gamma_\cat W
H_x\cong\Gamma_\cat U\Gamma_\cat V H_x\cong H_x$. Thus $H_x\in\Coh\cat W$ and therefore
$x\in\cat W$.
\end{proof}

The next pair of lemmas shows that meet and join preserve central
elements.

\begin{lemma}\label{le:central-res}
  Let $\cat U,\cat V\in\Thick(\cat T)$ and suppose that $\cat U$ is central.
  \begin{enumerate}
  \item The element $\cat U\wedge\cat V$ is central in $\Thick(\cat V)$. 
  \item The element $(\cat U\vee\cat V)/\cat V$ is central in $\Thick(\cat T/\cat V)$.
\end{enumerate}
\end{lemma}
\begin{proof}
  (1) Let $\cat X\in\Thick(\cat V)$. Then any morphism between objects of $\cat U\wedge\cat V$
  and $\cat X$ (in either direction) factors through an object of
  $\cat U\wedge\cat X=(\cat U\wedge\cat V)\wedge\cat X$. 

  (2) Observe that for any $\cat X\in\Thick(\cat T)$ containing
  $\cat V$ the functor $\Gamma_\cat X$ restricts to a functor
  $\Coh\cat T/\cat V\to\Coh\cat T/\cat V$ where it identifies with
  $\Gamma_{\cat X/\cat V}$, since
  $\Gamma_\cat X L_\cat V\cong L_\cat V\Gamma_\cat X$.  Using
  \eqref{eq:commuting} we obtain
  \[\Gamma_\cat X\Gamma_{\cat U\vee\cat V} L_\cat V\cong \Gamma_\cat X \Gamma_{\cat U} L_\cat V\cong
    \Gamma_\cat U\Gamma_\cat X L_\cat V\cong \Gamma_\cat U L_\cat V \Gamma_\cat X \cong \Gamma_{\cat U\vee\cat V}
    L_\cat V \Gamma_\cat X \cong \Gamma_{\cat U\vee\cat V} \Gamma_\cat X L_\cat V,\] and therefore
  $\Gamma_\cat X \Gamma_{\cat U\vee\cat V}\cong \Gamma_{\cat U\vee\cat V}\Gamma_\cat X$ when restricted to
  $\Coh\cat T/\cat V$.
\end{proof}  

\begin{lemma}\label{le:finite-sup}
  Let $\cat U,\cat V\in\Thick(\cat T)$ be central. Then
  $\cat U\vee\cat V$ and   $\cat U\wedge\cat V$ are central. Moreover, we have  for all $\cat W\in\Thick(\cat T)$
  \[(\cat U\wedge \cat W)\vee(\cat V\wedge\cat W)=(\cat U\vee\cat V)\wedge\cat W\] 
and
  \[(\cat U\vee \cat W)\wedge(\cat V\vee\cat W)=(\cat U\wedge\cat V)\vee\cat W\]
\end{lemma}
\begin{proof}
It is clear that $\cat U\wedge\cat V$ is central since
\[\Gamma_{\cat U\wedge\cat V}\Gamma_\cat W\cong \Gamma_{\cat U}\Gamma_{\cat V}\Gamma_\cat W\cong \Gamma_\cat W\Gamma_{\cat U}\Gamma_\cat V\cong
  \Gamma_\cat W\Gamma_{\cat U\wedge\cat V}.\] 
Now we show that $\cat U\vee\cat V$ is central and combine this with the proof
  of the first identity.  Consider the following Mayer--Vietoris
  sequences, where the middle one uses that $\cat U\wedge\cat W$ and
  $\cat V\wedge\cat W$ are commuting by Lemma~\ref{le:central-res}.
  \[\begin{tikzcd}
      \cdots\arrow{r}&\Gamma_{\cat U\wedge\cat V}\Gamma_\cat
      W\arrow{r}&\Gamma_\cat U\Gamma_\cat W\oplus\Gamma_\cat V\Gamma_\cat W\arrow{r}&
      \Gamma_{\cat U\vee\cat V}\Gamma_\cat W\arrow{r}&\cdots\\
      \cdots\arrow{r}&\Gamma_{\cat U\wedge\cat V\wedge\cat W}\arrow{r}\arrow[u, "\sim" labl]\arrow[d, "\sim" labl]&
      \Gamma_{\cat U\wedge\cat W}\oplus\Gamma_{\cat V\wedge\cat W}\arrow{r}\arrow[u, "\sim" labl]\arrow[d, "\sim" labl]&
      \Gamma_{(\cat U\wedge
        \cat W)\vee(\cat V\wedge\cat W)}\arrow{r}\arrow{u}\arrow{d}&\cdots\\
      \cdots\arrow{r}&\Gamma_\cat W \Gamma_{\cat U\wedge\cat
        V}\arrow{r}&\Gamma_\cat W \Gamma_\cat U\oplus\Gamma_\cat
      W\Gamma_\cat V\arrow{r}&
      \Gamma_\cat W \Gamma_{\cat U\vee\cat V}\arrow{r}&\cdots
    \end{tikzcd}\] The vertical isomorphisms are clear since $\cat U$ and
  $\cat V$ central. The inclusion
  \[(\cat U\wedge \cat W)\vee(\cat V\wedge\cat W)\subseteq (\cat
    U\vee\cat V)\wedge\cat W\] is automatic. Thus we can apply
  Lemma~\ref{le:fin-dist}, and then the five lemma yields the
  assertion for $\cat U\vee\cat V$. The proof of the second identity
  is similar, using the other Mayer--Vietoris sequence involving functors of the form $L_\cat U$ instead of
  $\Gamma_\cat U$.
\end{proof}

An immediate consequence of Lemma~\ref{le:finite-sup} is the fact that
the central elements form a sublattice of $\Thick(\cat T)$ which is
distributive. The next step is to extend the above lemma and to
establish the same assertions for arbitrary joins, using some standard
arguments.  Thus the central subcategories form a frame.

\begin{proposition}\label{pr:centre}
  The central subcategories of $\cat T$ are closed under finite meets
  and arbitrary joins.  Let $(\cat U_i)_{i\in I}$ be a family of
  central elements in $\Thick(\cat T)$ and $\cat V\in\Thick(\cat
  T)$. Then the following holds:
 \begin{align*}
   \bigvee_{i\in I}(\cat U_i\wedge \cat V)&=\left(\bigvee_{i\in
   I}\cat U_i\right)\wedge\cat V\qquad (I\text{ arbitrary})\\
   \bigwedge_{i\in I}(\cat U_i\vee \cat V)&=\left(\bigwedge_{i\in
   I}\cat U_i\right)\vee\cat V\qquad (I\text{ finite})
 \end{align*}
\end{proposition}
\begin{proof}
  Set $\cat U=\bigvee_{i\in I}\cat U_i$. The element
  $\cat U_J\colonequals\bigvee_{i\in J}\cat U_i$ is central for all finite $J\subseteq I$
  by Lemma~\ref{le:finite-sup}.  Now choose objects $u\in\cat U$ and
  $v\in\cat V$. Then $u\in\cat U_J$ for some finite $J\subseteq I$ and
  therefore any morphism between $u$ and $v$ factors through an object
  in $\cat U_J\wedge \cat V\subseteq \cat U\wedge \cat V$. Thus $\cat U$ is central.

  The inclusion
  $\bigvee_{i\in I}(\cat U_i\wedge \cat V)\subseteq\cat U\wedge\cat V$
  is clear. For the other inclusion we use again
  Lemma~\ref{le:finite-sup}. In fact any object
  $x\in\cat U\wedge \cat V$ belongs to
  $\cat U_J\wedge V=\bigvee_{i\in J}(\cat U_i\wedge \cat V)$ for some
  finite $J\subseteq I$, and therefore $x$ belongs to
  $\bigvee_{i\in I}(\cat U_i\wedge \cat V)$.

  The second equality is already covered by Lemma~\ref{le:finite-sup}.
\end{proof}

In specific examples we may restrict to a sublattice
$T\subseteq \Thick(\cat T)$ which is closed under arbitrary joins, and then
look only at elements which commute with all other elements from $T$.
For instance, $T$ could be the lattice of thick tensor ideals when
$\cat T$ is a tensor triangulated category; cf.\
Proposition~\ref{pr:tensor-central}. The above Proposition~\ref{pr:centre}
generalises to this situation; so the central elements in $T$ form a
frame.

\begin{example}
  For the category of perfect complexes over a commutative ring all
  thick subcategories are central. One way of seeing this is that the
  category of perfect complexes is tensor triangulated. Then one uses
  that each thick subcategory is a tensor ideal, because the tensor
  unit is generating, and that tensor ideals commute with each other
  by Proposition~\ref{pr:tensor-central}.
\end{example}

\begin{example}\label{ex:distrib}
  Consider the path algebra $A=k(\circ\to\circ)$ over a field $k$ and
  the bounded derived category $\cat T=\mathbf D^{\mathrm b}(\mod A)$
  of finitely generated $A$-modules, which identifies with the
  category of perfect complexes over $A$.  There are three
  isoclasses of indecomposable $A$-modules and each generates a proper thick
  subcategory. Thus the Hasse diagram of the
  lattice of thick subcategories is the following:
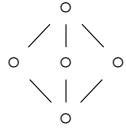
\[\begin{tikzcd}[column sep = tiny, row sep = small]
    &\circ\arrow[dash]{d}\arrow[dash]{ld}\arrow[dash]{rd}\\
    \circ \arrow[dash]{rd}&\circ\arrow[dash]{d}&\circ \arrow[dash]{ld}\\
    &\circ
  \end{tikzcd}\]
In particular, the lattice is not distributive. Central subcategories are $\{0\}$ and $\cat T$.
\end{example}

\section{Support via central actions}

We consider the action of a graded-commutative ring on a triangulated
category. This provides a notion of \emph{cohomological support} for
objects, and we collect its basic properties. A useful tool are Koszul
objects. For a tensor triangulated category there is a natural notion of
\emph{tensor triangulated support} which is based on its spectrum of
prime ideals. For any ring action on a tensor triangulated category
there is a canonical comparison map which relates both notions of
support.

References are \cite{Balmer:2005a, Benson/Iyengar/Krause:2008a, Hochster:1969a}.

\subsection*{Support for modules}

Let $R$ be a \emph{graded-commutative ring}; thus $R$ is $\bbZ$-graded
and satisfies \[rs=(-1)^{|r||s|}sr\quad \text{for}\quad r,s\in R,\]
where $|r|$ denotes the degree of a homogeneous element $r$.
We consider only homogeneous elements or ideals
of $R$, and $\Spec R$ denotes the set of homogeneous prime ideals of
$R$ with the Zariski topology. Analogously, only graded $R$-modules
are considered. We write $\Mod R$ for the category of graded
$R$-modules with morphisms given by degree zero $R$-linear maps.

For $R$-modules $M,N$ we set
\[ \Hom_{R}^*(M,N)\colonequals \bigoplus_{i\in\bbZ}\Hom_{R}(M,N[i]) \quad \text{and}\quad
  \End_{R}^{*}(M)\colonequals \Hom_{R}^{*}(M,M)\,.\] Scalar
multiplication yields a  homomorphism
$\phi_M\colon R\to \End^*_R(M)$ such that the induced left and right
actions of $R$ on $\Hom_R^{*}(M,N)$ are compatible: For any $r\in R$
and $\alpha\in\Hom^*_R(M,N)$, one has
\[ \phi_N(r)\alpha=(-1)^{|r||\alpha|}\alpha\phi_M(r)\,.
\]

For a multiplicative set $S\subseteq R$ of homogeneous elements we
write $R[S^{-1}]$ for the localisation and set
$M[S^{-1}]\colonequals M\otimes_R R[S^{-1}]$ for any $R$-module
$M$. This yields an exact functor $\Mod R\to \Mod R[S^{-1}]$, and an
$R$-module is \emph{$S$-local} if the canonical map $M\to M[S^{-1}]$
is an isomorphism.

For $\fp\in\Spec R$ set $R_\fp\colonequals R[(R\smallsetminus \fp)^{-1}]$
and $M_\fp\colonequals M\otimes_R R_\fp$ for any $R$-module $M$. An
$R$-module $M$ is called \emph{$\fp$-local} if the canonical map
$M\to M_\fp$ is an isomorphism.

\begin{definition}
  The \emph{support} of an $R$-module $M$ is the
  subset \[\supp_R(M)\colonequals\{\fp\in\Spec R\mid M_\fp\neq 0\}.\]
\end{definition}

The support of a module is a subset $V\subseteq\Spec R$ that is
\emph{specialisation closed}, so for primes $\fp\subseteq\fq$ we have
that $\fp\in V$ implies $\fq\in V$. Further basic properties of
$\supp_R$ are derived from the fact that $(-)_\fp$ is exact and
preserves all coproducts.

For a multiplicative set $S\subseteq R$ of homogeneous elements the
canonical map $R\to R[S^{-1}]$ identifies $\Spec R[S^{-1}]$ with the
subset $\{\fq\in\Spec R\mid \fp\cap S=\varnothing\}$.  In particular,
for $\fp\in\Spec R$ the canonical map $R\to R_\fp$ identifies
$\Spec R_\fp$ with the subset $\{\fq\in\Spec R\mid
\fq\subseteq\fp\}$. This has the following consequence.

\begin{lemma}\pushQED{\qed} 
  For an $R$-module $M$, we have
  \[M[S^{-1}]=0\quad\iff \quad\supp_R(M)\subseteq\{\fp\in\Spec R\mid \fp\cap S\neq\varnothing\}.\qedhere\]
\end{lemma}

For an ideal $\fa$ of $R$ set
\[V(\fa)\colonequals\{\fp\in\Spec R\mid \fa\subseteq\fp\}.\] If $\fa$ is
generated by a single element $r$ we write $V(r)$ for $V(\fa)$. Let
$\ann_R(M)$ denote the kernel of the canonical map $R\to \End^*_R(M)$.

\begin{lemma}
  We have $\supp_R(M)\subseteq V(\ann_R(M))$, and equality holds when
  $M$ is finitely generated.\qed
\end{lemma}

An $R$-module $M$ is called \emph{$\fa$-torsion} if for each
homogeneous $x\in M$ we have $\fa^nx=0$ for $n\gg 0$. An equivalent
condition is $\supp_R(N)\subseteq V(\fa)$ for each finitely generated
submodule $N\subseteq M$, which means that $\supp_R(M)\subseteq
V(\fa)$.

Let $\alpha\colon R\to S$ be a homomorphism of graded-commutative
rings. The assignment $\fp\mapsto\alpha^{-1}(\fp)$ yields  a map
\[ \alpha^*\colon\Spec S\lto\Spec R\]
which is also denoted by $\Spec\alpha$.

\begin{lemma}\label{le:supp-ring-hom}
<  We have $\supp_R(S)=V(\Ker\alpha)$.
\end{lemma}
\begin{proof}
This follows from the above lemma. For one inclusion observe that
  $\Ker\alpha=\ann_R(S)$. For the other inclusion observe that
  $R/(\Ker\alpha)$ identifies with an $R$-submodule of $S$.
\end{proof}

\subsection*{The graded centre}

For a triangulated category $\cat T$ with suspension $\Si\colon\cat
T\iso\cat T$ we introduce its graded centre and actions of graded-commutative rings.

\begin{definition}
  Let $R$ be a graded-commutative ring. We say that the triangulated category $\cat T$ is
\emph{$R$-linear}, or that $R$ \emph{acts} on $\cat T$, if for each
$x$ in $\cat T$ there is a homomorphism of graded rings
\[ \phi_x\colon R\lto \End_{\cat T}^{*}(x)
\] such that the induced left and right actions of $R$ on $\Hom_{\cat
T}^{*}(x,y)$ are compatible: For any $r\in R$ and
$\alpha\in\Hom^*_{\cat T}(x,y)$, one has
\[ \phi_y(r)\alpha=(-1)^{|r||\alpha|}\alpha\phi_x(r)\,.
\]
This means that the representable functors $\Hom^*_{\cat T}(x,-)$ and
$\Hom^*_{\cat T}(-,y)$ are $R$-linear cohomological functors $\cat
T\to \Mod R$.
\end{definition}

There is an alternative way to describe a central ring action on
$\cat T$ via its centre.

\begin{definition}
  The \emph{graded centre} of $\cat T$ is the  graded-commutative ring 
\[
Z^*(\cat T)=\bigoplus_{n\in\bbZ}\{\eta\colon\id_{\cat T}\to\Si^n\mid\eta\Si=(-1)^n\Si\eta\}
\]
with multiplication given by
composition of morphisms in $\cat T$. Note that $Z^*(\cat T)$ may not
be a set.
\end{definition}

If the ring $R$ acts on $\cat T$ via homomorphisms
$\phi_x\colon R\to \End_{\cat T}^{*}(x)$, then this yields a
homomorphism $\phi\colon R\to Z^*(\cat T)$ by setting
$\phi(r)_x\colonequals\phi_x(r)$ for $x\in\cat T$ and $r\in
R$. Conversely, any homomorphism $R\to Z^*(\cat T)$ gives an
$R$-action on $\cat T$.

For example, one has a homomorphism $\mathbb Z\to Z^*(\cat T)$ which sends $n\in\bbZ$
to $n\cdot \id \colon \id_{\cat T}\to\Si^0$. Thus any triangulated
category is $\mathbb Z$-linear.

\subsection*{Central localisation}

Let $R$ be a graded-commutative ring and $\cat T$ an $R$-linear
triangulated category. For a multiplicative set $S\subseteq R$ of
homogeneous elements we define a triangulated
category $\cat T[S^{-1}]$ as follows. The objects are the same as in
$\cat T$. For the morphisms set
\[\Hom^*_{\cat T[S^{-1}]}(x,y)\colonequals\Hom^*_{\cat T}(x,y)[S^{-1}].\]
The exact triangles in $\cat T[S^{-1}]$ are
given by the ones isomorphic to images of exact triangles in $\cat T$
under the canonical functor $\cat T\to \cat T[S^{-1}]$. The kernel of this
functor is the thick subcategory
\[\cat T_S\colonequals\{x\in\cat T\mid \End_{\cat T}^*(x)[S^{-1}]=0\}\]
and it is not difficult to check that   $\cat T\to \cat T[S^{-1}]$
induces a triangle equivalence
\begin{equation*}\label{eq:mult-closed}
  \cat T/\cat T_S\longiso \cat T[S^{-1}].
\end{equation*}
The notation $x\mapsto x[S^{-1}]$ is used for the canonical functor $\cat
T\to\cat T[S^{-1}]$.

\begin{lemma}\label{le:central-loc}
An $R$-linear cohomological functor $H\colon\cat T\to\Mod R$ factors through the
canonical functor $\cat T\to \cat T[S^{-1}]$ if and only if $H(x)\iso
H(x)[S^{-1}]$ for all $x\in\cat T$.
\end{lemma}
\begin{proof}
  For any pair $x,y\in\cat T$ we have the following  commutative square:
  \[  \begin{tikzcd}
      \Hom^*_\cat T(x,y)\ar[r]\ar[d]& \Hom^*_\cat T(x,y)[S^{-1}]\ar[d]\\
\Hom^*_R(H(x),H(y))\ar[r]& \Hom^*_R(H(x),H(y))[S^{-1}]  
\end{tikzcd}\]
The assertion follows from the fact that  the map $\Hom^*_\cat T(x,y)\to
\Hom^*_R(H(x),H(y))$ induced by $H$ factors through 
the top horizontal map if and only if the bottom horizontal map is invertible.
\end{proof}
  
A special case of the above construction arises for each $\fp\in\Spec
R$. Let $\cat T_\fp$ denote the triangulated category given by
localising the morphisms at $\fp$,
so \[\Hom^*_{\cat T_\fp}(x,y)\colonequals\Hom^*_{\cat T}(x,y)_\fp.\]
The notation $x\mapsto x_\fp$ is used for the canonical functor $\cat
T\to\cat T_\fp$.

The next lemma provides a supply of central subcategories.

\begin{lemma}\label{le:central}
Let $R$ be a graded-commutative ring 
  acting on $\cat T$ and $S\subseteq R$ a multiplicative subset. Then
  $\cat T_S$ is a central thick subcategory of $\cat T$.
\end{lemma}
\begin{proof}
  To simplify notation we set $\Gamma_S\colonequals\Gamma_{\cat T_S}$ and
  $L_S\colonequals L_{\cat T_S}$. Let $F\in\Coh\cat T$. Then $F(x)$ is an
  $R$-module for all $x\in\cat T$ and we have an isomorphism
  \[(L_S F)\cong F[S^{-1}].\] This holds for a
  representable functor $F=\Hom^*_\cat T(-,x)$ by the definition of $\cat
  T[S^{-1}]$, and the general case follows since $L_S$ preserves
  filtered colimits.

  Now fix a thick subcategory
  $\cat U\subseteq\cat T$. We have canonical isomorphisms
  \[L_S \Gamma_\cat U\leftiso \Gamma_\cat U L_S \Gamma_\cat
    U\iso\Gamma_\cat U L_S\] since $L_S$ restricts to a functor
  $\Coh\cat U\to\Coh\cat U$.  Consider the following commutative
  diagram, where the rows are induced by the long exact sequence
  \eqref{eq:coh-loc-seq}.
  \[\begin{tikzcd}
      \cdots\arrow{r}&\Gamma_S\Gamma_\cat U\arrow{r}&\Gamma_\cat U\arrow{r}&
      L_S\Gamma_\cat U\arrow{r}&\cdots\\
      \cdots\arrow{r}&\Gamma_\cat U\Gamma_S\Gamma_\cat U\arrow{r}\arrow{u}\arrow{d}&
      \Gamma_\cat U\Gamma_\cat U\arrow{r}\arrow[u, "\sim" labl]\arrow[d, "\sim" labl]&
      \Gamma_\cat U L_S\Gamma_\cat U\arrow{r}\arrow[u, "\sim" labl]\arrow[d, "\sim" labl]&\cdots\\
      \cdots\arrow{r}&\Gamma_\cat U \Gamma_S\arrow{r}&\Gamma_\cat U\arrow{r}&
      \Gamma_\cat U L_S\arrow{r}&\cdots\\
    \end{tikzcd}\]
The five lemma yields the
  isomorphism $\Gamma_S\Gamma_\cat U\cong \Gamma_\cat U\Gamma_S$, and therefore
  $\cat T_S$ is central.
\end{proof}

\subsection*{Koszul objects}

Let $R$ be a graded-commutative ring and $\cat T$ an $R$-linear
triangulated category.

\begin{definition}
  Let $r\in R$ be homogeneous of degree $d$ and $x$ an object in
  $\cat T $.  A \emph{Koszul object of $r$ on $x$} is any object
  $\kos xr$ that appears in an exact triangle
\[ x\stackrel{r}\lto \Si^{d}x\lto \kos xr \lto \Si x\,.
\] 
For a sequence of elements $\mathbf r=r_1,\ldots,r_n$ in $R$, consider
objects $x_i$ defined by
\begin{equation*}
\label{eq:koszul} x_i = \begin{cases} x & \text{for $i=0$,}\\
\kos{x_{i-1}}{r_i} & \text{for $i\geq 1$.}
\end{cases}
\end{equation*} Set $\kos x{\mathbf r} = x_n$; this is a \emph{Koszul object
  of $\mathbf r$ on $x$}.
A Koszul object is well defined up to (nonunique) isomorphism.
\end{definition}

\begin{lemma}\label{le:Koszul-r}
  Let $r\in R$  and set $S=\{r^i\mid i\ge 0\}$.  If
  $\cat T=\thick(\cat X)$ for some class $\cat X\subseteq \cat T$,
  then we have \[\cat T_S=\thick(\{\kos xr\mid x\in\cat X\}).\]
\end{lemma}
\begin{proof}
  We use the fact that a thick subcategory $\cat U\subseteq\cat T$ is
  determined by the class of $R$-linear cohomological functors
  $H\colon \cat T\to\Mod R$ such that $H(\cat U)=0$. For $x\in\cat T$
  let $r_x$ denote the morphism $x\to\Si^{d} x$ given by
  multiplication with $r$, where $d=|r|$. Given a cohomological functor
  $H\colon \cat T\to\Mod R$, the exact triangle defining $\kos xr$
  induces an exact sequence
  \[\cdots \lto H(x)\stackrel{r}\lto H(x)[d]\lto H(\kos xr)\lto H(x)[1]\stackrel{r}\lto\cdots\]
  and therefore $H(\kos xr)=0$ if and only if $H(r_x)$ is
  invertible. Now $H(r_x)$ is invertible for all $x\in\cat X$ if and
  only if $H(x)\iso H(x)[S^{-1}]$ for all $x\in\cat T$. From
  Lemma~\ref{le:central-loc} it follows that the latter condition is
  equivalent to $H$ factoring through the canonical functor
  $\cat T\to \cat T[S^{-1}]$, which means that $H$ annihilates
  $\cat T_S$.
\end{proof}

For a finitely generated ideal $\fa$ of $R$ set
\[\cat T_{V(\fa)}\colonequals\{x\in\cat T\mid x_\fp=0 \text{ for } \fp\in (\Spec
  R)\smallsetminus V(\fa)\}.\] Choose a sequence of homogeneous generators
$\mathbf r =r_1,\ldots,r_n$ of $\fa$ and set $\kos x{\fa}\colonequals\kos
x{\mathbf r}$. This Koszul object depends on the choice of generators but
the thick subcategory generated by it does not.

\begin{proposition}\label{pr:koszul}
  Let $\cat T=\thick(\cat X)$ for some class $\cat X\subseteq \cat T$.
  Then we have
  \[\cat T_{V(\fa)}=\thick(\{\kos x{\fa}\mid x\in\cat X\}).\]
  Moreover, $\cat T_{V(\fa)}$ is a central thick subcategory.
\end{proposition}
\begin{proof}
  For $r\in R$ and $S=\{r^i\mid i\ge 0\}$  Lemma~\ref{le:Koszul-r} implies
  \[\cat T_{V(r)}=\cat T_S=\thick(\{\kos xr\mid x\in\cat X\})\,.\]
  Now choose a sequence of homogeneous generators $r_1,\ldots,r_n$ of
  $\fa$, so
  \[V(\fa)=V(r_1)\cap\cdots\cap V(r_n).\] Let $\fb$ denote the ideal
  generated by $r_1,\ldots,r_{n-1}$. Then
  $\cat T_{V(\fa)}=(\cat T_{V(\fb)})_{V(r_n)}$ and an induction on $n$
  yields the assertion, since $\kos x{\fa}=\kos{(\kos x{\fb})}{r_n}$.

  Finally, we have
  \[\cat T_{V(\fa)}=\cat T_{V(r_1)}\cap\cdots\cap \cat T_{V(r_n)}\]
  and therefore $\cat T_{V(\fa)}$ is central by
  Proposition~\ref{pr:centre} and Lemma~\ref{le:central}.
\end{proof}

\subsection*{Cohomological support}

Let $R$ be a graded-commutative ring and $\cat T$ an $R$-linear
triangulated category.  The subsets $V(\fa)\subseteq\Spec R$ given by finitely generated
ideals form a distributive lattice with
\[V(\fa)\cap V(\fb)=V(\fa+\fb)\qquad\text{and}\qquad V(\fa)\cup V(\fb)=V(\fa\fb).\]

\begin{lemma}\label{le:union}
  For $U=V(\fa)$ and $V=V(\fb)$ we have
  \[\cat T_{U\cap V}=\cat T_U\wedge\cat T_V\qquad\text{and}\qquad
  \cat T_{U\cup V}=\cat T_U\vee\cat T_V.\]
\end{lemma}
\begin{proof}
  The first equality is clear. For the second equality choose
  homogeneous generators
  $r_1,\ldots,r_m$ of $\fa$ and  $s_1,\ldots,s_n$ of $\fb$ such that
  \[V(\fa)=V(r_1)\cap\cdots\cap V(r_m)\qquad\text{and}\qquad\text
    V(\fb)=V(s_1)\cap\cdots\cap V(s_n).\] For a homogeneous element
  $r\in R$ one has
  \[\cat T_{V(r)}=\thick(\{\kos{x}{r}\mid x\in\cat T\})\]
  by Lemma~\ref{le:Koszul-r}. For homogeneous elements $r,s\in R$ an
  application of the octahedral axiom shows that $\kos{x}{rs}$ is an
  extension of $\kos{x}{r}$ and $\kos{\Si^{|r|}x}{s}$. Thus
  \[\cat T_{V(r)\cup V(s)}=\cat T_{V(rs)}=\cat T_{V(r)}\vee\cat
    T_{V(s)}.\] We may assume that $\cat T$ is essentially small
  because for $\cat T'\subseteq\cat T$  we have $ \cat T'_{V(\fa)}=
  \cat T_{V(\fa)}\cap\cat T'$. Thus applying the distributivity discussed in
  Proposition~\ref{pr:centre} together with Lemma~\ref{le:central} we
  obtain
\[    \cat T_{V(\fa)}\vee\cat T_{V(\fb)}=\bigwedge_{i,j}( \cat T_{V(r_i)}\vee  \cat T_{V(s_j)})
=\bigwedge_{i,j} \cat T_{V(r_i s_j)}=\cat T_{V(\fa\fb)}=\cat T_{V(\fa)\cup
  V(\fb)}\]
since
\[V(\fa\fb)=\bigcap_{i,j} V(r_i s_j).\qedhere\]
\end{proof}

\begin{definition}
  A subset $V\subseteq\Spec R$ is said to be \emph{Thomason} if it can
  be written as $V=\bigcup_i V_i$ such that each
  $(\Spec R)\smallsetminus V_i$ is Zariski open and quasi-compact. When $R$
  is noetherian, Thomason subsets are precisely the specialisation
  closed subsets of $\Spec R$.  In general, the Thomason subsets are
  by definition the open subsets for the \emph{Hochster dual topology}
  on $\Spec R$; see \cite{Hochster:1969a}.
\end{definition}

\begin{definition}
  The open subsets of any space form a \emph{frame}, that is to say, a
complete lattice where arbitrary joins distribute over finite meets.
A morphism of frames is a map that preserves arbitrary joins and
finite meets.  Given a frame $F$, one calls $x\in F$ \emph{finite} or
\emph{compact} if $x\le\bigvee_{i\in I}y_i$ implies
$x\le\bigvee_{i\in J}y_i$ for some finite subset $J\subseteq I$. A
frame is \emph{coherent} when every element can be written as a join
of finite elements and the finite elements form a sublattice.
\end{definition}

\begin{example}
  The Thomason sets of $\Spec R$ form a coherent frame and the finite
  elements are precisely the sets of the form $V(\fa)$ given by a
  finitely generated ideal of $R$.
\end{example}

We consider the lattice of thick subcategories in
$\cat T$.  For a Thomason subset $V\subseteq \Spec R$ we set
  \[
  \tau_R(V)\colonequals \bigvee_{V(\fa)\subseteq V} \cat T_{V(\fa)},
\]
where $\fa$ runs through all finitely generated ideals of $R$.

\begin{proposition}\label{pr:Thomason}
 Let $\cat T$ be an essentially small triangulated category and $R$ a ring acting on
  it.   Then the assignment $V\mapsto \tau_R(V)$  induces a lattice morphism
  \begin{equation*}
  \label{eq:tau}
  \{\text{Thomason subsets of }\Spec
  R\}\lto\Thick(\cat T)\,.
\end{equation*}  
which preserves arbitrary joins. Moreover, each $\tau_R(V)$ is a
central thick subcategory.
\end{proposition}
\begin{proof}
We use the fact that the map $V(\fa)\mapsto\cat T_{V(\fa)}$ is
a lattice morphism by Lemma~\ref{le:union}. Let $V=\bigcup_i V_i$ be a
join of Thomason sets. We need to show that
\[\tau_R\left(\bigcup_i V_i\right)\subseteq \bigvee_i \tau_R(V_i);\]
the other inclusion is automatic. Let $V(\fa)\subseteq V$. Then there is a finite set of finitely
generated ideals $\fa_i$ such that $V(\fa_i)\subseteq V_i$ and
$V(\fa)\subseteq \bigcup_i V(\fa_i)$. Thus
\[\cat T_{V(\fa)}\subseteq \tau_R\left(\bigcup_i V(\fa_i)\right) =
  \bigvee_i \tau_R(V(\fa_i))\subseteq \bigvee_i \tau_R(V_i).\]
Now choose Thomason sets $V_1,V_2$. We need to show that
\[\tau_R(V_1)\wedge\tau_R(V_2)\subseteq\tau_R(V_1\cap V_2);\]
the other inclusion is automatic. For any object $x\in \tau_R(V_1)\wedge\tau_R(V_2)$ we find finitely
generated ideals $\fa_i$ such that  $V(\fa_i)\subseteq  V_i$ and
\[x\in\cat T_{V(\fa_1)}\wedge\cat T_{V(\fa_2)}=\cat
  T_{V(\fa_1+\fa_2)}\subseteq \tau_R(V_1\cap V_2),\]
since $V(\fa_1+\fa_2)=V(\fa_1)\cap V(\fa_2)\subseteq V_1\cap V_2$.

The final assertion follows from  Propositions~\ref{pr:centre} and \ref{pr:koszul}.
\end{proof}

In view of the next definition it is useful to recall for $x\in\cat T$
and $\fp\in\Spec R$
\begin{align*}
  x_\fp\neq 0&\iff\Hom^*_\cat T(-,x)_\fp\neq 0\\
             &\iff\End^*_\cat T(x)_\fp\neq 0\\
             &\iff\fp\in V(\Ker\phi_x)
\end{align*}
where the last equivalence follows Lemma~\ref{le:supp-ring-hom}.

\begin{definition}
  For an object $x\in \cat T$ the \emph{cohomological support} is
  \[ \supp_R(x) \colonequals \{\fp\in\Spec R\mid x_\fp\neq 0 \}.\] For
  a class of objects $\cat X\subseteq\cat T$ set
\[\supp_R(\cat X) \colonequals\bigcup_{x\in\cat X}\supp_R(x).\]
We say that $\supp_R$ is \emph{finite} if for each $x\in\cat T$ there
is a finitely generated ideal $\fa$ of $R$ such that
$\supp_R(x)=V(\fa)$.
\end{definition}

The cohomological support is a specialisation closed subset of
$\Spec R$, but not necessarily a Thomason set. Also, $\supp_R$
preserves arbitrary joins in $\Thick(\cat T)$, but not necessarily
finite meets. The following lemma explains the finiteness condition.

\begin{lemma}\label{le:adjoint-supp}
  The following conditions are equivalent.
  \begin{enumerate}
  \item The cohomological support $\supp_R$ is finite.
\item The maps $\supp_R$ and $\tau_R$ provide an adjoint pair
\begin{equation*}\label{eq:supp-tau}
\begin{tikzcd}
\Thick(\cat T) \ar[rr,yshift=2.5,"\supp_R"] &&  \ar[ll,yshift=-2.5,"\tau_R"]
\{\text{Thomason subsets of }\Spec R\}.
\end{tikzcd}
\end{equation*}
\end{enumerate}
 In this case one has for each Thomason set
$V\subseteq\Spec R$
\[
\tau_R(V)=\{x\in\cat T\mid \supp_R(x)\subseteq V\}.
\]
\end{lemma}
\begin{proof}
  Suppose first that $\supp_R$ is finite.  Let $V$ be Thomason and
  $x\in \cat T$. Then we have
\begin{align*}
\supp_R(x)\subseteq V&\iff \supp_R(x)\subseteq V(\fa) \text{ for some }V(\fa)\subseteq V\\
                     &\iff x\in\cat T_{V(\fa)}\text{ for some }V(\fa)\subseteq V\\
                     &\iff  x\in\tau_R(V).
\end{align*}
The first equivalence uses the assumption on $\supp_R(x)$ and the
other ones are immediate from the definitions.

Now suppose there is an adjoint pair of maps. This implies that
$\supp_R(x)$ is Thomason for each   $x\in \cat T$. A standard argument
shows that the left adjoint preserves finiteness if the right
adjoint  preserves all joins. It remains to note that finiteness in
$\Spec R$ means that there
is a finitely generated ideal $\fa$ of $R$ such that 
$\supp_R(x)=V(\fa)$.
\end{proof}

\begin{example}
  Let $R$ be noetherian ring. Then for any $R$-linear triangulated
  category the cohomological support $\supp_R$ is finite.
\end{example}

\begin{lemma}\label{le:koszul-supp}
  For a finitely generated ideal $\fa$ of $R$ and $x\in\cat T$ we have
  \[\supp_R(\kos x{\fa})=\supp_R(x)\cap V(\fa).\]
\end{lemma}
\begin{proof}
The inclusion $\subseteq$ follows from
Proposition~\ref{pr:koszul}. For the other inclusion we use an
induction on the number of generators of $\fa$, as in the proof of
Proposition~\ref{pr:koszul}. Thus we claim for $r\in R$ that
$\supp_R(\kos xr) \supseteq\supp_R(x)\cap V(r)$. Choose $\fp$ in
$\supp_R(x)\cap V(r)$. If $(\kos xr)_\fp=0$, then the canonical map
$R\to\End^*(x)_\fp$ sends $r$ to a unit, which is impossible since $r\in\fp$. 
\end{proof}

\begin{proposition}
  \label{pr:hopkins}
  Let $\cat T$ be an essentially small triangulated category and $R$ a ring acting on
  it.  Then cohomological support induces a lattice isomorphism
  \begin{equation*}
     \Thick(\cat T)\longiso\{\text{Thomason subsets of }\Spec
     R\}
\end{equation*}  
if and only if the following holds:
  \begin{enumerate}
  \item  $\supp_R(\cat T)=\Spec R$;  
  \item $\supp_R$ is finite;
  \item $\supp_R(x)\subseteq\supp_R(y)$ implies
    $\thick(x)\subseteq\thick(y)$ for all $x,y\in\cat T$.
\end{enumerate}
\end{proposition}

\begin{proof}
  The forward direction is clear. For the other direction we claim
  \[\text{(1)--(2)}\implies {\supp_R}\circ \tau_R=\id\qquad\text{and}\qquad
    \text{(2)--(3)}\implies \tau_R\circ{\supp_R}=\id.  \]

  First suppose that (1)--(2) hold. Then
\[\Spec R=\supp_R(\cat T)=\bigcup_{x\in\cat T}\supp_R(x)\] implies 
that there  is an object
  $x_0\in\cat T$ with $\supp_R(x_0)=\Spec R$, since $\Spec R$ (with
  the Hochster dual topology) is quasi-compact. This implies
  $\supp_R(\kos{x_0}{\fa})=V(\fa)$ for any ideal $\fa$ of $R$ by
  Lemma~\ref{le:koszul-supp}. Thus $\supp_R \tau_R=\id$.
Now suppose that (2)--(3) hold. This implies for $x\in\cat T$ that
  $\tau_R\supp_R(\thick(x))=\thick(x)$, and therefore
  $\tau_R\supp_R=\id$.  Thus (1)--(3) imply that $\supp_R$ induces a lattice isomorphism.
\end{proof}

\begin{corollary}\label{co:hopkins}
\pushQED{\qed}  If the conditions \emph{(1)--(2)} in
Proposition~\ref{pr:hopkins} hold, then \[{\supp_R}\circ\tau_R=\id.\]
On the other hand, conditions \emph{(2)--(3)} imply
\[\tau_R\circ{\supp_R}=\id.\qedhere\]
\end{corollary}

\subsection*{Tensor triangulated support}

Let $\cat T= (\cat T,\otimes,\one)$ be a \emph{tensor triangulated
  category}. Thus the triangulated category $\cat T$ admits a closed
symmetric monoidal structure
$\otimes\colon\cat T\times\cat T\to\cat T$. In particular, there are
natural isomorphisms $x\otimes y\cong y\otimes x$, which we may treat
as identities, and a unit $\one$ for the tensor product, so
$\one\otimes x\cong x$.  Moreover, there exists an adjoint
$\fHom\colon \cat T^\op\times\cat T\to \cat T$ such that
\begin{equation}\label{eq:function-obj}
  \Hom_\cat T(x\otimes y,z)\cong\Hom_\cat T(x,\fHom(y,z)).
\end{equation}
The functors $\otimes$ and $\fHom$ are assumed to be exact in each
variable.

An object $x$ is \emph{dualisable} if for each
$y\in\cat T $ the natural map
\[
\fHom(x,\one)\otimes y\longrightarrow \fHom(x,y)
\]
is an isomorphism. We assume that the category $\cat T$ is
\emph{rigid}, which means that every object is dualisable. Finally,
assume that $\cat T$ is essentially small and let
$\Thick_\otimes(\cat T)$ denote the lattice of thick tensor ideals of
$\cat T$.

\begin{proposition}\label{pr:tensor-central}
  Let $\cat T=(\cat T,\otimes,\one) $ be an essentially small rigid tensor triangulated
  category.  Then any pair of thick tensor ideals is commuting and
  $\Thick_\otimes(\cat T)$ is a frame.
\end{proposition}
\begin{proof}
  Fix a pair of thick tensor ideal $\cat U,\cat V$.  For $x\in\cat T$
  we set $Dx\colonequals\fHom(x,\one)$.  Given $u\in\cat U$ and $v\in\cat V$, any
  morphism $\phi\colon u\to v$ admits a factorisation
  \[u\iso \one\otimes u\xrightarrow{\text{coev}\otimes u} u\otimes
    Du\otimes u\xrightarrow{\phi\otimes Du\otimes u} v\otimes
    Du\otimes u\xrightarrow{v\otimes\text{ev}} v\otimes\one\iso v\]
  with $v\otimes Du\otimes u$ in $\cat U\wedge\cat V$.  By symmetry
  the same holds for a morphism $v\to u$. Thus the pair
  $\cat U,\cat V$ is commuting.  The (infinite) distributivity now
  follows, as in the proof of Proposition~\ref{pr:centre}.
\end{proof}

\begin{remark}
(1)  For Proposition~\ref{pr:tensor-central} the assumption on the tensor
  triangulated category to be rigid is needed. On the other hand, the
  commutativity of the tensor product is not needed.

(2)   There are examples of  thick subcategories $\cat
  U\subseteq \cat T$ such that $\cat U$ is tensor ideal but not
  central, and conversely that $\cat U$ is central but not tensor ideal.
\end{remark}

\begin{definition}
  A thick tensor ideal $\cat P\subsetneq\cat T$ is \emph{prime} if
  $x\otimes y\in\cat P$ implies $x\in\cat P$ or $y\in\cat P$ for each
  pair $x,y\in\cat T$.  The \emph{Balmer spectrum} of $\cat T$ is the
  set $\Spc\cat T$ of prime thick tensor ideals, endowed with the
  \emph{Zariski topology}. The basic open sets are of the form
  $\{\cat P\in \Spc\cat T\mid x\in\cat P\}$ for some $x\in\cat T$.
\end{definition}

\begin{definition}\label{de:tensor-supp}
  For an object $x\in\cat T$ its \emph{tensor triangulated support} is
\begin{equation*}
\supp_\otimes(x) \colonequals\{\cat P\in \Spc\cat T\mid x\not\in\cat P\}\,.
\end{equation*}
\end{definition}

The support provides a classification of all thick tensor ideals; see
\cite[Theorem~4.10]{Balmer:2005a}.

\begin{theorem}\pushQED{\qed}\label{th:Balmer}
  Let $\cat T$ be an essentially small rigid tensor triangulated
  category.  Then there are mutually inverse isomorphisms of frames
\begin{equation}\label{eq:Balmer-supp}
\begin{tikzcd}
\Thick_\otimes(\cat T) \ar[rr,yshift=2.5,"\supp_\otimes"] &&  \ar[ll,yshift=-2.5,"\tau_\otimes"]
\{\text{Thomason subsets of }\Spc\cat T\}
\end{tikzcd}
\end{equation}
where for each pair $\cat U\subseteq\cat T$ and $V\subseteq\Spc\cat T$ one has
\[
\supp_\otimes(\cat U)\colonequals\bigcup_{x\in\cat U}\supp_\otimes(x)\quad\text{and}\quad
  \tau_\otimes(V)\colonequals\{x\in\cat T\mid\supp_\otimes(x)\subseteq V\}\,.\qedhere
\]
\end{theorem}

The graded endomorphism ring of the tensor unit acts on $\cat T$
via
\[\End^*_{\cat T}(\one)\lto Z^* (\cat T),\quad
  \alpha\longmapsto\alpha\otimes -.\]
Now fix a graded-commutative ring $R$ that acts on $\cat T$
\emph{canonically}, that is to say, via a map of rings
$\phi\colon R\to \End^*_{\cat T}(\one)$. In many cases $R$ is the
graded endomorphism ring of $\one$ and $\phi=\mathrm{id}$, but not
always. The Balmer spectrum is related to the spectrum of the ring $R$
that acts on $\cat T$, via the \emph{comparison map}
\begin{equation}
\label{eq:bs-zs}
\Spc\cat T\xrightarrow{\ \rho_{\cat T}\ } \Spec \End^*_{\cat T}(\one)\xrightarrow{\ \Spec\phi\ } \Spec R
\end{equation}
where $\rho_{\cat T}$ takes a prime ideal $\cat P\subseteq\cat T$ to
the preimage of the unique maximal ideal of
$\End^*_{\cat T/\cat P}(\one)$ under the canonical homomorphism
$\End^*_{\cat T}(\one)\to \End^*_{\cat T/\cat P}(\one)$; see
\cite{Balmer:2010b}.

We wish to compare  tensor triangulated and  cohomological
support. This requires a sequence of lemmas.

\begin{lemma}
  For  objects  $x,y$ in $\cat T$ we have
  \[\supp_R(x\otimes y)\subseteq\supp_R(x)\cap\supp_R(y).\]
\end{lemma}
\begin{proof}
  Let $\fp\in\Spec R$. Then $\fp\in\supp_R(x)$ if and only if
  $\Hom^*_\cat T(x,-)_\fp\neq 0$.  From \eqref{eq:function-obj} it
  follows that $\Hom^*_\cat T(x\otimes y,-)_\fp\neq 0$ implies
  $\Hom^*_\cat T(x,-)_\fp\neq 0$.
\end{proof}

The lemma implies that for any subset $V\subseteq \Spec R$ the thick
subcategory consisting of objects $x\in\cat T$ with $\supp_R(x)\subseteq V$ is a
tensor ideal. In particular, $\tau_R(V)$ is a tensor ideal when $V$ is Thomason.

\begin{lemma}\label{le:koszul}
For a finitely generated ideal $\fa$ of $R$ we have $\cat T_{V(\fa)}=\thick_\otimes(\kos{\one}{\fa})$.
\end{lemma}
\begin{proof}
  The equality
  $\cat T_{V(\fa)}=\thick(\{\kos{x}{\fa}\mid x\in\cat T\})$ follows from
  Proposition~\ref{pr:koszul} and it remains
  to note that $\kos{x}{\fa}= \kos{\one}{\fa}\otimes x$.
\end{proof}

\begin{lemma}\label{le:comparison}
For $r\in\End^*_\cat T(\one)$ we have $\rho^{-1}_\cat T(V(r))=\supp_\otimes(\kos{\one}{r})$.
\end{lemma}
\begin{proof}
  Let $\cat P\in\Spc\cat T$. Then $\cat
  P\in\supp_\otimes(\kos{\one}{r})$ if and only if the morphism
  $\one\xrightarrow{r}\Si^{|r|}\one$ is not invertible in $\cat T/\cat P$. This
  means $r\in\rho_\cat T(\cat P)$, so $\cat P\in \rho^{-1}_\cat T(V(r))$.
\end{proof}

The following result provides a comparison between cohomological and
tensor triangulated support.

\begin{proposition}
\label{pr:bs-lattices}
  Let $\cat T$ be an essentially small rigid tensor triangulated
  category. Then the comparison map \eqref{eq:bs-zs} induces a
  commutative diagram:
\begin{equation*}
  \begin{tikzcd}
    \{\text{Thomason subsets of }\Spec
    R\} \ar[r,"\tau_R"]\ar[d,"\rho^{-1}_{\cat T}(\Spec\phi)^{-1}"] &\Thick_\otimes(\cat T)\ar[d,equal]\\
    \{\text{Thomason subsets of }\Spc\cat T\} \ar[r,"\tau_\otimes", "\sim" swap]&\Thick_\otimes(\cat T)
\end{tikzcd}
\end{equation*}
In particular, the comparison map  \eqref{eq:bs-zs} is a homeomorphism if and only if the
  map $\tau_R$ is an isomorphism.
\end{proposition}

\begin{proof}
  The map $\rho^{-1}_{\cat T}$ is given by taking preimages along
  $\rho_{\cat T}$, and $(\Spec\phi)^{-1}$ is defined analogously.  We
  need to check that these maps are well defined and that they make
  the diagram commutative. First observe that for any $r\in R$ the
  preimage of $V(r)$ under $\Spec\phi$ equals $V(\phi(r))$, and its
  preimage under $\rho_{\cat T}$ equals $\supp_\otimes(\kos{\one}{r})$
  by Lemma~\ref{le:comparison}.  Thus
  \[
  (\rho^{-1}_{\cat T}(\Spec\phi)^{-1})(V(\fa))=\supp_\otimes(\kos{\one}{\fa})
  \] 
  for any finitely generated ideal $\fa$ of $R$, since
  $\rho^{-1}_{\cat T}(\Spec\phi)^{-1}$ and $\supp_\otimes$ preserve finite
  intersections.  Using Lemma~\ref{le:koszul} this yields the
  commutativity, keeping in mind that all maps in the diagram preserve
  arbitrary joins. The diagram implies that $\tau_R$ is a bijection if
  and only if $\rho^{-1}_{\cat T}(\Spec\phi)^{-1}$ is a bijection. Since the
  spaces in question are sober, so determined up to homeomorphism by
  the lattice of open subsets, $\rho^{-1}_{\cat T}(\Spec\phi)^{-1}$ is a
  bijection if and only if $(\Spec\phi)\rho_{\cat T}$ is a
  homeomorphism.
\end{proof}

The tensor triangulated analogue of Proposition~\ref{pr:hopkins} is a consequence.

\begin{corollary}\label{co:HNT}
  For an essentially small rigid tensor triangulated category $\cat T$
  and a ring $R$ acting on it the following are equivalent.
  \begin{enumerate}
\item   
  Cohomological support induces a lattice isomorphism
  \[\Thick_\otimes(\cat T)\longiso   \{\text{Thomason subsets of }\Spec
    R\}.\]
\item The comparison map \eqref{eq:bs-zs} is a homeomorphism.
\item We have $\supp_R(\one)=\Spec R$, $\supp_R$ is finite, and  for all $x,y\in\cat T$
  \[\supp_R(x)\subseteq\supp_R(y)\quad \implies\quad
    \thick_\otimes(x)\subseteq\thick_\otimes(y).\]
\end{enumerate}
\end{corollary}
\begin{proof}
  We apply Proposition~\ref{pr:bs-lattices}.  The equivalence (1)
  $\Leftrightarrow$ (2) uses the adjointness of the pair
  $(\supp_R,\tau_R)$; see Lemma~\ref{le:adjoint-supp}.  For the
  equivalence (1) $\Leftrightarrow$ (3) one adapts the proof of
  Proposition~\ref{pr:hopkins}.
\end{proof}

The prototypical example is the category of perfect complexes over a
commutative ring $A$, on which $A$ acts canonically.

\begin{example}
  Let $A$ be a commutative ring.  The classification of thick tensor
  ideals of the category of perfect complexes via $\Spec A$ means that
  the equivalent conditions in Corollary~\ref{co:HNT} are satisfied;
  it is due to Hopkins, Neeman, and Thomason
  \cite{Hopkins:1987a,Neeman:1992a,Thomason:1997a}.
\end{example}

\section{Local cohomology and support}

We study compactly generated triangulated categories and the basic
properties of localising subcategories that are generated by compact
objects.  Local cohomology functors are introduced when the category
of compact objects is fibred over a space, and this provides a notion
of support.

References are \cite{Benson/Iyengar/Krause:2008a,Hovey/Palmieri/Strickland:1997a}.

\subsection*{Localisation  and  adjoints}

Let $\cat T$ be a triangulated category and $\cat S\subseteq\cat T$ a triangulated
subcategory. Suppose that the canonical functor $q\colon\cat T\to\cat T/\cat S$
admits a right adjoint $q_\rho\colon \cat T/\cat S\to\cat T$.
 Then $q_\rho$ is fully
faithful and induces an equivalence
\[\cat T/\cat S\longiso\cat S^\perp\qquad\text{with quasi-inverse}\qquad
  \cat S^\perp\hookrightarrow\cat T\xrightarrow{\ q\ }\cat T/\cat S,\]
where
\[
  \cat S^\perp\colonequals\{x\in\cat T\mid \Hom_\cat T(s,x)=0\text{ for all
  }s\in\cat S\}
\] and
\[
  ^\perp\cat S\colonequals\{x\in\cat T\mid \Hom_\cat T(x,s)=0\text{ for all
                }s\in\cat S\}\,.
\]
The
unit of the adjunction yields for $x$ in $\cat T$ an exact triangle
\[x'\lto x\stackrel{\eta}\lto q_\rho q(x)\lto \Si x'\]
with $x'$ a direct summand of an object in $\cat S$ since $q(\eta)$ is
invertible.

The following proposition expresses the symmetry which arises from
localising a triangulated category with respect to a thick
subcategory; see for example \cite[Proposition~3.2.8]{Krause:2022a}.

\begin{proposition}\label{pr:verdier-adjoint}\pushQED{\qed} 
  Let $\cat S\subseteq\cat T$ be a thick subcategory. Then the following are equivalent:
\begin{enumerate}
\item The canonical functor $\cat T\to\cat T/\cat S$ admits a right adjoint. 
\item The inclusion $\cat S\to\cat T$ admits a right adjoint.
\item For each $x\in\cat T$ there exists an exact triangle
  $x'\to x\to x''\to\Si x'$ with $x'\in\cat S$ and $x''\in\cat S^\perp$.
\item The composite $\cat S^\perp\hookrightarrow\cat T\twoheadrightarrow\cat T/\cat S$ is a triangle equivalence.
\end{enumerate}
 In that case the right
  adjoint $\cat T\to\cat S$ induces a triangle
  equivalence
  \[\cat T/(\cat S^\perp)\longiso\cat S\qquad\text{and}\qquad^\perp(\cat S^\perp)=\cat S.\qedhere\]
\end{proposition}

We capture the situation in the following diagram 
\begin{equation}\label{eq:tria-loc-seq}
\begin{tikzcd}
\cat S \arrow[tail,yshift=0.75ex]{rr}{i} &&\cat T  \arrow[twoheadrightarrow,yshift=-0.75ex]{ll}{i_\rho}
\arrow[twoheadrightarrow,yshift=0.75ex]{rr}{q} &&\cat T/\cat S \arrow[tail,yshift=-0.75ex]{ll}{q_\rho}
\end{tikzcd}
\end{equation}
which is a localisation sequence. We set
\[\Gamma_\cat S\colonequals i\circ
  i_\rho\qquad\text{and}\qquad L_\cat S\colonequals q_\rho \circ
  q\,.\] Then the adjunctions yield for each object $x\in\cat T$ an
exact triangle
\begin{equation}\label{eq:loc-tria}
  \Gamma_\cat S(x)\lto x\lto L_\cat S(x)\lto\Si \Gamma_\cat S(x).
\end{equation}
In particular, there are natural isomorphisms
\[\Gamma_\cat S \Gamma_\cat S\iso \Gamma_\cat S\qquad\text{and}\qquad  L_\cat S\iso L_\cat S L_\cat S\,.\]

\subsection*{Compactly generated triangulated categories}

Let $\cat T$ be a triangulated category that admits arbitrary
coproducts. An object $x$ in $\cat T$ is called \emph{compact} (or
\emph{small}) if for any morphism
$\phi\colon x\to \coprod_{i\in I}y_i$ in $\cat T$ there is a finite
set $J\subseteq I$ such that $\phi$ factors through
$\coprod_{i\in J}y_i$. It is easily checked that $x$ is compact if and
only if the canonical map
\[ \coprod_{i\in I}\Hom_{\cat T}(x,y_i)\lto \Hom_{\cat T}(x,\coprod_{i\in I}y_i)\]
is bijective for all coproducts $\coprod_{i\in I}y_i$ in $\cat T$. It
follows that the compact objects form a thick subcategory of $\cat T$
which is denoted by $\cat T^{\mathrm c}$.

\begin{definition}
  A set $\cat C$ of compact objects is called \emph{compactly
  generating} if $\cat T$ has no proper localising subcategory
containing $\cat C$. In this case $\cat T$ is called \emph{compactly
  generated}.
\end{definition}

We record some basic facts. For the following, see
\cite[Proposition~3.4.15]{Krause:2022a}.

\begin{lemma}\label{le:comp-gen}\pushQED{\qed} 
  Let $\cat T$ be compactly generated by a set $\cat C$ of compact
  objects. Then
\[\cat T^{\mathrm c}=\thick(\cat C)\,.\qedhere\]
\end{lemma}

We consider the \emph{restricted Yoneda functor}
\[h_\cat T\colon\cat T\lto\Coh\cat T^{\mathrm c},\quad x\mapsto h_x\colonequals\Hom_\cat
  T(-,x)|_{\cat T^{\mathrm c}}\] which is in general neither full nor
faithful. Nonetheless, it is a useful tool. For instance, the functor
\emph{reflects isomorphisms}, so a morphism $\phi$ in $\cat T$ is an
isomorphism if $h_\cat T(\phi)$ is an isomorphism.

Let us fix an adjoint pair of exact functors 
\begin{equation*}
\begin{tikzcd}
\cat T \ar[rr,yshift=2.5,"f"] &&  \ar[ll,yshift=-2.5,"g"]
\cat U
\end{tikzcd}
\end{equation*}
between triangulated categories that admit arbitrary coproducts, and
suppose that $\cat T$ is compactly generated. Then the following
holds.

\begin{lemma}\label{le:comp-gen-adjoints}
  The left adjoint $f$ preserves compactness if and only the right
  adjoint $g$ preserves all coproducts.\qed
\end{lemma}
\begin{proof}
  Fix objects $x\in\cat T$ and $\coprod_{i\in I}y_i\in\cat U$, and suppose that $x$ is
  compact.  We consider the following commutative diagram.
  \[\begin{tikzcd}[column sep = scriptsize]
      \coprod_i\Hom_\cat U(f(x),y_i)
      \arrow{rr}{\alpha}\arrow{d}{\wr}&&\Hom_\cat U(f(x),\coprod_i y_i) \arrow{d}{\wr}\\
      \coprod_i\Hom_\cat T(x,g(y_i)) \arrow{r}{\sim}&\Hom_\cat T(x,\coprod_i
      g(y_i))\arrow{r}{\beta}&\Hom_\cat T(x,g\left(\coprod_i y_i\right))
    \end{tikzcd}\]
  Suppose that $g$ preserves coproducts. Then $\beta$ is an isomorphism,
  and therefore $\alpha$ is an isomorphism. Thus $f(x)$ is compact. The
  converse requires that the compact objects of $\cat T$ are generating.
\end{proof}

Now assume that $f$ preserves compactness. Thus there is an adjoint
pair of functors
\begin{equation*}
\begin{tikzcd}
\Coh\cat T^{\mathrm c} \ar[rr,yshift=2.5,"f^*"] &&  \ar[ll,yshift=-2.5,"f_*"]
\Coh\cat U^{\mathrm c}
\end{tikzcd}
\end{equation*}
where $f^*=(f^{\mathrm c})^*$ and $f_*=(f^{\mathrm c})_*$.

\begin{lemma}\label{le:res-yoneda}
  The pair $(f,g)$ induces the following square
\begin{equation*}
  \begin{tikzcd}
     \cat T \ar[rr,yshift=2.5,"f"]\ar[d,swap,"h_{\cat T}"] && \ar[ll,yshift=-2.5,"g"]
     \cat U \ar[d,"h_{\cat U}"]
     \\
  \Coh{\cat T^{\mathrm c}} \ar[rr,yshift=2.5,"f^*"] && \ar[ll,yshift=-2.5,"f_*"]
  \Coh{\cat U^{\mathrm c}}
\end{tikzcd}
\end{equation*}
which commutes up to isomorphisms.
\end{lemma}
\begin{proof}
For $x\in\cat T$ we have \[h_\cat T(x)=\colim_{C\to X}\Hom_\cat T(-,c)\] where
$c\to x$ runs through all morphisms with $c\in\cat T^{\mathrm c}$. From this we
get a functorial morphism
\[(f^*h_\cat T)(x)=f^*(\colim_{C\to X}\Hom_\cat T(-,c))\cong \colim_{C\to
    X}\Hom_\cat U(-,f(c))\lto h_{\cat U}(f(x)).
\]
The objects  $x\in\cat T$ where this is an isomorphism form a localising
subcategory which contains $\cat T^{\mathrm c}$. Thus $f^*h_\cat T\cong h_{\cat
  U}f$. On the other hand, we have for $y\in\cat U$
\[(h_\cat T g)(y)=\Hom_\cat T(-,g(y))|_{\cat T^{\mathrm c}}\cong\Hom_\cat
  U(f^{\mathrm c}-,y)=(f_*h_\cat U)(y).\qedhere\]
\end{proof}

\subsection*{The lattice of localising subcategories}

Let $\cat T$ be a triangulated category and suppose that arbitrary
coproducts exists in $\cat T$. Let $\Loc(\cat T)$ denote the lattice
of localising subcategories of $\cat T$.  This may not be a set, but
the operations $\vee$ and $\wedge$ are well defined.  For a class of
objects $\cat X\subseteq\cat T$ let $\loc(\cat X)$ denote the smallest
localising subcategory containing $\cat X$. Then for a family
$(\cat S_i)_{i\in I}$ in $\Loc(\cat T)$ we have
\[\bigwedge_{i\in I}\cat S_i=\bigcap_{i\in I}\cat
  S_i\qquad\text{and}\qquad \bigvee_{i\in I}\cat S_i=\loc\left(\bigcup_{i\in
    I}\cat S_i\right).\]

Now assume that $\cat T$ is compactly generated. We are interest are
in the localising subcategories that are generated by compact objects.

Recall that an additive functor $\cat C\to\cat D$ is an
\emph{equivalence up to direct summands} if the functor is fully
faithful and every object in $\cat D$ is isomorphic to a direct
summand of an object in the image.

\begin{proposition}\label{pr:bousfield-comp-gen}
  Let $\cat T$ be a compactly generated triangulated category and 
  $\cat S\subseteq \cat T$ a localising subcategory that is generated by a set
  of compact objects in $\cat T$.
\begin{enumerate}
\item
  There is a localisation sequence
\[\begin{tikzcd}
\cat S \arrow[tail,yshift=0.75ex]{rr} &&\cat T  \arrow[twoheadrightarrow,yshift=-0.75ex]{ll}
\arrow[twoheadrightarrow,yshift=0.75ex]{rr} &&\cat T/\cat S \arrow[tail,yshift=-0.75ex]{ll}
\end{tikzcd}\]
and the category $\cat T/\cat S$ is compactly generated.
\item  We have $\cat S^{\mathrm c}=\cat S\cap\cat T^{\mathrm c}$ and the canonical functor
$\cat T^{\mathrm c}\to\cat T/\cat S$ induces a triangle equivalence up to direct
summands
\[(\cat T^{\mathrm c})/(\cat S^{\mathrm c})\lto (\cat T/\cat S)^{\mathrm c}\,.\]
\item The functors $\Gamma_{\cat S}$ and $L_{\cat S}$ preserve all
  coproducts.
  \end{enumerate}
\end{proposition}
\begin{proof}
  The category $\cat S$ is compactly generated and therefore the
  inclusion $\cat S\to\cat T$ admits a right adjoint, by Brown
  representability. Thus we are in the situation of
  Proposition~\ref{pr:verdier-adjoint}. We apply
  Lemma~\ref{le:comp-gen-adjoints}. The inclusion $\cat S\to\cat T$
  preserves compactness and therefore the right adjoint preserves all
  coproducts. It follows that the right adjoint of
  $\cat T\to\cat T/\cat S$ preserves all coproducts, so
  $\cat T\to\cat T/\cat S$ preserves compactness.  Thus the canonical
  functor $\cat T^{\mathrm c}\to\cat T/\cat S$ induces a functor
  $(\cat T^{\mathrm c})/(\cat S^{\mathrm c})\to (\cat T/\cat S)^{\mathrm c}$, and a cofinality
  argument shows that this is fully faithful. Its image is a full
  triangulated subcategory consisting of compact objects generating
  $\cat T/\cat S$. Thus every compact object in $\cat T/\cat S$ is a
  direct summand of an object in this image, by
  Lemma~\ref{le:comp-gen}. The same lemma yields that
  $\cat S^{\mathrm c}=\cat S\cap\cat T^{\mathrm c}$.
\end{proof}

The assignments
\[\cat T^{\mathrm c}\supseteq\cat C\longmapsto\bar{\cat C}\colonequals\loc(\cat
  C)\qquad\text{and}\qquad\cat T\supseteq\cat S\longmapsto \comp(\cat S)\colonequals\cat S\cap\cat
  T^{\mathrm c}\] induce an adjoint pair of order-preserving maps
\[
  \begin{tikzcd}
\Thick(\cat T^{\mathrm c}) \ar[rr,yshift=2.5,"\loc"] &&  \ar[ll,yshift=-2.5,"\comp"]
\Loc(\cat T)
\end{tikzcd}
\]
which plays a significant role. We collect some basic facts. First
observe that $\bar{\cat C}$ is generated by objects that are compact
in $\cat T$; so Proposition~\ref{pr:bousfield-comp-gen} applies.

\begin{lemma}\label{le:loc-comp-gen}
For a thick subcategory  $\cat C\subseteq\cat T^{\mathrm c}$ the following holds.
\begin{enumerate}
\item $\bar{\cat C}$ consists of all objects $x\in\cat T$ such that every morphism
  $\cat T^{\mathrm c}\ni c\to x$ factors through an object in $\cat C$.
\item $\bar{\cat C}^{\mathrm c}=\cat C=\bar{\cat C}\cap\cat T^{\mathrm c}$.     
\item $\bar{\cat C}^\perp$ consists of all objects $x\in\cat T$ such that $\Hom_\cat T(c,x)=0$ for all
  $c\in\cat C$.
\item $\bar{\cat C}^\perp$ is a localising subcategory of $\cat T$.
\end{enumerate}
\end{lemma}
      
\begin{proof}
  (1) Fix $x\in \cat T$. We may view $\Coh\cat C$ as a subcategory of
  $\Coh\cat T^{\mathrm c}$. Then we have $x\in\bar{\cat C}$ if and only if
  $h_x\in\Coh\cat C$; see Lemma~\ref{le:res-yoneda}. This condition
  means that we can write $h_x=\colim h_{x_i}$ with all
  $x_i\in\cat C$. From this the assertion follows.

(2) Apply Lemma~\ref{le:comp-gen}. 

(3) Fix $x\in \cat T$. The kernel of $\Hom^*_\cat T(-,x)$ is a
localising subcategory; so it contains $\bar{\cat C}$ if it contains
$\cat C$.

(4) From (3) it follows that $\bar{\cat C}^\perp$ is closed under
coproducts.
\end{proof}

\begin{lemma}\label{le:loc-commute}
For  a thick subcategory
$\cat C\subseteq\cat T^{\mathrm c}$ we have
\[\Gamma_{\cat C}\circ h_\cat T\cong h_\cat T\circ\Gamma_{\bar{\cat
  C}}\qquad\text{and}\qquad
  L_{\cat C}\circ h_\cat T\cong h_\cat T\circ L_{\bar{\cat C}}\, .\]
\end{lemma}
\begin{proof}
This follows from Lemma~\ref{le:res-yoneda}, applied to the pair of canonical functors
$\bar{\cat C}\to\cat T$ and $\cat T\to \cat T/\bar{\cat C}$.
\end{proof}

Next we show that the assignment $\cat C\mapsto\bar{\cat C}$ preserves
arbitrary joins, while preserving meets requires an assumption.

\begin{lemma}
For any set of thick subcategories
$\cat C_i\subseteq\cat T^{\mathrm c}$ we have
\[\loc\left(\bigvee_i \cat C_i\right)=\bigvee_i\loc(\cat C_i).\]
\end{lemma}
\begin{proof}
  It follows from Lemma~\ref{le:comp-gen} that both localising
  subcategories have the same class of compact objects; so they agree.
\end{proof}

\begin{proposition}\label{pr:loc-commuting}
  For  a pair of commuting thick subcategories
  $\cat C,\cat D$ of $\cat T^{\mathrm c}$  there are isomorphisms
\[\Gamma_{\bar{\cat C}}\Gamma_{\bar{\cat D}}\leftiso\Gamma_{\overline{\cat C\wedge\cat
      D}}\iso \Gamma_{\bar{\cat D}}\Gamma_{\bar{\cat
      C}}\qquad\text{and}\qquad L_{\bar{\cat C}}L_{\bar{\cat D}}\iso L_{\overline{\cat C\vee\cat
      D}}\leftiso L_{\bar{\cat D}}L_{\bar{\cat C}}\,.\]
Moreover, there is an isomorphism
\[\Gamma_{\bar{\cat C}} L_{\bar{\cat D}}\cong L_{\bar{\cat D}}\Gamma_{\bar{\cat C}}\,.\]
\end{proposition}
\begin{proof}
Using Lemma~\ref{le:loc-commute} we deduce from the commutativity of
$\cat C$ and $\cat D$ that there  are isomorphisms
\[h_\cat T\Gamma_{\bar{\cat C}}\Gamma_{\bar{\cat D}}\leftiso h_\cat
  T\Gamma_{\overline{\cat C\wedge\cat D}}\iso h_\cat
  T\Gamma_{\bar{\cat D}}\Gamma_{\bar{\cat C}}\,.\] This yields the
first pair of isomorphisms since $h_\cat T$ reflects isomorphisms. For
the other isomorphisms the proof is analogous, using \eqref{eq:commuting}.
\end{proof}

\begin{corollary}
 For  a pair of commuting thick subcategories
  $\cat C,\cat D$ of $\cat T^{\mathrm c}$ we have
  \[\loc(\cat C\wedge\cat D)=\loc(\cat C)\wedge\loc(\cat D)\,.\]
\end{corollary}  
\begin{proof}
  For $x\in\cat T$ and a localising subcategory
  $\cat S\subseteq\cat T$ we have $x\in\cat S$ if and only if
  $\Gamma_\cat S(x)\cong x$. With this observation the  assertion
  follows from Proposition~\ref{pr:loc-commuting}
\end{proof}

Any pair of commuting localising subcategories yields a pair of
\emph{Mayer--Vietoris triangles}.

\begin{proposition}\label{pr:loc-MV}
  Let $\cat U$ and $\cat V$ be localising subcategories of $\cat T$
  that are generated by compact objects. Suppose that $\cat U^{\mathrm c}$ and
  $\cat V^{\mathrm c}$ commute. Then for each $x\in\cat T$ there are the
  following canonical exact triangles:
\[
  \Gamma_{\cat U\wedge\cat
    V}(x)\lto\Gamma_{\cat U}(x)\oplus\Gamma_{\cat V}(x)\lto\Gamma_{\cat U\vee\cat
    V}(x)\lto
\]
and
\[
  L_{\cat U\wedge\cat V}(x)\lto L_{\cat U}(x)\oplus L_{\cat V}(x)\lto L_{\cat
    U\vee\cat V}(x)\lto
\]
\end{proposition}
\begin{proof}
  We use the fact that both sequences induce long exact sequences in
  $\Coh\cat T^{\mathrm c}$ which identify with the Mayer--Vietoris
  sequences from Proposition~\ref{pr:MV}; this follows from
  Lemma~\ref{le:loc-commute}. Completing
  $\Gamma_{\cat U\wedge\cat V}(x)\to\Gamma_{\cat
    U}(x)\oplus\Gamma_{\cat V}(x)$ to an exact triangle
\[\Gamma_{\cat U\wedge\cat
    V}(x)\lto\Gamma_{\cat U}(x)\oplus\Gamma_{\cat V}(x)\lto y\lto\]
yields a comparison morphism $y\to\Gamma_{\cat U\vee\cat V}(x)$, and
using the above observation it is straightforward to show that this is
an isomorphism.  For the second triangle the proof is analogous.
\end{proof}

\subsection*{Local cohomology functors}

Let $\cat T$ be a compactly generated triangulated category. Also, let
$X$ be a topological space and write $\Omega(X)$ for its frame of open
subsets. 

\begin{definition}
  Let $\tau\colon\Omega(X)\to \Thick(\cat T^{\mathrm c})$ be a lattice morphism
   that preserves all joins  and such that for each pair $U,V$ of open
  subsets the pair $\tau(U),\tau(V)$ is commuting. Then we say that
  $\cat T$ is \emph{fibred over $X$ via $\tau$}.
\end{definition}

Our examples involve the following duality for spectral spaces.

\begin{definition}
  Let $T$ be a topological space that is spectral. A subset
  $V\subseteq T$ is said to be \emph{Thomason} if it can be written as
  $V=\bigcup_i V_i$ such that each $T\smallsetminus V_i$ is
  quasi-compact and open in $T$. The Thomason subsets are by
  definition the open subsets for the \emph{Hochster dual topology} on
  $T$.  Let $T^\vee$ denote the set $T$ with the Hochster dual
  topology. Then $T^\vee$ is again spectral and $(T^\vee)^\vee=T$; see
  \cite{Hochster:1969a}.
\end{definition}

\begin{example}
  (1) Let $R$ be a graded-commutative ring acting on $\cat T$. Then
  $\cat T$ is fibred over $(\Spec R)^\vee$ via $\tau_R$; see
  Proposition~\ref{pr:Thomason}.

  (2) Suppose that $\cat T^{\mathrm c}$ is a rigid tensor triangulated
  category. Then $\cat T$ is fibred over
  $(\Spc\cat T^{\mathrm c})^\vee$ via the bijection $\tau_\otimes$ in
  \eqref{eq:Balmer-supp}; see Proposition~\ref{pr:tensor-central}.
\end{example}

Now assume that $\cat T$ is fibred over $X$ via a map $\tau$.  For an
open subset $U\subseteq X$ we set
\[\cat T_U\colonequals\loc(\tau(U))\]
and this yields a pair of functors
\[\Gamma_U\colonequals\Gamma_{\cat T_U}\qquad\text{and}\qquad
  L_U\colonequals L_{\cat T_U}\,.\] Note that all functors of the form
$\Gamma_U$ and $L_U$ commute with each other, by
Proposition~\ref{pr:loc-commuting}. This will be used throughout
without further reference. In particular, for any inclusion
$U\subseteq V$ the functor $\Gamma_U$ restricts to a functor
$\cat T_V\to\cat T_U$.

We proceed with a generalisation. Let $Y=V\cap(X\smallsetminus U)$ be
a locally closed subset of $X$, given by a pair $U,V$ of open subsets. Then we
set
\[\cat T_Y\colonequals \cat T_V\cap (\cat
  T_U)^\perp\qquad\text{and}\qquad \Gamma_Y\colonequals \Gamma_V\circ
  L_U\,.\]

\begin{lemma}
  Let $Y=V\cap(X\smallsetminus U)$ be locally closed. Then the functor
  $\Gamma_Y$ and the localising subcategory $\cat T_Y\subseteq\cat T$ do
  not depend on the choice of $U$ and $V$.
\end{lemma}
\begin{proof}
  We use that for any inclusion $V\subseteq U$ we have
  $\Gamma_V L_U=0$.  Let $V\smallsetminus U=V'\smallsetminus U'$. Then
  \[V=(V\smallsetminus U)\cup(V\cap U)\subseteq V'\cup U\] implies
  $L_{V'}\Gamma_V L_U\cong\Gamma_V L_{V'\cup U}=0$. Using the triangle
  \eqref{eq:loc-tria} given by $V'$ yields the isomorphism
  \[\Gamma_{V'}\Gamma_V L_U\cong\Gamma_V L_U\,.\]
  Analogously, $U'\cap V\subseteq U$ and therefore
  $\Gamma_V L_U \Gamma_{U'}\cong\Gamma_{U'\cap V}L_U=0$, which implies
  \[\Gamma_V L_U L_{U'}\cong  \Gamma_V L_U\,.\]
  We conclude that
  \[\Gamma_{V'} L_{U'}\cong \Gamma_{V}\Gamma_{V'} L_{U'} L_{U}\cong
    \Gamma_{V'}\Gamma_{V} L_{U} L_{U'}\cong\Gamma_V L_U\,.\]
  Thus $\Gamma_Y$ does not depend on the choice of $U$ and $V$. For
  $x\in\cat T$ one checks that $\Gamma_Y(x)\cong x$ if and only $x\in\cat T_Y$.
It follows that $\cat T_Y$ does not depend on the choice of $U$ and $V$.
\end{proof}

Let us collect a few basic properties of the functor
$\Gamma_Y\colon\cat T\to\cat T$ given by a locally closed set
$Y\subseteq X$. For a localising subcategory $\cat S\subseteq\cat T$
we set
\[\Gamma_Y\cat S\colonequals\{x\in\cat T\mid \Gamma_Y(x)\cong x\}.\]
Note that $\Gamma_Y\cat T=\cat T_Y$.

\begin{lemma}
Let $Y\subseteq X$ be locally closed.          
  \begin{enumerate}
    \item The functor $\Gamma_Y$ is exact, preserves arbitrary
      coproducts, and $(\Gamma_Y)^2\cong\Gamma_Y$.
      \item If $\cat S\subseteq\cat T$ is localising, then $\Gamma_Y\cat S$
        is a localising subcategory of $\cat T$.
\item Let $(\cat S_i)_{i\in I}$ be a family of localising
  subcategories of $\cat T$. Then

   \[\Gamma_Y\left(\bigvee_i \cat S_i\right)=\bigvee_i \Gamma_Y(\cat
    S_i).\]
\end{enumerate}
\end{lemma}
\begin{proof}
  The proof is straightforward. For the last part note that for any class $\cat
  X\subseteq\cat T$
  \[\Gamma_Y(\loc(\cat X))=\loc(\Gamma_Y(\cat X)).\qedhere\]
\end{proof}

An intersection of two locally closed sets is again locally closed. With
this observation the following calculations are straightforward.

\begin{proposition}\label{pr:loc-closed} 
  Let $Y,Y'$ be a pair of locally closed subsets of $X$. Then
  \[\cat T_Y\cap\cat T_{Y'}=\cat T_{Y\cap Y'}=\Gamma_Y(\cat T_{Y'})\qquad\text{and}\qquad
\Gamma_{Y\cap Y'}\cong     \Gamma_Y \Gamma_{Y'}\,.\qedhere\] 
\end{proposition}
\begin{proof}
  Apply Proposition~\ref{pr:loc-commuting}.
\end{proof}

Here is one application.

\begin{proposition}\label{pr:loc-closed-MV}
Let $Y\subseteq X$ be locally closed, and suppose there are open sets
$U_1,U_2\subseteq X$ such that for $Y_i=Y\cap U_i$ we have $Y=Y_1\cup
Y_2$ and $Y_1\cap Y_2=\varnothing$. Then $\Gamma_Y=\Gamma_{Y_1}\oplus\Gamma_{Y_2}$.
\end{proposition}
\begin{proof}
  Apply $\Gamma_Y$ to the Mayer--Vietoris triangle for $U_1$ and $U_2$
  from Proposition~\ref{pr:loc-MV}.
\end{proof}

\begin{definition}
Let $p\in X$ be locally closed. Then we set $\Gamma_p\colonequals\Gamma_{\{p\}}$
and call this the \emph{local cohomology functor supported at $p$}.
\end{definition}

\subsection*{Support}

We keep a compactly generated triangulated category $\cat T$, fibred
over a space $X$ via $\tau\colon\Omega(X)\to\Thick(\cat T^{\mathrm c})$. We
assume that each point in $X$ is locally closed.

\begin{definition}\label{de:tau-supp}
  For an object $x\in\cat T$ its \emph{$\tau$-support} is
  \[\supp_\tau(x)\colonequals\{p\in X\mid\Gamma_p(x)\neq 0\}\,.\]
\end{definition}

A more suggestive notation is $\supp_R$ or $\supp_\otimes$ when $\tau$ is given by the
action of a ring $R$ or a
tensor product, respectively.

Basic properties of $\supp_\tau$ follow from the fact that for each
$p\in X$ the functor $\Gamma_p$ is exact and preserves all coproducts.

We wish to classify localising subcategories of
$\cat T$ via the support of objects and begin with a simple
observation.

\begin{lemma}\label{le:support-loc-closed}
  Let $Y\subseteq X$ be a locally closed subset and $x\in\cat T$. Then there is an
  equality
  \[\supp_\tau(\Gamma_Y(x))=\supp_\tau(x)\cap Y\,.\]
\end{lemma}
\begin{proof}
  Given $p\in X$, Proposition~\ref{pr:loc-closed} implies
  $\Gamma_p \Gamma_Y(x)=\Gamma_p(x)$ when $p\in Y$, and
  $\Gamma_p \Gamma_Y(x)=0$  when $p\not\in Y$. This yields the assertion.
\end{proof}

Let $P(X)$ denote the power set of $X$. Then we
consider the adjoint pair of maps between posets
\begin{equation*}\label{eq:supp-tau-bar}
\begin{tikzcd}
\Loc(\cat T) \ar[rr,yshift=2.5,"\supp_\tau"] &&  \ar[ll,yshift=-2.5,"\bar\tau"]
P(X)
\end{tikzcd}
\end{equation*}
given by
\[
\supp_\tau(\cat S)\colonequals\bigcup_{x\in\cat S}\supp_\tau(x)\quad\text{and}\quad
  \bar\tau(U)\colonequals\{x\in\cat T\mid\supp_\tau(x)\subseteq U\}\,.
\]
Suppose there is an adjoint pair of maps between posets:
\begin{equation*}
\begin{tikzcd}
\Thick(\cat T^{\mathrm c}) \ar[rr,yshift=2.5,"\sigma"] &&  \ar[ll,yshift=-2.5,"\tau"]\Omega(X)
\end{tikzcd}
\end{equation*}
In our examples there are explicit descriptions of $\sigma$ given by
the support of compact objects, so $\sigma(x)=\sigma(\thick(x))$ for
each compact $x\in\cat T$.  Then we obtain a diagram of poset
morphisms:
\begin{equation}\label{eq:supp-square}
\begin{tikzcd}
\Loc(\cat T) \ar[rr,yshift=2.5,"\supp_\tau"] \ar[d,xshift=2.5,"\comp"] &&
\ar[ll,yshift=-2.5,"\bar\tau"] P(X) \ar[d,xshift=2.5,"\mathrm{int}"]\\
\Thick(\cat T^{\mathrm c}) \ar[u,xshift=-2.5,"\loc"] \ar[rr,yshift=2.5,"\sigma"]&&
\ar[ll,yshift=-2.5,"\tau"]\ar[u,xshift=-2.5,"\inc"]\Omega(X)
\end{tikzcd}
\end{equation}

The following result says that $\supp_\tau$ extends the definition of
support for compact objects.

\begin{proposition}\label{pr:supp-square}
  The square of left adjoints and the square of right adjoints in
  \eqref{eq:supp-square} both commute. In particular,
  $\supp_\tau(x)=\sigma(x)$ for all compact $x\in\cat T$.
\end{proposition}
\begin{proof}
  Let $\cat U\in\Thick(\cat T^{\mathrm c})$ and set
  $V\colonequals\sigma(\cat U)$. Then $\cat U\subseteq\tau(V)$, and
  therefore $\loc(\cat U)\subseteq\cat T_V$.  Now
  Lemma~\ref{le:support-loc-closed} implies
  $\supp_\tau(\loc(\cat U))\subseteq V$. For the other inclusion let
  $p\in V$ and choose an open subset $U\subseteq X$ such that
  $V\smallsetminus U=\{p\}$. Then $\cat U\subseteq \cat T_V$, but
  $\cat U\not\subseteq \cat T_U$. Thus $\Gamma_V(x)=x$ for all
  $x\in\cat U$, but for some $x_0\in\cat U$ we have
  $\Gamma_U(x_0)\neq x_0$, and therefore $\Gamma_p(x_0)\neq 0$. Thus
  $p\in\supp_\tau(\cat U)$. This yields the equality
  $\supp_\tau(\loc(\cat U))=\sigma(\cat U)$ and this means the square of left
  adjoints commutes.  That the square of right adjoints commutes is then a
  formal consequence. The assertion about $\supp_\tau(x)$ follows from
  the first part, because
  \[\supp_\tau(x)=\supp_\tau(\loc(x))=\sigma(\thick(x))=\sigma(x)\,.\qedhere\]
\end{proof}

\begin{lemma}\label{le:supp-surjective}
  For each subset $U\subseteq \supp_\tau(\cat T)$ we have
  $\supp_\tau(\bar\tau(U))=U$.
\end{lemma}
\begin{proof}
  By definition we have $\supp_\tau(\bar\tau(U))\subseteq
  U$. Let $p\in U$ and choose $x\in\cat T$ such that $\Gamma_p(x)\neq
    0$. Then $\Gamma_p(x)\in\bar\tau(U)$ by Lemma~\ref{le:support-loc-closed}.
\end{proof}  

For any locally closed subset $Y\subseteq X$ one may compare $\cat T_Y$ and
$\bar\tau(Y)$. There is one inclusion but it may be strict. In fact, the case
$Y=\varnothing$ yields the obstruction for equality.

\begin{lemma}\label{le:supp-detection}
  For a locally closed subset $Y\subseteq X$ and $x\in\cat T$ we have
  \[x\in\cat T_Y\quad\implies\quad \supp_\tau(x)\subseteq Y,\]
and the reverse implication holds when $\supp_\tau(y)\neq\varnothing$ for
all $0\neq y\in\cat T$.
\end{lemma}
\begin{proof}
 The first implication follows from
  Lemma~\ref{le:support-loc-closed}. For the reverse implication suppose
  $\supp_\tau(y)=\varnothing$ implies $y=0$ for each $y\in\cat T$.
  Let $\supp_\tau(x)\subseteq Y$ and $Y=V\smallsetminus W$ with
  $V,W\in\Omega(X)$. Consider the exact triangle
  $\Gamma_V(x)\to x\to L_V(x)\to$. Then
  $\supp_\tau(L_V(x))=\varnothing$, again by
  Lemma~\ref{le:support-loc-closed}. Thus $\Gamma_V(x)\cong x$ by our
  assumption. The same argument yields $L_W(\Gamma_V(x))\cong x$. Thus
  $x\in\cat T_Y$.
\end{proof}

\subsection*{Comparison of supports}

We keep a compactly generated triangulated category $\cat T$, fibred
over a space $X$ via
$\tau\colon\Omega(X)\to\Thick(\cat T^{\mathrm c})$. Suppose there is a
space $X'$ and a continuous map $f\colon X\to X'$. We set
$\tau':=\tau\circ f^*$, where $f^*\colon\Omega(X')\to\Omega(X)$ denotes
the map induced by $f$, and then $\cat T$ is fibred over $X'$ via
$\tau'$. For any locally closed subset $Y\subseteq X'$ it follows from the definitions that
\[\cat T_Y=\cat T_{f^{-1}(Y)}\qquad\text{and}\qquad\Gamma_Y=\Gamma_{f^{-1}(Y)}.\]

Now assume that points in $X$ and $X'$ are locally closed.

\begin{lemma}\label{le:compare-supp}
  Let $x\in\cat T$. Then
  \[ f(\supp_{\tau}(x)) \subseteq\supp_{\tau'}(x)\] and equality holds
  when $\supp_{\tau}$ detects nonzero objects.
\end{lemma}
\begin{proof}
Let $p\in X'$ and set $U\colonequals f^{-1}(p)$. Then
$\Gamma_p(x)=\Gamma_U(x)$. If $\fq\in\supp_\tau(x)$, then
$\Gamma_p(x)\neq 0$ for $\fp=f(\fq)$. On the other hand, if
$\Gamma_p(x)\neq 0$ and  $\supp_{\tau}$ detects nonzero objects, then
$\Gamma_\fq(x)\neq 0$ for some $\fq\in U$.
\end{proof}

\subsection*{Support via central actions}

In this section we fix a graded-commutative ring $R$ and keep a
compactly generated triangulated category $\cat T$. We set
$\cat C\colonequals\cat T^{\mathrm c}$ and suppose that $\cat T$ is
$R$-linear. Our aim is to provide for a Thomason subset
$V\subseteq \Spec R$ a cohomological description for the objects in
$\cat T_V$ and in $(\cat T_V)^\perp$. More precisely, for
$V=V(\fa)$ given by a finitely generated ideal $\fa$,
we have essentially by definition that
\[(\cat T_V)^{\mathrm c}=(\cat T^{\mathrm c})_V\]
and we extend this description from compact to all objects in $\cat T$.

Let $S\subseteq R$ be a multiplicative set of homogeneous elements. We
identify
\[\Spec R[S^{-1}]=\{\fp\in\Spec R\mid\fp\cap S=\varnothing\}\]
and its complement is Thomason, since
\[(\Spec R)\smallsetminus \Spec R[S^{-1}]=\bigcup_{s\in S}V(s).\]
Recall that an $R$-module $M$ is $S$-local if
$M\cong M[S^{-1}]$, and we have $M[S^{-1}]=0$ if and only if
$\supp_R(M) \subseteq (\Spec R)\smallsetminus \Spec R[S^{-1}]$.

Let $\cat C_S$ denote the kernel of the
canonical functor $\cat C\to\cat C[S^{-1}]$. We set
\[\cat T_S\colonequals\loc(\cat
  C_S)\qquad\text{and}\qquad\cat T[S^{-1}]\colonequals (\cat
  T_S)^\perp.\]
Then there is a triangle equivalence
\[\cat T/\cat T_S\longiso \cat T[S^{-1}]\]
by Proposition~\ref{pr:verdier-adjoint}. Moreover, $\cat T[S^{-1}]$ is
compactly generated and
\[(\cat T^{\mathrm c})[S^{-1}]=(\cat T[S^{-1}])^{\mathrm c}\] up to
direct summands, by Proposition~\ref{pr:bousfield-comp-gen}. We write
$L_S$ for the localisation functor with respect to $\cat T_S$ and set
$x[S^{-1}]\colonequals L_S(x)$ for $x\in\cat T$.

A special case of the above construction arises for each
$\fp\in\Spec R$ by setting
$\cat T_\fp\colonequals\cat T[(R\smallsetminus\fp)^{-1}]$ and
$x_\fp\colonequals L_{R\smallsetminus\fp}(x)$ for each $x\in\cat T$.

\begin{lemma}\label{le:S-localisation}
  Let $S\subseteq R$ be multiplicative. For an object $x\in\cat T$ the
  following holds.
\begin{enumerate}
\item  The canonical
  morphism $x\to x[S^{-1}]$ induces an isomorphism
\[\Hom^*_{\cat T}(c,x)[S^{-1}]\iso \Hom^*_{\cat T}(c,x[S^{-1}])\quad\text{for
    all}\quad c\in\cat T^{\mathrm c}.\]
\item The object $x$ is in $\cat T[S^{-1}]$ if and only if
  $\Hom^*_{\cat T}(c,x)$ is $S$-local for all $c\in\cat T^{\mathrm c}$.
\item  The object $x$ is in $\cat T_S$ if and only if \[\supp_R(\Hom^*_{\cat T}(c,x))\cap\Spec R[S^{-1}]=\varnothing
\quad\text{for all}\quad c\in\cat T^{\mathrm c}.\]
\end{enumerate}
\end{lemma}

\begin{proof}
(1)  Let $c\in\cat T$ be compact. We claim that
\[\supp_R(\Hom^*_{\cat T}(c,x))\cap\Spec R[S^{-1}]=\varnothing\quad\text{for}\quad x\in \cat T_S.\]
This follows from the definitions when $x$ is compact, and holds for
an arbitrary object because it can be built from compacts using
coproducts, suspensions, and extensions. The same argument is used to
show that $\Hom^*_{\cat T}(c,x)$ is $S$-local when $x\in\cat T[S^{-1}]$.
The assertion about objects in $\cat T_S$ implies
that applying \[\Hom^*_{\cat T}(c,-)[S^{-1}]\] to the exact triangle
$\Gamma_S(x)\to x\to x[S^{-1}]\to $ induces an
isomorphism
  \[\Hom^*_{\cat T}(c,x)[S^{-1}]\iso \Hom^*_{\cat T}(c,x[S^{-1}])[S^{-1}].\]
The assertion about objects in $\cat T[S^{-1}]$ implies
that 
\[\Hom^*_{\cat T}(c,x[S^{-1}])\iso \Hom^*_{\cat T}(c,x[S^{-1}])[S^{-1}].\]

(2) One implication has already been shown. So it remains to show that
$x\in\cat T[S^{-1}]$ when $\Hom^*_{\cat T}(c,x)$ is $S$-local for
all $c\in\cat T^{\mathrm c}$. For this it suffices to show that $\Hom^*_{\cat T}(c,x)=0$
when $c\in \cat T_S$. But this is clear since any
$S$-local $R$-module with support in $\Spec R\smallsetminus \Spec
R[S^{-1}]$ is zero.

(3) This follows from (1).
\end{proof}

\begin{lemma}\label{le:tria-supp-V(a)}
  Let $V=V(\fa)\subseteq \Spec R$ be given by a finitely generated
  ideal $\fa$. Then $x\in\cat T$ belongs to $\cat T_V$ if and only if
\[\supp_R(\Hom^*_{\cat T}(c,x))\subseteq V\quad\text{for all}\quad c\in\cat T^{\mathrm c}.\]
\end{lemma}

\begin{proof}
  Let $r_1,\ldots,r_n$ be generators of $\fa$, so
  \[V(\fa)=V(r_1)\cap\cdots\cap V(r_n).\]
  The case $n=1$ follows from Lemma~\ref{le:S-localisation}, and then
  the general case by an induction on $n$ using Proposition~\ref{pr:loc-closed} .
\end{proof}

\begin{proposition}\label{pr:noeth-supp}
  Suppose that $R$ is noetherian and let $x\in\cat T$. Then
  \[\supp_R(x)=\varnothing\quad\implies\quad x=0.\]
\end{proposition}
\begin{proof}
Let $x\neq 0$ and choose a prime $\fp$ that is minimal  in
\[\bigcup_{c\in\cat T^{\mathrm c}}\supp_R(\Hom^*_{\cat T}(c,x)).\] This exists
since $R$ is noetherian. Then it follows from
Lemmas~\ref{le:S-localisation} and \ref{le:tria-supp-V(a)} that
$\Gamma_\fp(x)=x_p\neq 0$, because
$\Gamma_\fp(x)=\Gamma_{V(\fp)}(x_\fp)$. Here, one uses that a
$\fp$-local $R$-module $M$ has $\supp_R(M)\subseteq V(\fp)$  if $M_\fq=0$ for each
prime $\fq$ strictly contained in $\fp$.
\end{proof}

\begin{corollary}
  Suppose that $R$ is noetherian and let $V\subseteq\Spec R$ be a
  Thomason subset. Then for  $x\in\cat T$ the following are
  equivalent.
  \begin{enumerate}
  \item $x\in\cat T_V$.
    \item  $\supp_R(x)\subseteq V$. 
    \item $\supp_R(\Hom^*_{\cat T}(c,x))\subseteq V$ for all $c\in\cat
      T^{\mathrm c}$.
    \end{enumerate}
\end{corollary}
\begin{proof}
  (1) $\Leftrightarrow$ (2) Combine Lemma~\ref{le:supp-detection} and
  Proposition~\ref{pr:noeth-supp}.

  (1) $\Rightarrow$ (3) Apply Lemma~\ref{le:tria-supp-V(a)}.

  (3) $\Rightarrow$ (1) Consider the exact triangle
  $\Gamma_V(x)\to x\to y\to$.  We need to show that $y=0$.  First
  observe that $\supp_R(y)\cap V=\varnothing$ by
  Lemma~\ref{le:supp-detection}. On the other hand,
  $\supp_R(\Hom^*_{\cat T}(c,y))\subseteq V$ for all
  $c\in\cat T^{\mathrm c}$, and therefore $y_\fp=0$ when
  $\fp\not\in V$, by Lemma~\ref{le:S-localisation}. Thus
  $\supp_R(y)=\varnothing$, and therefore $y= 0$ by
  Proposition~\ref{pr:noeth-supp}.
\end{proof}

\section{Stratification}

For a compactly generated triangulated category that is fibred over a
space, we wish to classify its localising subcategories via subsets of
the corresponding space. This leads to the notion of a stratification,
which reduces the classification to a local problem via a local-global
principle. Of prime interest are tensor triangulated categories, but
for actual calculations we employ appropriate ring actions.

References are \cite{Balmer/Favi:2011a,Benson/Iyengar/Krause:2011a,
  Benson/Iyengar/Krause:2012b}.

\subsection*{A local-global principle}

We keep a compactly generated triangulated category $\cat T$, fibred
over a space $X$ via $\tau\colon\Omega(X)\to\Thick(\cat T^{\mathrm c})$. We
assume that each point in $X$ is locally closed.

In general, there is no hope to classify all localising subcategories
of $\cat T$ via the support of objects. For that reason we restrict to a
class of localising subcategories that are compatible with the
additional structure given by $\tau$.

\begin{definition}
  A localising subcategory $\cat S\subseteq\cat T$ is
  \emph{$\tau$-localising}, if $\Gamma_U\cat S\subseteq\cat S$ for
  every open subset $U\subseteq X$, equivalently if
  \[\Gamma_Y\cat S=\cat S\cap \cat T_Y\] for every locally closed subset
  $Y\subseteq X$.
\end{definition}

The next lemma collects some basic facts about $\tau$-localising subcategories.

\begin{lemma}\label{le:tau-local}
  The  $\tau$-localising subcategories are closed
  under all joins, all meets, and also under the assignment
  $\cat S\mapsto\Gamma_Y\cat S$ for every locally closed subset
  $Y\subseteq X$. In this case
  \[\Gamma_Y\left(\bigvee_i \cat S_i\right)=\bigvee_i \Gamma_Y(\cat
    S_i)\qquad\text{and}\qquad \Gamma_Y\left(\bigwedge_i \cat S_i\right)=\bigwedge_i \Gamma_Y(\cat
    S_i).\]
\end{lemma}  
\begin{proof}
  Let $\cat S=\bigvee_i\cat S_i$ be a join of localising subcategories
  of $\cat T$ and $Y\subseteq X$ a locally closed subset. If each
  $\cat S_i$ is $\tau$-localising, then
\[\Gamma_Y\left(\bigvee_i \cat S_i\right)\subseteq \bigvee_i
  \Gamma_Y(\cat S_i) \subseteq \bigvee_i \cat S_i\]
since  $\Gamma_Y$ is exact and preserves all
  coproducts. Thus $\cat S$ is $\tau$-localising. The full subcategory
  $\Gamma_Y\cat S$ is localising, and therefore the reverse inclusion
\[\bigvee_i \Gamma_Y(\cat S_i)\subseteq \Gamma_Y\left(\bigvee_i \cat S_i\right)\]
holds. Now consider a meet $\bigwedge_i\cat S_i$ of $\tau$-localising
subcategories. It is clear that this is $\tau$-localising, and  then
\[\Gamma_Y\left(\bigwedge_i \cat S_i\right)=\left(\bigwedge_i \cat
    S_i\right) \wedge\cat T_Y=\bigwedge_i \left(\cat
    S_i\wedge\cat T_Y\right) =\bigwedge_i \Gamma_Y(\cat S_i).\]

If $\cat S\subseteq\cat T$ is $\tau$-localising, then $\Gamma_Y\cat S$ is
$\tau$-localising, by Proposition~\ref{pr:loc-closed}.
\end{proof}

For a class $\cat X\subseteq\cat T$ let $\loc_\tau(\cat X)$ denote the smallest
$\tau$-localising subcategory of $\cat T$ that contains $\cat X$. One
easily checks that
\[\loc_\tau(\cat X)=\bigvee_{Y\text{ loc.\ cl.}}\loc(\Gamma_Y(\cat X)).\]
In particular, $\Gamma_p\cat S$ is $\tau$-localising for each
localising $\cat S\subseteq\cat T$ and $p\in X$.

\begin{definition}
  The \emph{local-global principle} holds for $(\cat T,\tau)$ if for each
  $\tau$-localising subcategory $\cat S\subseteq\cat T$ one has
  \[\cat S=\bigvee_{p\in X}\Gamma_p\cat S\,.\]
\end{definition}

There are sufficient conditions for the local-global principle
to hold in terms of the space $X$, but this is not needed for our discussion.
Let us reformulate the condition and then we look at its consequences.

\begin{lemma}
  For   $(\cat T,\tau)$ the following are equivalent.
  \begin{enumerate}
  \item The local-global principle holds.
  \item For each $x\in\cat T$ one has
    \[x\in \bigvee_{p\in X}\loc(\Gamma_p(x)).\]
  \item For each $x\in\cat T$ and each $\tau$-localising subcategory
    $\cat S\subseteq\cat T$ one has
\[x\in\cat S\qquad\iff\qquad \Gamma_p(x)\in\cat S\;\text{ for all
  }\;p\in X.\]
\end{enumerate}
\end{lemma}
\begin{proof}
(1) $\Rightarrow$ (2) Apply the local-global principle to $\cat
S=\loc_\tau(x)$.

 (2) $\Rightarrow$ (3) If  $\Gamma_p(x)\in\cat S$  for all $p\in X$,
 then (2) implies that $x\in\cat S$.

 (3) $\Rightarrow$ (1) Set
 $\cat S'\colonequals \bigvee_{p\in X}\Gamma_p\cat S$. Clearly,
 $\cat S'\subseteq\cat S$.  Applying (3)  to $\cat S'$ yields the other inclusion.
\end{proof}

\begin{proposition}\label{pr:tau-local-classification}
  The local-global principle holds for $(\cat
  T,\tau)$ if and only if the assignments
  \[\cat S\mapsto (\Gamma_p\cat S)_{p\in X}\qquad\text{and}\qquad
(\cat S_p)_{p\in X}\mapsto \bigvee_{p\in X}\cat S_p\]
    induce mutually inverse bijections between
\begin{enumerate}
\item $\tau$-localising subcategories  $\cat S\subseteq\cat T$, and
\item families  of $\tau$-localising
      subcategories $(\cat S_p\subseteq\cat T_{\{p\}})_{p\in X}$.
\end{enumerate}
  \end{proposition}

\begin{proof}  
It follows from Lemma~\ref{le:tau-local} that both maps are well
defined. That one composite equals the identity is precisely the
statement of the local-global
principle. For the other composite note that
\[\Gamma_q\left( \bigvee_{p\in X}\cat S_p\right)=\bigvee_{p\in
    X}\Gamma_q\cat S_p=\cat S_q\]
for each $q\in X$. Thus this composite equals the identity.
\end{proof}

The result suggests the additional assumption that for each point
$p\in X$ there are no proper nonzero $\tau$-localising subcategories
of $\cat T_{\{p\}}$, because then the $\tau$-localising subcategories
are classified via their support. This leads to the notion of
`stratification' which we discuss in the setting of compactly
generated tensor triangulated categories.

\subsection*{Compactly generated tensor triangulated categories}

Let $\cat T= (\cat T,\otimes,\one)$ be a \emph{compactly generated
  tensor triangulated category}. To be more precise, the triangulated
category $\cat T$ is compactly generated and admits a closed symmetric
monoidal structure $\otimes\colon\cat T\times\cat T\to\cat T$ with
unit $\one$. Moreover, there exists an adjoint
$\fHom\colon \cat T^\op\times\cat T\to \cat T$ such that
\[\Hom_\cat T(x\otimes y,z)\cong\Hom_\cat T(x,\fHom(y,z))\,.\]
The functors $\otimes$ and $\fHom$ are assumed to be exact in each
variable.

\begin{definition}
An object $x\in\cat T$ is called \emph{dualisable} or \emph{rigid} if,
for every $y\in\cat T$, the natural map
$\fHom(x,\one) \otimes y\to \fHom(x,y)$ is an isomorphism.  We say that
$\cat T$ is \emph{rigidly-compactly generated} if $\one$ is compact
and every compact object is dualisable, which is equivalent to saying
that compact objects coincide with dualisable objects.
\end{definition}

Now suppose that $\cat T$ is rigidly-compactly generated.  Then
$\cat T$ is fibred over $(\Spc\cat T^{\mathrm c})^\vee$ via the bijection
$\tau_\otimes$ in \eqref{eq:Balmer-supp}; see Proposition~\ref{pr:tensor-central}.  We write
$\Loc_\otimes(\cat T)$ for the lattice of localising tensor ideals,
and for any class $\cat X\subseteq \cat T$ let $\loc_\otimes(\cat X)$
denote the smallest localising tensor ideal containing $\cat X$.

\begin{lemma}\label{le:tensor-ideal}
  Let $U\subseteq (\Spc\cat T^{\mathrm c})^\vee$ be an open subset. Then $\cat T_U$ and  $(\cat
  T_U)^\perp$ are localising tensor
  ideals of $\cat T$.
\end{lemma}
\begin{proof}
  The subcategory $\cat U\colonequals\tau_\otimes(U)$ is a thick
  tensor ideal of $\cat T^{\mathrm c}$. By definition we have
  $\cat T_U=\loc(\cat U)$ and $(\cat T_U)^\perp=\cat U^\perp$. We need
  to show that $x\otimes \cat T_U\subseteq \cat T_U$ and
  $x\otimes (\cat T_U)^\perp\subseteq (\cat T_U)^\perp$ for all
  $x\in\cat T$. We may assume that $x$ is compact since $\cat T$ is
  compactly generated. The first inclusion is clear since
  $x\otimes \cat U\subseteq \cat U$. For the second inclusion we may
  assume in addition that $x$ is dualisable, so $x\cong\fHom(\fHom(x,\one),\one)$.
  Then the inclusion
  follows from the fact that $\Hom_\cat T(u,y)=0$ for all $u\in\cat U$
  implies
  \[\Hom_\cat T(u,x\otimes y)\cong \Hom_\cat
    T(u,\fHom(\fHom(x,\one),y))\cong
    \Hom_\cat T(u\otimes\fHom(x,\one), y)=0\]
  for all $u\in\cat U$.
\end{proof}

\begin{lemma}\label{le:tensor-ideal-gamma}
  Let $U\subseteq (\Spc\cat T^{\mathrm c})^\vee$ be an open subset. Then there are isomorphisms
  \[\Gamma_U\cong\Gamma_U(\one)\otimes -\qquad\text{and}\qquad
    L_U\cong L_U(\one)\otimes -\,.\] For a locally closed subset
  $Y\subseteq (\Spc\cat T^{\mathrm c})^\vee$ one has
  \[ \Gamma_Y\cong\Gamma_Y(\one)\otimes -\qquad\text{and}\qquad \cat
    T_Y=\loc_\otimes(\Gamma_Y(\one)) \,.\]
\end{lemma}
\begin{proof}
  Let $x\in\cat T$ and apply $-\otimes x$ to the exact triangle
  \[\Gamma_U(\one)\lto\one\lto L_U(\one)\lto \,.\] Then
  $\Gamma_U(\one)\otimes x$ is in $\cat T_U$ and $L_U(\one)\otimes x$
  is in $(\cat T_U)^\perp$, since both subcategories are tensor ideals
  by Lemma~\ref{le:tensor-ideal}. Applying $\Gamma_U$ to
  \[\Gamma_U(\one) \otimes x\lto x\lto L_U(\one) \otimes x\lto\]
  yields the
  isomorphism $\Gamma_U(\one)\otimes x\iso\Gamma_U(x)$, while applying
  $L_U$ yields $L_U(x)\iso L_U(\one)\otimes x$, since
  $\Gamma_U L_U=0=L_U\Gamma_U$.

Now let $Y$ be locally closed and choose open subsets $U,V$ such that
  $Y=V\smallsetminus U$. Then $\Gamma_Y=\Gamma_VL_U$, and the
  isomorphism  $\Gamma_Y\cong\Gamma_Y(\one)\otimes -$ follows from the
  first part.  The identity for $\cat T_Y$ is an immediate consequence.
\end{proof}

\begin{lemma}\label{le:tensor-tau-local}
  Each localising tensor ideal of $\cat T$ is $\tau_\otimes$-localising.
\end{lemma}
\begin{proof}
  When $U\subseteq(\Spc\cat T^{\mathrm c})^\vee$ is open we have
  $\Gamma_U\cong\Gamma_U(\one)\otimes -$ by
  Lemma~\ref{le:tensor-ideal-gamma}. Thus for any tensor ideal
  $\cat S\subseteq\cat T$ one has $\Gamma_U\cat S\subseteq\cat S$.
\end{proof}

\subsection*{Tensor triangulated support}

Let $\cat T$ be a rigidly-compactly generated tensor triangulated
category. We consider $\cat T$ as fibred over $(\Spc\cat T^{\mathrm c})^\vee$
via $\tau_\otimes$. In addition we assume that each point in
$(\Spc\cat T^{\mathrm c})^\vee$ is locally closed. In particular, the local
cohomology functor $\Gamma_p$ is defined for each point $p$.

For the support of objects in $\cat T$ we recall
Definition~\ref{de:tau-supp} which extends
Definition~\ref{de:tensor-supp} for compact objects.

\begin{definition}
For an object $x\in\cat T$ its \emph{tensor triangulated support} is
  \[\supp_\otimes(x)\colonequals\{p\in \Spc\cat T^{\mathrm c}\mid\Gamma_p(x)\neq
    0\}\,.\]
\end{definition}

This definition yields an adjoint pair of inclusion preserving maps  
\begin{equation*}
\begin{tikzcd}
\Loc_\otimes(\cat T) \ar[rr,yshift=2.5,"\supp_\otimes"] &&  \ar[ll,yshift=-2.5,"\tau_\otimes"]
\{\text{subsets of }\Spc\cat T^{\mathrm c}\}
\end{tikzcd}
\end{equation*}
given by
\[
\supp_\otimes(\cat S)\colonequals\bigcup_{x\in\cat S}\supp_\otimes(x)\quad\text{and}\quad
  \tau_\otimes(U)\colonequals\{x\in\cat T\mid\supp_\otimes(x)\subseteq U\}\,.
\]
This adjoint pair of maps extends \eqref{eq:Balmer-supp}.  More
precisely, we obtain the following  diagram of poset
morphisms
\begin{equation}\label{eq:tensor-supp-square} 
\begin{tikzcd}
\Loc_\otimes(\cat T) \ar[rr,yshift=2.5,"\supp_\otimes"] \ar[d,xshift=2.5,"\comp"] &&
\ar[ll,yshift=-2.5,"\tau_\otimes"]\{\text{subsets of }\Spc\cat T^{\mathrm c}\}\ar[d,xshift=2.5,"\mathrm{int}"]\\
\Thick_\otimes(\cat T^{\mathrm c}) \ar[u,xshift=-2.5,"\loc"] \ar[rr,yshift=2.5,"\supp_\otimes"]&&
\ar[ll,yshift=-2.5,"\tau_\otimes"]\ar[u,xshift=-2.5,"\inc"]\{\text{Thomason subsets of }\Spc\cat T^{\mathrm c}\}
\end{tikzcd}
\end{equation}
which commutes in the following sense.

\begin{proposition}
The square of left adjoints and the square of right
adjoints in \eqref{eq:tensor-supp-square} both commute.
In particular, for each compact $x\in\cat T$ there is an equality
\[\{\cat P\in \Spc\cat T^{\mathrm c}\mid x\not\in\cat P\}=\{p\in \Spc\cat T^{\mathrm c}\mid\Gamma_p(x)\neq
  0\}\,.\]
\end{proposition}
\begin{proof}
  The commutativity of the square of left adjoints means that
  \[\supp_\otimes(\loc(\cat U))=\supp_\otimes(\cat U)\] for each thick
  tensor ideal $\cat U\subseteq\cat T^{\mathrm c}$, where the definition of
  $\supp_\otimes$ depends on the ambient category. This commutativity
  follows as in Proposition~\ref{pr:supp-square}, and then the
  commutativity of right adjoints is a formal consequence. In
  particular, the assertion about $\supp_\otimes(x)$ for a compact
  object $x$ follows since
  \[\supp_\otimes(\thick_\otimes(x)) =\{\cat P\in \Spc\cat T^{\mathrm c}\mid
    x\not\in\cat P\}\]
  and
\[\supp_\otimes(\loc_\otimes(x))=\{p\in \Spc\cat T^{\mathrm c}\mid\Gamma_p(x)\neq
  0\}\,.\qedhere\]
\end{proof}

In general, the support map $\supp_\otimes$ for tensor ideal
localising subcategories is not bijective, though it is surjective.

\begin{lemma}\label{le:tensor-supp-surjective}
  For each subset $U\subseteq \Spc\cat T^{\mathrm c}$ we have
  $\supp_\otimes(\tau_\otimes(U))=U$.
\end{lemma}
\begin{proof}
  This follows from Lemma~\ref{le:supp-surjective}.
\end{proof}

\subsection*{Stratification}

Let $\cat T$ be a rigidly-compactly generated tensor triangulated
category, fibred over $(\Spc\cat T^{\mathrm c})^\vee$ via $\tau_\otimes$. Assume
that each point in $(\Spc\cat T^{\mathrm c})^\vee$ is locally closed.

\begin{definition}
The \emph{local-global principle} holds for localising tensor ideals of $\cat
T$ if one has
\[\cat S=\bigvee_{p\in\Spc\cat T^{\mathrm c}}\Gamma_p\cat S\]
for every localising tensor ideal $\cat S\subseteq \cat T$.
An equivalent condition is that $x\in\cat T$ belongs to a localising
tensor ideal $\cat S$ if
and only if $\Gamma_p(x)\in\cat S$ for each point $p$.
\end{definition}

The local-global principle for $(\cat T,\tau_\otimes)$ implies the
local-global principle for localising tensor ideals, because any
localising tensor ideal is $\tau_\otimes$-localising by
Lemma~\ref{le:tensor-tau-local}.

\begin{definition}
  The tensor triangulated category $\cat T$ is \emph{stratified} via
  $\Spc\cat T^{\mathrm c}$ provided that
\begin{enumerate}
\item   the local-global principle holds for
  localising tensor ideals of $\cat T$, and
\item  for every
  $p\in\Spc\cat T^{\mathrm c}$ there are no proper nonzero localising tensor
  ideals of $\cat T_{\{p\}}$.
\end{enumerate}
The minimality condition (2) has an equivalent formulation:
\begin{equation}\label{eq:minimality}
  p\in\supp_\otimes(x)\qquad\iff\qquad\loc_\otimes(\Gamma_p(\one))=\loc_\otimes(\Gamma_p(x))
\end{equation}  
for all $x\in\cat T$ and $p\in\Spc\cat T^{\mathrm c}$. 
\end{definition}

The next result shows that the stratification of $\cat T$ yields a
classification of all localising tensor ideals via the map
$\supp_\otimes$.  Note that $\supp_\otimes$ is surjective; so we ask
when the map is injective.

Recall that $P(X)$ denotes the power set
of a set $X$.

\begin{proposition}\label{pr:stratification-tensor}
For a rigidly-compactly generated tensor triangulated category $\cat T$, the
following are equivalent.
\begin{enumerate}
\item The category $\cat T$ is stratified via $\Spc\cat T^{\mathrm c}$.
\item The map $\supp_\otimes$ induces
  a lattice isomorphism $\Loc_\otimes(\cat T)\iso P(\Spc\cat T^{\mathrm c})$.
\item $\supp_\otimes(x)\subseteq\supp_\otimes(y)$ implies
  $\loc_\otimes(x)\subseteq\loc_\otimes(y)$ for all $x,y\in\cat T$.
\end{enumerate}
\end{proposition}

\begin{proof}
  (1) $\Rightarrow$ (2) The assignment
  $ U\mapsto \bigvee_{p\in U}\cat T_{\{p\}}$ yields an inclusion
  preserving map $P(\Spc\cat T^{\mathrm c})\to \Loc_\otimes(\cat T)$ which
  satisfies $\supp_\otimes(\bigvee_{p\in U}\cat T_{\{p\}})=U$.
  Assuming that $\cat T$ is stratified, this map is an inverse for
  $\supp_\otimes$, since for any localising tensor ideal
  $\cat S\subseteq\cat T$ one has $\Gamma_p\cat S=\cat T_{\{p\}}$ for
  each $p\in\supp_\otimes(\cat S)$. Thus $\supp_\otimes$ provides a
  lattice isomorphism $\Loc_\otimes(\cat T)\iso P(\Spc\cat T^{\mathrm c})$.

  (2) $\Rightarrow$ (1) Suppose that $\supp_\otimes$ is a
  bijection. For a localising tensor ideal $\cat S$ we have
\[\supp_\otimes(\cat S)=\supp_\otimes\left(\bigvee_{p\in\Spc\cat T^{\mathrm c}}\Gamma_p\cat S\right)\]
and therefore
\[\cat S=\bigvee_{p\in\Spc\cat T^{\mathrm c}}\Gamma_p\cat S\,.\]
Thus the local-global principle holds for localising tensor
ideals. For a nonzero localising tensor ideal
$\cat S\subseteq \cat T_{\{p\}}$ we have
$\supp_\otimes(\cat S)=\{p\}$, and therefore $\cat S=\cat T_{\{p\}}$. We
conclude that $\cat T$ is stratified.

(2) $\Rightarrow$ (3) This is clear since
$\supp_\otimes(\loc_\otimes (x))=\supp_\otimes(x)$ for all
$x\in\cat T$.

  (3) $\Rightarrow$ (2) We need to show that
  $\supp_\otimes\colon\Loc_\otimes(\cat T)\to P(\Spc\cat T^{\mathrm c})$ is
  injective. Then this map is bijective since
  $\supp_\otimes\tau_\otimes=\id$ by
  Lemma~\ref{le:tensor-supp-surjective}. Suppose that $\cat S,\cat S'$ is a
  pair of localising tensor ideals such that
  $\supp_\otimes(\cat S)= \supp_\otimes(\cat S')$, but
  $\cat S\neq\cat S'$. Choose $x\in\cat S\smallsetminus\cat S'$. Then
  $\supp_\otimes(x)\subseteq
  \bigcup_i\supp_\otimes(x_i)=\supp_\otimes(\coprod_i x_i)$ for some
  family of objects $x_i\in\cat S'$. This implies
  $\loc_\otimes(x)\subseteq\loc_\otimes(\coprod_i x_i)\subseteq\cat
  S'$, which is a contradiction. Thus $\supp_\otimes$ is injective.
\end{proof}

Here is one consequence of stratification.

\begin{corollary}
  Let $\cat T$ be a rigidly-compactly generated tensor triangulated
  and suppose it is stratified via $\Spc\cat T^{\mathrm c}$. Then for
  objects $x,y\in\cat T$ we have
  \[\supp_\otimes(x)\cap \supp_\otimes(y)=\supp_\otimes(x\otimes y)\]
  and
  \[\loc_\otimes(x)\cap \loc_\otimes(y)=\loc_\otimes(x\otimes y).\]
\end{corollary}
\begin{proof}
  The isomorphism $\Loc_\otimes(\cat T)\iso P(\Spc\cat T^{\mathrm c})$
  shows that both equalities are equivalent. Thus it suffices to prove
  the first one. We use the isomorphism
  $\Gamma_p(x)\cong\Gamma_p(\one)\otimes x$ for
  $p\in\Spc\cat T^{\mathrm c}$.  Then the inclusion $\supseteq$ is
  clear. The other inclusion uses the stratification, and more
  precisely condition \eqref{eq:minimality}.  Then for
  $p\in \supp_\otimes(x)\cap \supp_\otimes(y)$ we have
\begin{align*}
  \loc_\otimes(\Gamma_p(\one))
  &=\loc_\otimes(\Gamma_p(y))\\
  &=\loc_\otimes(\Gamma_p(\one)\otimes y)\\
  &=\loc_\otimes(\Gamma_p(x)\otimes y) \\
  &=\loc_\otimes(\Gamma_p(x\otimes y)),
\end{align*}
where the first equality uses that $p\in\supp_\otimes(y)$ and the
third equality uses that $p\in\supp_\otimes(x)$.  Thus
$p\in\supp_\otimes(x\otimes y)$.
\end{proof}

\subsection*{Stratification via a central action}

Let $\cat T$ be a rigidly-compactly generated tensor triangulated
category, and suppose that there is a graded-com\-mu\-ta\-tive
ring $R$ that acts on $\cat T$ via  a homomorphism $\phi\colon R\to\End^*_{\cat
  T}(\one)$. Then $\cat T$ is fibred over $(\Spec R)^\vee$ via
$\tau_R$. We assume that each point of  $(\Spec R)^\vee$ is locally
closed. For instance, this holds when $R$ is noetherian.

As before when $\cat T$ is fibred over $(\Spc \cat T^{\mathrm c})^\vee$, we have
for each locally closed subset $Y\subseteq (\Spec R)^\vee$ an
isomorphism $\Gamma_Y\cong\Gamma_Y(\one)\otimes-$ and $\cat T_Y$
is a localising tensor ideal of $\cat T$. This follows from the fact
that $\tau_R$ factors through $\tau_\otimes$; see
Proposition~\ref{pr:bs-lattices}. More precisely, let
$f=(\Spec\phi)\circ\rho_\cat T$ be the comparison map
\eqref{eq:bs-zs}. Then
\[\cat T_Y=\cat T_{f^{-1}(Y)}\qquad\text{and}\qquad\Gamma_Y=\Gamma_{f^{-1}(Y)}.\]

\begin{definition}
  The tensor triangulated category $\cat T$ is \emph{stratified via
    the action of $R$} provided that
  \[\cat S=\bigvee_{\fp\in\Spec R}\Gamma_\fp\cat S\] for every
  localising tensor ideal $\cat S\subseteq \cat T$, and for every
  $\fp\in\Spec R$ there are no proper nonzero localising tensor ideals
  of $\cat T_{\{\fp\}}$.
\end{definition}

It can be shown that the first condition in the above definition is
automatic when the ring $R$ is noetherian. The following result is an
analogue of Proposition~\ref{pr:stratification-tensor}. In fact, it
provides a strategy for establishing a stratification of $\cat T$ via
$\Spc\cat T^c$.

\begin{proposition}\label{pr:stratification-ring}
  For a rigidly-compactly generated tensor triangulated category
  $\cat T$ with the action of a ring $R$, the following are equivalent.
\begin{enumerate}
\item The category $\cat T$ is stratified via the action of $R$.
\item The map $\supp_R$ induces
  a lattice isomorphism $\Loc_\otimes(\cat T)\iso P(\supp_R(\cat T))$.
\item $\supp_R(x)\subseteq\supp_R(y)$ implies
  $\loc_\otimes(x)\subseteq\loc_\otimes(y)$ for all $x,y\in\cat T$.
\end{enumerate}
Moreover, in this case $\cat T$ is also stratified via $\Spc\cat T^c$,
provided that all points of $\Spc\cat T^c$ are locally closed, and the
comparison map \eqref{eq:bs-zs} induces a homeomorphism  $\Spc\cat
T^c\iso\supp_R(\cat T)$ that maps $\supp_\otimes(x)$ to $\supp_R(x)$
for each $x\in\cat T$.
\end{proposition}
\begin{proof}
  For the first part of the assertion the proof of
  Proposition~\ref{pr:stratification-tensor} carries over, simply by
  replacing $\supp_\otimes$ with $\supp_R$.

  Now assume that $\cat T$ is stratified via $R$ and set
  $f\colonequals (\Spec\phi)\circ\rho_\cat T$. Let $\fp\in\Spec R$.
  Then $\cat T_{\{\fp\}}=\cat T_U$ for $U= f^{-1} (\fp)$. Because the
  localising tensor ideal $\cat T_{\{\fp\}}$ is either zero or
  minimal, it follows that the cardinality of $U$ equals either zero
  or one. Thus the comparison map $f$ induces a bijection
  $\Spc\cat T^c\to\supp_R(\cat T)$. In particular, $\cat T$ is is also
  stratified via $\Spc\cat T^c$. The diagram in
  Proposition~\ref{pr:bs-lattices} shows that the comparison map $f$
  induces a lattice isomorphism between Thomason subsets of $\Spec R$
  intersected with $\supp_R(\cat T)$ and Thomason subsets of
  $\Spc\cat T^c$. Thus $\Spc\cat T^c\to\supp_R(\cat T)$ is actually a
  homeomorphism. Finally, $f\circ\supp_\otimes=\supp_R$ by
  Lemma~\ref{le:compare-supp}.
\end{proof}

\begin{example}
  The prototypical example is the derived category of  a
  commutative noetherian ring $A$, on which $A$ acts canonically. The
  classification of localising tensor ideals via $\Spec A$ means that the
  equivalent conditions in Proposition~\ref{pr:stratification-ring} are satisfied; it is
  due to  Neeman \cite{Neeman:1992a}. For a generalisation, see Theorem~\ref{th:Neeman-dga}.
\end{example}

\subsection*{Change of rings and categories}

We fix a graded-commutative ring $R$. Let $\cat T$ and $\cat U$ be
$R$-linear triangulated categories. We say that a functor
$F\colon \cat T\to \cat U$ is \emph{$R$-linear} if it is an exact
functor such that for each $x$ in $\cat T$ the following diagram is
commutative:
\[
  \begin{tikzcd}
    &R\ar[ld,swap,"\phi_x"]\ar[rd,"\phi_{F(x)}"]\\
    \End^*_{\cat T}(x)\ar[rr,"F"]&&   \End^*_{\cat U}(F(x))
  \end{tikzcd}
\]

\begin{lemma}
\label{le:change-cats}
Let $F\colon\cat T\to \cat U$ be an exact functor between compactly
generated $R$-linear triangulated categories which preserves all
products and coproducts. For any locally closed subset $Y$ of
$(\Spec R)^\vee$ we have
\[F(\cat T_{Y})\subseteq\cat U_{Y}\qquad\text{and}\qquad
F\Gamma_{Y}\cong\Gamma_{Y}F .
\]
\end{lemma}

\begin{proof}
  Let $E\colon\cat U\to\cat T$ denote the right adjoint of $F$, which
  exists by Brown representability. Note that $E$ preserves
  compactness by Lemma~\ref{le:comp-gen-adjoints}.  The adjunction
  isomorphisms are $R$-linear, and this implies that $E$ is also
  $R$-linear.
  
  Choose a Thomason subset $V\subseteq\Spec R$. Using for any ideal
  $\fa$ of $R$ the description of
  $(\cat T^{\mathrm c})_{V(\fa)}=(\cat T_{V(\fa)})^{\mathrm c}$ via
  Koszul objects in Proposition~\ref{pr:koszul}, it follows that
  $E((\cat U^{\mathrm c})_{V(\fa)})\subseteq (\cat T^{\mathrm
    c})_{V(\fa)}$, and therefore $E(\cat U_{V})\subseteq \cat
  T_{V}$. Then the adjunction isomorphism
  \[ \Hom^*_{\cat T}(E(x),y)\cong \Hom^*_{\cat U}(x,F(y))\] yields
  that $F((\cat T_{V})^\perp)\subseteq (\cat U_{V})^\perp$.  On the
  other hand, the cohomological description of objects from
  $\cat T_{V(\fa)}$ in Lemma~\ref{le:tria-supp-V(a)} plus the
  adjunction isomorphism imply that
  $F(\cat T_{V})\subseteq \cat U_{V}$. Here one uses that
  \[\cat T_V=\bigvee_{V(\fa)\subseteq V}\cat T_{V(\fa)}.\] This
  establishes the inclusion $F(\cat T_{Y})\subseteq\cat U_{Y}$.

For the isomorphism $F\Gamma_{Y}\cong\Gamma_{Y}F$ it suffices to show
that $F$ commutes with $\Gamma_V$ and $L_V$ since
$\Gamma_Y=\Gamma_V L_U$ for $Y=V\smallsetminus U$.  For each object $x\in\cat T$, one has an
exact triangle
\[
F\Gamma_{V}(x) \lto F(x) \lto F L_{V}X\lto
\]
given by the localisation triangle for $V$. From the first part of
the proof we know that $F\Gamma_{V}(x)$ is in $\cat U_V$ and that
$F L_{V}(x)$ is in $(\cat U_V)^\perp$. Thus the above triangle
identifies with the localisation triangle of $F(x)$.
\end{proof}

Let $S$ be a graded-commutative  ring, and $\cat T$ an
$S$-linear triangulated category.  Given a homomorphism of rings
$\alpha\colon R\to S$, there is a natural $R$-linear structure on
$\cat T$
induced by homomorphisms
\[
R\xrightarrow{\alpha}
S\xrightarrow{\phi_{x}}\End^{*}_{\cat T}(x)\quad\text{for}\quad x\in
  \cat T.
\]

As usual, $\alpha$ induces a map $\alpha^*\colon \Spec S\to \Spec R$,
with $\alpha^{*}(\fq)=\alpha^{-1}(\fq)$ for each $\fq$ in $\Spec
S$. Observe that if $Y\subseteq(\Spec R)^\vee$ is locally closed,
then so is the subset $(\alpha^*)^{-1}Y$ of $(\Spec S)^\vee$.

\begin{lemma}
\label{le:change-rings}
Let $\alpha\colon R\to S$ be a homomorphism of rings, and $\cat T$ an
$S$-linear triangulated category, with induced $R$-linear structure
via $\alpha$. For a locally closed set $Y\subseteq(\Spec R)^\vee$ we
have
\[\cat T_Y=\cat T_{(\alpha^*)^{-1}Y} \qquad\text{and}\qquad
\Gamma_{Y}= \Gamma_{(\alpha^*)^{-1}Y}.\]
\end{lemma}
\begin{proof}
  The second assertion is immediate from the first. We begin with the
  case that $Y=V(\fa)$ for a finitely generated ideal $\fa$ of
  $R$. Then \[(\alpha^*)^{-1}V(\fa)=V(\alpha(\fa))\] and setting 
  $\cat C=\cat T^c$ we get
  \[\cat T_{V(\fa)}=\loc(\cat C_{V(\fa)})=\loc(\cat C_{V(\alpha(\fa))})=\cat
    T_{V(\alpha(\fa))},\] where one uses the description of
  $\cat C_{V(\fa)}$ via Koszul objects in Proposition~\ref{pr:koszul}.

  For a Thomason subset $V\subseteq\Spec R$ we have
  \[(\alpha^*)^{-1}V=\bigcup_{V(\fa)\subseteq V}V(\alpha(\fa))\]
and therefore
  
  \[
    \cat T_V=\bigvee_{V(\fa)\subseteq V}\cat
    T_{V(\fa)}=\bigvee_{V(\fa)\subseteq V}\cat T_{V(\alpha(\fa))}=
\cat T_{(\alpha^*)^{-1}V}.\qedhere
  \]
\end{proof}

Now let  $\cat T$ and $\cat U$ be compactly
generated $R$-linear and $S$-linear triangulated categories,
respectively, $\alpha\colon R\to S$  a homomorphism of graded-commutative rings,
and $F\colon\cat T\to\cat U$  an exact functor that is $R$-linear with respect to the
induced $R$-linear structure on $\cat U$. Thus the
diagram
\[
  \begin{tikzcd}
    R\ar[d,swap,"\phi_x"]\ar[r,"\alpha"]&S\ar[d,"\phi_{F(x)}"]\\
    \End^*_{\cat T}(x)\ar[r,"F"]&   \End^*_{\cat U}(F(x))
  \end{tikzcd}
\]
is commutative for each $x\in \cat T$.

\begin{proposition}
  \label{pr:change-cats-rings}
  For the pair $(F,\alpha)$ suppose that $F$ preserves all products
  and coproducts.  For a locally closed set $Y\subseteq(\Spec R)^\vee$
  we have
\[F\Gamma_{Y}\cong \Gamma_{(\alpha^*)^{-1}Y}F.\]
\end{proposition}
\begin{proof}
Since $F$ is linear with respect to the induced $R$-linear structure
on $\cat U$, Lemmas~\ref{le:change-cats} and
\ref{le:change-rings} yield the assertion.
\end{proof}

\begin{remark}
  There is an analogue of the above proposition for functors between
  compactly generated triangulated categories that are fibred over
  spaces; its formulation is left to the interested reader.
\end{remark}

\section{Differential graded algebras and modular representations}

For a finite group one is interested in its modular representations
over a field of prime characteristic, and we may assume that the prime
divides the order of the group. Otherwise, the group algebra is
semisimple by Maschke's theorem.

In this context we study either the essentially small bounded derived category of
finitely generated modules over the group algebra, or we pass to a
compactly generated triangulated category which contains the bounded derived
category as its subcategory of compact objects. The big category
equals the category of complexes of injective modules, up to homotopy.
Note that projective and injective modules agree over a group algebra,
because it is self-injective. However, we prefer the
injectives since representations are identified with their injective
resolutions. A tensor structure on modular representations is given by
the tensor product over the field and the diagonal action of the
group.

We begin with a classification of the thick tensor ideals of the
bounded derived category. This is done in two steps, so first for
elementary abelian $p$-groups via a Bernstein--Gelfand--Gelfand
correspondence, and then for an arbitrary finite group. The reduction
involves a subgroup theorem for cohomological support which uses
Quillen's stratification of group cohomology. For the group algebra of
an elementary abelian $p$-group we pass to its Koszul complex; this is
a differential graded Hopf algebra which is quasi-isomorphic to an
exterior algebra. So a brief discussion of differential graded
algebras and their derived categories is included.

References are
\cite{Benson/Iyengar/Krause:2011a,Benson/Iyengar/Krause:2011b,
Carlson:2000c,Carlson/Iyengar:2015a}.

\subsection*{Differential graded algebras}

In this section we consider the derived category of diﬀerential graded
(dg for short) modules over a dg algebra. We recall the relevant
definitions.  A \emph{dg algebra} is a $\bbZ$-graded associative
algebra \[A=\bigoplus_{n\in\bbZ}A^n\] over some fixed commutative ring
$k$, together with a \emph{differential} $d\colon A\to A$, so a
homogeneous $k$-linear map of degree $+1$ satisfying $d^2=0$ and the
\emph{Leibniz rule}
\[d(xy)=d(x)y+(-1)^nxd(y)\quad\text{for all}\quad x\in
A^n\quad\text{and}\quad y\in A.\] A
\emph{dg $A$-module} is a
$\bbZ$-graded (right) $A$-module $X$, together with a \emph{differential} $d\colon
X\to X$, so a homogeneous $k$-linear map of degree one
satisfying $d^2=0$ and the \emph{Leibniz rule}
\[d(xy)=d(x)y+(-1)^nxd(y)\quad\text{for all}\quad x\in
  X^n\quad\text{and}\quad y\in A.\] A morphism of dg $A$-modules
$X\to Y$ is an $A$-linear map which is homogeneous of degree zero and
commutes with the differential; it is a \emph{quasi-isomorphism} if it
induces a bijection $H^n(X)\to H^n(Y)$ for all $n\in\bbZ$.

We consider the category of dg $A$-modules and let $\bfD(A)$ denote the
\emph{derived category} which is obtained by formally inverting all
quasi-isomorphisms. This is a compactly generated triangulated
category which is generated by the free module $A$, since
$H^*(X)\cong\Hom_{\bfD(A)}^*(A,X)$ for every $X\in\bfD(A)$.  Thus
the category of compact objects identifies with the full subcategory
of \emph{perfect dg modules}, so direct summands of modules which are
finite iterated extensions of free modules of finite rank. We write
$\bfD^\perf(A)$ for this subcategory, and $\bfD^\fin(A)$ denotes the
full subcategory of dg $A$-modules $X$ such that $H^*(X)$ is finitely
generated over $H^*(A)$.

A morphism $A\to B$ of dg algebras is a \emph{quasi-isomorphism} if it
induces a bijection $H^*(A)\to H^*(B)$.
The dg algebras $A$ and $B$ are \emph{quasi-isomorphic} if there is a chain of
quasi-isomorphisms linking them. The multiplication on $A$ induces one
on its cohomology. We say that $A$ is \emph{formal} if it is
quasi-isomorphic to $H^*(A)$, viewed as a dg algebra with zero
diﬀerential.

A dg algebra $A$ is said to be \emph{commutative} if its underlying
ring is graded-commutative.  In this case, the derived tensor product
of dg modules endows $\bfD(A)$ with a structure of a tensor triangulated
category, with unit $A$. In fact, it is a rigidly-compactly generated
tensor triangulated category with a central action of $H^*(A)$.

A theorem of Neeman classifies the localising subcategories
of the derived category of a commutative noetherian ring
\cite{Neeman:1992a}. The following  is its dg analogue
\cite[Theorem~8.1]{Benson/Iyengar/Krause:2011a}.

\begin{theorem}\label{th:Neeman-dga}
Let $A$ be a commutative dg algebra such that the ring $H^*(A)$ is noetherian.
If $A$ is formal, then $\bfD(A)$ is stratified via the canonical
$H^*(A)$-action. In particular, the comparison map
\[\Spc \bfD^\perf(A)\lto\Spec H^*(A)\] is a homeomorphism.\qed
\end{theorem}

\subsection*{A differential graded Bernstein--Gelfand--Gelfand correspondence}

Let $k$ be a field and let $\Lambda=k\langle\xi_1,\ldots,\xi_r\rangle$
denote an exterior $k$-algebra on 
elements $\xi_1,\ldots,\xi_r$ of odd negative degrees. We
consider $\Lambda$ as dg algebra with zero differential. Let $k$
denote the trivial $\Lambda$-module. Then we have
\[\bfD^\fin(\Lambda)=\thick(k).\]

Let $S=k[x_1,\ldots,x_r]$ denote the symmetric algebra on elements
$x_1,\ldots,x_r$ of degree
$|x_i|=-|\xi_i|+1$, again viewed as a dg algebra with zero
diﬀerential. In this case we have
\[\bfD^\fin(S)=\thick(S).\]

Consider the graded dual $\Hom_k(S,k)$ which is a dg $S$-module. Then
\[F\colonequals\Hom_k(S,k)\otimes_k\Lambda\] with differential given by
multiplication by $\sum_i x_i\otimes\xi_i$ is a  dg $S$-$\Lambda$-bimodule.

\begin{lemma}
  \begin{enumerate}
    \item The augmentations $\Hom_k(S,k)\to k$ and $\Lambda\to k$
      yield a morphism $F\to k$ of dg $\Lambda$-modules which is a quasi-isomorphism.
    \item The $S$-multiplication on $F$ induces a morphism
      $S\to \Hom_\Lambda(F,F)$ of dg $S$-modules which is a
      quasi-isomorphism.\qed
      \end{enumerate}
    \end{lemma}

The following result is a dg analogue of the Bernstein--Gelfand--Gelfand
correspondence \cite{Avramov/Buchweitz/Iyengar/Miller:2010a}.

\begin{theorem}\label{th:BGG-dga}
      The adjoint pair of exact functors
\begin{equation*}
\begin{tikzcd}[column sep = large]
\bfD(S) \ar[rr,yshift=2.5,"-\lotimes_S F"] &&  \ar[ll,yshift=-2.5,"{\RHom_\Lambda(F,-)}"]
\bfD(\Lambda)
\end{tikzcd}
\end{equation*}    
restricts to a pair of equivalences
$\bfD^\fin(S) \rightleftarrows
\bfD^\fin(\Lambda)$ which induces an isomorphism
\[S\longiso\Ext_\Lambda^*(k,k).\]
\end{theorem}
\begin{proof}
  We apply the above lemma. It shows that the pair of functors
  identifies the $S$-module $S$ and the $\Lambda$-module $k$, up to
  isomorphisms. Moreover, the functor $-\lotimes_S F$ induces an
  isomorphism
\[\Hom^*_{\bfD(S)}(S,S)\iso \Hom^*_{\bfD(\Lambda)}(F,F)\cong
  \Hom^*_{\bfD(\Lambda)}(k,k)\] which identifies with the isomorphism
$S\iso\Ext_\Lambda^*(k,k)$ from Koszul duality. Therefore a d{\'e}visage argument
yields the equivalence
\[\bfD^\fin(S)=\thick(S) \iso\thick(k)= \bfD^\fin(\Lambda).\qedhere\]
\end{proof}

The triangle equivalence $\bfD^\fin(S)\iso \bfD^\fin(\Lambda)$
provides an $S$-action on $\bfD^\fin(\Lambda)$ and therefore
Theorem~\ref{th:Neeman-dga} yields a classification of thick
subcategories.

\begin{corollary}\label{co:BGG-dga}
  Cohomological support induces a lattice isomorphism
  \begin{equation*}
     \Thick(\bfD^\fin(\Lambda))\longiso\{\text{Thomason subsets of }\Spec
     S\}\,.
\end{equation*}  
\end{corollary}
\begin{proof}
Combine  Theorems~\ref{th:Neeman-dga} and \ref{th:BGG-dga}.
\end{proof}

\subsection*{Elementary abelian $p$-groups}

We fix a prime $p>0$ and consider the elementary abelian $p$-group
$E=(\bbZ/p)^r$ of rank $r$. Its group algebra over a field $k$ of
characteristic $p$ is isomorphic to
\[A\colonequals k[z_1,\ldots,z_r]/(z_1^p,\ldots,z_r^p)\] which we view as dg algebra
concentrated in degree zero with zero differential.

\begin{proposition}
  For the trivial $A$-module $k$ one has
  \begin{equation}\label{eq:ext-algebra}
    \Ext^*_A(k,k)\cong\begin{cases} k[\eta_1,\ldots,\eta_r]&\text{for }p=2,\\
      k\langle\eta_1,\ldots,\eta_r\rangle\otimes_k
      k[\theta_1,\ldots,\theta_r]&\text{for }p>2,
    \end{cases}
    \end{equation}
  with  $|\eta_i|=1$ and $|\theta_i|=2$ for all $i$.
\end{proposition}
\begin{proof}
  This is an explicit calculation in case $r=1$ and the general case
  follows from the Künneth formula.
\end{proof}

Let $B$ denote the Koszul complex on the sequence $z_1,\ldots,z_r$,
viewed as a dg algebra. It is thus an exterior algebra over $A$ on elements
$y_1,\ldots,y_r$  each of degree $-1$, with diﬀerential determined by
$d(z_i) = 0$ and $d(y_i) = z_i$. Note that $A=B^0$.

\begin{lemma}\label{le:ext-alg-qis}
  Let $\Lambda= k\langle\xi_1,\ldots,\xi_r\rangle$ be an exterior
  algebra over $k$ on degree $-1$ elements $\xi_1,\ldots,\xi_r$,
  viewed as a dg algebra with zero diﬀerential. The morphism
  $\phi\colon\Lambda\to B$ of dg $k$-algebras defined by
  $\phi(\xi_i)= z_i^{p-1}y_i$ is a quasi-isomorphism.  In particular,
  one has \[\Ext^*_B(k,k)=k[x_1,\ldots,x_r]\] with each $x_i$ of degree $2$.
\end{lemma}

\begin{proof}
  A routine calculation shows that $\phi$ is a morphism of dg
  $k$-algebras.  What needs to be verified is that it is a
  quasi-isomorphism, and this is determined only by the structure of
  $\Lambda$ and $B$ as complexes of $k$-vector spaces.  Let
  $\Lambda(i)$ be the exterior algebra on the variable $\xi_i$ and
  $B(i)$ the Koszul dg algebra over $k[z_i]/(z_i^p)$, with exterior
  generator $y_i$. Observe that
  $\phi= \phi(1)\otimes_k \cdots\otimes_k\phi(r)$ where
  $\phi(i)\colon\Lambda(i)\to B(i)$ is the morphism of complexes
  mapping $\xi_i$ to $z_i^{p-1}y_i$. Each $\phi(i)$ is a quasi-isomorphism by
  inspection, and hence so is $\phi$.  Since $\phi$ is a quasi-isomorphism the
  $k$-algebras $\Ext^*_\Lambda(k,k)$ and $\Ext^*_B(k,k)$ are
  isomorphic, and we refer to Theorem~\ref{th:BGG-dga} for the
  description as a symmetric algebra.
\end{proof}

The quasi-isomorphism $\phi\colon\Lambda\to B$ induces an equivalence
$\bfD(\Lambda)\iso\bfD(B)$ which restricts to an equivalence
$\bfD^\fin(\Lambda)\iso\bfD^\fin(B)$.

Next consider the morphism of dg algebras
$\iota\colon A\to B$ which identifies $A$ with $B^0$. This induces an adjoint pair of
exact functors
\begin{equation*}
\begin{tikzcd}[column sep = huge]
\bfD^\fin(A) \ar[rr,yshift=2.5,"\iota^*=-\lotimes_A B"] &&  \ar[ll,yshift=-2.5,"{\iota_*=\RHom_B(B,-)}"]
\bfD^\fin(B)
\end{tikzcd}
\end{equation*}
and restriction along $\iota$ induces a homomorphism
\[ S\colonequals\Ext^*_B(k,k)\xrightarrow{\ \iota_k\ } \Ext^*_A(k,k)=:R.\]

\begin{lemma}\label{le:ext-alg-homeom}
Let $S=k[x_1,\ldots,x_r]$ and keep the generators $\eta_1,\ldots,\eta_r,\theta_1,\ldots,\theta_r$ of $R$
from \eqref{eq:ext-algebra}. The homomorphism $S\to R$ is
injective and given by \[ x_i\longmapsto \begin{cases} \eta_i^2&\text{for }p=2,\\
    \theta_i&\text{for }p>2.
  \end{cases}\] It induces a homeomorphism $\Spec R\to\Spec S$.
\end{lemma}
\begin{proof}
An explicit calculation yields the image of each $x_i$. The induced
map between spectra is a homeomorphism. This is clear for $p=2$ and
uses that odd degree elements are
nilpotent when $p>2$.
\end{proof}

\begin{lemma}\label{le:thick-ind-res}
  For $X\in \bfD^\fin(A)$ one has \[\thick(X)=\thick(\iota_*\iota^*(X)).\]
\end{lemma}
\begin{proof}
  It suffices to show that $\thick(A)=\thick(B)$, because then one can
  apply $X\lotimes_A-$. We use the fact that $\bfD^\fin(A)$ is an
  $A$-linear triangulated category.  For the maximal ideal
  $\fz=(z_1,\ldots,z_r)$ of $A$ we may identify $B=\kos{A}{\fz}$,
 and the equality $\thick(A)=\thick(\kos{A}{\fz})$ follows from
  Proposition~\ref{pr:koszul}.
\end{proof}

The comultiplication of the group algebra $kE$ yields a monoidal
structure on the category of $kE$-modules, which is given by the
tensor product over $k$ and therefore exact. Thus $\bfD^\fin(A)$ is a
tensor triangulated category with unit $k$ and a canonical action of
$R=\Ext_{A}^*(k,k)$. Note that $\bfD^\fin(A)=\thick(k)$ since $A$ is
local ring.

Also, $\bfD^\fin(B)$ is a tensor triangulated category with unit $k$
and a canonical action of $S=\Ext_{B}^*(k,k)$, since there is a pair
of triangle equivalences
\[\bfD^\fin(S)\iso\bfD^\fin(\Lambda)\iso \bfD^\fin(B)\] where the first
is given by Theorem~\ref{th:BGG-dga} and the second is induced by the
quasi-isomorphism $\Lambda\to B$; see Lemma~\ref{le:ext-alg-qis}.

\begin{theorem}\label{th:hopkins-elem-ab}
  Let $A$ be the group algebra of an elementary abelian $p$-group over
  a field $k$ of characteristic $p$.  For the derived category
  $\bfD^\fin(A)$ with the canonical action of $R=\Ext_{A}^*(k,k)$,
  cohomological support induces a lattice isomorphism
  \begin{equation*}
    \Thick(\bfD^\fin(A))\longiso\{\text{Thomason subsets of }\Spec
    R\}\,.
\end{equation*}  
\end{theorem}
\begin{proof}
The adjoint pairs $(\supp,\tau)$  given by the actions of $R$ and $S$ yield the following diagram.
\begin{equation*}
  \begin{tikzcd}
    \Thick(\bfD^\fin(B)) \ar[rr,
    yshift=2.5,"\supp_S"]\ar[d,swap,"\Thick(\iota_*)"] &&\ar[ll,
    yshift=-2.5,"\tau_S"]
    \{\text{Thomason subsets of }\Spec S\}\ar[d,"(\Spec \iota_k)^{-1}"]\\
    \Thick(\bfD^\fin(A)) \ar[rr, yshift=2.5,"\supp_R"] && \ar[ll,
    yshift=-2.5,"\tau_R"]\{\text{Thomason subsets of }\Spec R\}
\end{tikzcd}
\end{equation*}
We wish to show that the map $\tau_R$ is a lattice isomorphism, and
then $\supp_R$ is its inverse.

First observe that the maps of the top row are isomorphisms by
Corollary~\ref{co:BGG-dga}. Also, the vertical map on the right is an
isomorphism by Lemma~\ref{le:ext-alg-homeom}. The map $\tau_R$ is
injective by Corollary~\ref{co:hopkins}. On the other hand, the map
$\Thick(\iota_*)$ is surjective, since
\[\Thick(\iota_*)\circ \Thick(\iota^*)=\id\] by
Lemma~\ref{le:thick-ind-res}.  It remains to show that the diagram
commutes, using the maps from right to left, because that would imply
that all maps are bijections and even lattice isomorphisms. Pick an ideal $\fa$ of
$S$. Then
\[\tau_S(V(\fa))=\bfD^\fin(B))_{V(\fa)}=\thick(\kos{k}{\fa})\] 
by Proposition~\ref{pr:koszul},
and \[(\Spec \iota_k)^{-1}(V(\fa))=V(\iota_k(\fa)).\] We have
$\iota_*(\kos{k}{\fa})=\kos{k}{\iota_k(\fa)}$; see
Remark~\ref{re:iota-koszul} below. Thus
\[\Thick(\iota_*)\tau_S(V(\fa))=\thick(\kos{k}{\iota_k(\fa}))=\tau_R(\Spec \iota_k)^{-1}(V(\fa)).\]
This yields the commutativity of the diagram since all maps in the diagram preserve
arbitrary joins.
\end{proof}

\begin{remark}\label{re:iota-koszul}
  The action of $R$ on $\bfD^\fin(A)$ depends on its monoidal
  structure, and therefore on the comultiplication of $A$, which is
  induced by the group structure of $E$ and given by
\[\Delta(z_i)=z_i\otimes 1 +z_i\otimes z_i + 1\otimes z_i.\]
  There is another
  comultiplication on $A$ that is compatible via
  $\iota\colon A\to B$ with a comultiplication on $B$; it
  corresponds to the multiplication of $S$ and is given by
\[\Delta(z_i)=z_i\otimes 1 + 1\otimes z_i.\] This yields a second
  action of $R$ on $\bfD^\fin(A)$; so for each object
  $X\in \bfD^\fin(A)$ there is a pair of homomorphisms
  $R\to\End^*_{\bfD(A)}(X)$. It can be shown that both agree when
  restricted to $S$ via $\iota_k$; see
  \cite[Theorem~4.4]{Carlson/Iyengar:2017a}. For that reason the above
  equality $\iota_*(\kos{k}{\fa})=\kos{k}{\iota_k(\fa)}$ holds.
\end{remark}

\subsection*{Finite groups}

Let $k$ be a field of characteristic $p>0$ and $G$ a finite group. The
comultiplication of the group algebra $kG$ yields a monoidal structure
on the category of $kG$-modules, which is given by the tensor product
over $k$ and therefore exact. Thus $\bfD^\fin(kG)$ is a
tensor triangulated category with unit $k$ and a canonical action
of \[H^*(G,k)=\Ext_{kG}^*(k,k).\] We write
\[V_G\colonequals\Spec H^*(G,k)\]
and for $X\in\bfD^\fin(kG)$ we set
\[V_G(X)\colonequals\supp_{H^*(G,k)}(X)=V(\Ker\phi_X)\]
where $\phi_X\colon H^*(G,k)\to\Ext^*_{kG}(X,X)$ is given by $\alpha\mapsto\alpha\otimes X$.
Note that $H^*(G,k)$ is a noetherian ring, by a theorem of Evens, Venkov, and
Golod. In particular, the cohomological support $\supp_{H^*(G,k)}$ is finite.

In this section we establish the classification of thick tensor ideals
of  $\bfD^\fin(kG)$. The proof reduces this to the corresponding
classification for elementary abelian $p$-groups. This reduction uses
some results from modular representation theory. 

Let $H\le G$ be a subgroup and consider the induced homomorphism
$kH\to kG$. This yields a bimodule structure $_{kH}(kG)_{kG}$
which gives rise to \emph{restriction} and \emph{induction}. For a
$kG$-module $X$ and a $kH$-module $Y$ we set
\[X\da_H\colonequals\Hom_{kG}(kG,X) \qquad\text{and}\qquad
  Y\ua^G\colonequals Y\otimes_{kH}kG,\]
where the actions of $kH$ and $kG$ are given by the above bimodule
structure of $kG$.
The following natural
isomorphism is known as \emph{Frobenius reciprocity} (or
\emph{projection formula}) and follows from the associativity of the
tensor product:
\[(X\da_H\otimes_k Y)\ua^G\cong X\otimes_k Y\ua^G.\]
When specialising $Y=k$ we get
\[X\da_H\ua^G\cong  X\otimes_k k\da_H\ua^G.\]
Restriction and
induction yield exact functors which extend to complexes of modules
and their derived categories. Note that the induction functor is faithful.

Restriction provides a homomorphism
\[\res_{G,H}\colon H^*(G,k)\lto H^*(H,k)\]
which induces a map
\[\res_{G,H}^*\colon V_H\lto V_G.\]

The next result is known as \emph{subgroup theorem} and is based on
Quillen's stratification of group cohomology.
 
\begin{theorem}\label{th:subgroup-res-thm}
  For  $X\in\bfD^\fin(kG)$ we have
  \[V_H(X\da_H)=(\res^*_{G,H})^{-1}V_G(X).\]
\end{theorem}
\begin{proof}
See  for example Corollary~9.3.3 in \cite{Evens:1991a}.
\end{proof}

The following result is due to Carlson and based on a theorem of Serre
on group cohomology involving the Bockstein map \cite{Serre:1965a}.

\begin{proposition}\label{pr:filtration-elem-abelian}
  There exists a $kG$-module $V$ and a filtration
\[0=V_0\subseteq V_1\subseteq\cdots \subseteq V_t=k\oplus V\]
where for every $i = 1,\ldots,t$ there is an elementary abelian
$p$-subgroup $E_i\le G$ and
a finitely generated $kE_i$-module $W_i$ such that $V_i/V_{i-1}\cong (W_i)\ua^G$.
\end{proposition}
\begin{proof}
  See Theorem~2.1 in \cite{Carlson:2000c}.
\end{proof}

\begin{theorem}\label{th:filtration-elem-abelian}
  There are elementary abelian subgroups $E_1,\ldots,E_t$ such that
  for each $X\in \bfD^\fin(kG)$
  \[\thick_\otimes(X)=\thick\left(\bigoplus_{i=1}^t X\da_{E_i}\ua^G\right).\]
\end{theorem}
\begin{proof}
  Given $Y\in \bfD^\fin(kG)$, we have
  \[X\da_{E_i}\ua^G\otimes_k Y\cong X\otimes_k
    (k\da_{E_i}\ua^G)\otimes_k Y\cong X\otimes_k (Y\da_{E_i}\ua^G)\]
  by Frobenius reciprocity. This yields one inclusion when we
  specialise $Y=k$, and it shows that
  \[\thick\left(\bigoplus_{i=1}^t X\da_{E_i}\ua^G\right)\] is a tensor
  ideal, since   \[Y\da_{E_i}\ua^G\in\thick(k\da_{E_i}\ua^G).\]

  On the other
  hand, Proposition~\ref{pr:filtration-elem-abelian} yields a
  filtration
\[0=X\otimes_k V_0\subseteq X\otimes_k V_1\subseteq\cdots \subseteq
  X\otimes_k V_t=X\otimes_k (k\oplus V)=X\oplus (X\otimes_k V)\] with quotients
\[(X\otimes_k V_i)/(X\otimes_k V_{i-1}) \cong X\otimes_k (V_i/V_{i-1})\cong
  X \otimes_k(W_i)\ua^G\cong (X\da_{E_i}\otimes_k W_i)\ua^G.
\]
We have \[X\da_{E_i}\otimes_k W_i\in\thick(X\da_{E_i})\]
and therefore
\[(X\da_{E_i}\otimes_k W_i)\ua^G\in\thick(X\da_{E_i}\ua^G).\] Thus,
\[X\in \thick\left(\bigoplus_{i=1}^t (X\da_{E_i}\otimes_k W_i)\ua^G
  \right)\subseteq \thick\left(\bigoplus_{i=1}^t X\da_{E_i}\ua^G
  \right)\]
and this yields the other inclusion.
\end{proof}

The following result is due to Benson, Carlson, and Rickard
\cite{Benson/Carlson/Rickard:1997a}, and the proof given here follows
\cite{Carlson/Iyengar:2015a}.

\begin{theorem}\label{th:BCR}
  For objects $X,Y$ in $\bfD^\fin(kG)$ we have
  \[V_G(X)\subseteq V_G(Y)\qquad\implies\qquad
    \thick_\otimes(X)\subseteq\thick_\otimes(Y)\,.\]
Therefore cohomological support induces a lattice isomorphism
  \begin{equation*}
    \Thick_\otimes(\bfD^\fin(kG))\longiso\{\text{Thomason subsets of }\Spec
    H^*(G,k)\}\,.
  \end{equation*}
\end{theorem}
\begin{proof}
  Suppose that \[V_G(X)\subseteq V_G(Y)\]
  holds, and fix an elementary abelian $p$-subgroup $E\le G$.  Then
  Theorem~\ref{th:subgroup-res-thm} implies
 \[V_E(X\da_E)\subseteq V_E(Y\da_E)\]
 and therefore \[\thick(X\da_E)\subseteq\thick(Y\da_E)\] by
 Theorem~\ref{th:hopkins-elem-ab}. This implies
 \[\thick(X\da_E\ua^G)\subseteq \thick(Y\da_E\ua^G)\]
 and then Theorem~\ref{th:filtration-elem-abelian} implies
 \[\thick_\otimes(X)=\thick\left(\bigoplus_{i=1}^t X\da_{E_i}\ua^G\right)
   \subseteq \thick\left(\bigoplus_{i=1}^t
     Y\da_{E_i}\ua^G\right)=\thick_\otimes(Y).\] It remains to apply
 Proposition~\ref{pr:hopkins}, which establishes the isomorphism
 between $\Thick_\otimes(\bfD^\fin(kG))$ and the Thomason subsets of
 $\Spec H^*(G,k)$.
\end{proof}

\begin{remark}
  The full subcategory of objects in $\bfD^\fin(kG)$ that are supported
  at the maximal ideal $H^+(G,k)$ of positive degree elements
  identifies with the category of perfect complexes $\bfD^\perf(kG)$.
  On the other hand, the inclusion
  \[\mod kG\lto \bfD^{\mathrm b}(\mod kG)=\bfD^\fin(kG)\] induces a
  triangle equivalence
\[\stmod kG\longiso \bfD^\fin(kG)/\bfD^\perf(kG),\]
up to direct summands. It follows that cohomological support induces a lattice isomorphism
  \begin{equation*}
    \Thick_\otimes(\stmod kG)\longiso\{\text{Thomason subsets of }\Proj
    H^*(G,k)\}
  \end{equation*}
where
  \[\Proj H^*(G,k)=\Spec H^*(G,k)\smallsetminus\{H^+(G,k)\}.\]
\end{remark}

\section{Stratification for finite groups}

For a group algebra we consider the category of complexes of injective
modules up to homotopy. This is the appropriate compactly generated
tensor triangulated category, because its subcategory of compact
objects identifies with the bounded derived category of finitely
generated modules over the group algebra. The strategy for the
stratification of this tensor triangulated category is completely
analogous to the one for the bounded derived category. So we study
first the case of an elementary abelian $p$-group using a BGG
correspondence, and then we consider an arbitrary finite group. The
reduction uses a subgroup theorem for cohomological support, and the
way Quillen's stratification enters is carefully explained.

References are \cite{Benson/Iyengar/Krause:2011b,Quillen:1971b,Quillen:1971c}.

\subsection*{Graded-injective dg modules}

For a dg algebra $A$ let $A^\sharp$ denote its underlying graded
algebra. A dg $A$-module is called \emph{graded-injective} if it is an
injective object in the category of graded $A^\sharp$-modules. We
write $\bfK(\Inj A)$ for the homotopy category of graded-injective dg
$A$-modules. The objects of this category are graded-injective dg
$A$-modules and morphisms between such dg modules are identified if
they are homotopic. The category $\bfK(\Inj A)$ is given the usual
structure of a triangulated category, where the exact triangles
correspond to exact sequences of graded-injective dg modules.

For a dg algebra $A$ over a field $k$ we consider the following properties:
\begin{enumerate}
\item[(A1)] $A^n=0$ for $n>0$ and $A$ is finite dimensional over $k$. 
\item[(A2)] $A^0$ is a local ring with residue field $k$.
\item[(A3)] $A$ has a structure of a cocommutative dg Hopf $k$-algebra.
\end{enumerate}

Given a dg Hopf algebra $A$ over $k$ and dg $A$-modules $X$ and $Y$, there is a
dg $A$-module structure on $X\otimes_k Y$, obtained by restricting the
natural action of $A\otimes_k A$ along the comultiplication
$A\to A\otimes_k A$.  This is the \emph{diagonal action} of $A$ on
$X\otimes_k Y$.

\begin{proposition}
  Let $A$ be a dg algebra over a field $k$ satisfying
  \emph{(A1)--(A3)}. Then the triangulated category $\bfK(\Inj A)$ is
  compactly generated and the canonical functor
  $\bfK(\Inj A)\to\bfD(A)$ restricts to an equivalence
  \[\bfK^{\mathrm c}(\Inj A)\iso\bfD^\fin(A).\] Moreover, the tensor product
  $\otimes_k$ with diagonal $A$-action endows $\bfK(\Inj A)$ with a
  structure of a rigidly-compactly generated tensor triangulated
  category.
\end{proposition}
\begin{proof}
  See \cite[\S4]{Benson/Iyengar/Krause:2011b}.
\end{proof}

\begin{remark}
  An injective resolution of $k$ provides the tensor unit in
  $\bfK(\Inj A)$, and by abuse of notation we denote this by $k$. The
  graded endomorphism algebra of $k$ identifies with  $\Ext^*_A(k,k)$,
  and  there is a canonical action of this algebra on $\bfK(\Inj A)$.
\end{remark}

\subsection*{Stratification for elementary abelian $p$-groups}

Let $E=(\bbZ/p)^r$ be an elementary abelian $p$-group of rank $r$ and
$k$ a field of characteristic $p$. In this section we extend for the
group algebra $A=kE$ the classification of thick subcategories of
$\bfD^\fin(A)$ to a classification of all localising subcategories of
$\bfK(\Inj A)$.

As before, we use the exterior algebra
$\Lambda=k\langle\xi_1,\ldots,\xi_r\rangle$ and the Koszul complex $B$
on the sequence of generators of $A$.  We view those as dg algebras
and note that $\Lambda$, $A$, and $B$ satisfy the conditions
(A1)--(A3). A \emph{d\'evisage argument} shows that the pair of
triangle equivalences
\[\bfD^\fin(S)\iso\bfD^\fin(\Lambda)\iso \bfD^\fin(B)\] extends
to a pair of triangle equivalences
\[\bfD(S)\iso\bfK(\Inj\Lambda)\iso \bfK(\Inj B).\]
More precisely, the above equivalences are explicitly given and extend
to exact and coproduct preserving functors between compactly generated
triangulated categories, which are necessarily equivalences because
they restrict to equivalences on the level of compact objects.

\begin{proposition}\label{pr:BGG-dga-strat}
  The compactly generated triangulated category $\bfK(\Inj B)$ is
  stratified via the canonical action of $S=\Ext_B^*(k,k)$. In
  particular, cohomological support induces a lattice isomorphism
  \begin{equation*}
    \Loc(\bfK(\Inj B))\longiso\{\text{Subsets of }\Spec
    S\}\,.
  \end{equation*}
\end{proposition}
\begin{proof}
  We use the above equivalence $\bfD(S)\iso \bfK(\Inj B)$, and then the
  assertion follows from the stratification of $\bfD(S)$ via the
  canonical action of $S$; see Theorem~\ref{th:Neeman-dga}.
\end{proof}

Next consider the dg algebra homomorphism $\iota\colon A\to B$ which
induces an adjoint pair of
exact functors
\begin{equation*}
\begin{tikzcd}[column sep = huge]
  \bfK(\Inj A) \ar[rr,yshift=2.5,"\iota^*=-\otimes_A B"] &&
  \ar[ll,yshift=-2.5,"{\iota_*=\Hom_B(B,-)}"] \bfK(\Inj B).
\end{tikzcd}
\end{equation*}
This restricts to a pair of functors
\begin{equation*}
\begin{tikzcd}
\bfK^{\mathrm c}(\Inj A) \ar[rr, yshift=2.5] &&  \ar[ll, yshift=-2.5,]
\bfK^{\mathrm c}(\Inj B).
\end{tikzcd}
\end{equation*}
The proof of Theorem~\ref{th:hopkins-elem-ab} shows that there is an
induced pair of lattice isomorphisms
\begin{equation*}
\begin{tikzcd}
\Thick(\bfK^{\mathrm c}(\Inj A)) \ar[rr,yshift=2.5] &&  \ar[ll,yshift=-2.5]
\Thick(\bfK^{\mathrm c}(\Inj B)).
\end{tikzcd}
\end{equation*}

The above properties of the pair $(\iota^*,\iota_*)$ imply that
$\iota_*$ preserves the localising subcategories which are given by a
single prime.

For a localising subcategory $\cat U$ of $\bfK(\Inj B)$
set \[\Loc(\iota_*)(\cat U)\colonequals\loc(\iota_*(\cat U)).\]

\begin{lemma}\label{le:minimal-loc-ind-res}
 Let $\fp\in\Spec R$ correspond to $\fq\in\Spec S$ via the
 homomorphism $S\to R$. Then
\[\Loc(\iota_*)(\Gamma_\fq\bfK(\Inj B))=\Gamma_\fp\bfK(\Inj A).\]
\end{lemma}
\begin{proof}
The construction of $\Gamma_\fq\bfK(\Inj B)$ shows that it
is generated by compact objects which are mapped by $\iota_*$ to
compact generators of $\Gamma_\fp\bfK(\Inj A)$.
\end{proof}

The following is the analogue of Lemma~\ref{le:thick-ind-res} for
localising subcategories.

\begin{lemma}\label{le:loc-ind-res}
  For $X\in \bfK(\Inj A)$ one has \[\loc(X)=\loc(\iota_*\iota^*(X)).\]
\end{lemma}
\begin{proof}
  From Lemma~\ref{le:thick-ind-res} we know that
  $\thick(A)=\thick(B)$. Then the assertion follows by applying $X\otimes_A-$, since
  $X\otimes_AA=X$ and $X\otimes_AB=\iota_*\iota^*(X)$.
\end{proof}

The following result extends Theorem~\ref{th:hopkins-elem-ab} and its
proof uses a similar strategy.

\begin{theorem}\label{th:strat-elem-abelian}
  The compactly generated triangulated category $\bfK(\Inj A)$ is
  stratified via the canonical action of $R=\Ext_A^*(k,k)$. In
  particular, cohomological support induces a lattice isomorphism
  \begin{equation*}
    \Loc(\bfK(\Inj A))\longiso\{\text{Subsets of }\Spec
    R\}\,.
  \end{equation*}
\end{theorem}
\begin{proof}
  We consider the following diagram
\begin{equation*}
  \begin{tikzcd}
    \Loc(\bfK(\Inj B)) \ar[rr, "\supp_S"]\ar[d,swap,"\Loc(\iota_*)"]
    &&    \{\text{Subsets of }\Spec S\}\ar[d,"(\Spec \iota_k)^{-1}"]\\
    \Loc(\bfK(\Inj A)) \ar[rr,"\supp_R"] && \{\text{Subsets of }\Spec
    R\}
\end{tikzcd}
\end{equation*}
and wish to show that the map $\supp_R$ is a lattice isomorphism.

First observe that the top horizontal map is an isomorphism by
Proposition~\ref{pr:BGG-dga-strat}. Likewise, the vertical map on the
right is an isomorphism by Lemma~\ref{le:ext-alg-homeom}. Furthermore,
the map $\Loc(\iota_*)$ is surjective, since
\[\Loc(\iota_*)\circ \Loc(\iota^*)=\id\] by
Lemma~\ref{le:loc-ind-res}.  It remains to show that the diagram
commutes, because that would imply that all maps are bijections and
even lattice isomorphisms. All maps in the diagram preserve
arbitrary joins. So it suffices to pick $\fp\in\Spec R$ corresponding
to $\fq\in\Spec S$. Then Lemma~\ref{le:minimal-loc-ind-res} tells us that
\[\Loc(\iota_*)(\Gamma_\fq\bfK(\Inj B))=\Gamma_\fp\bfK(\Inj A).\]
From
$\supp_R(\Gamma_\fp\bfK(\Inj A))=\{\fp\}$ and
$\supp_S(\Gamma_\fq\bfK(\Inj B))=\{\fq\}$ it follows that
\[\supp_R\circ\Loc(\iota_*)=(\Spec \iota_k)^{-1}\circ\supp_S.\]
This completes the proof.
\end{proof}

\subsection*{Quillen stratification}

In \cite{Quillen:1971b,Quillen:1971c} Quillen described the
\emph{maximal} ideal spectrum\index{maximal ideal spectrum} of
$H^*(G,k)$ in terms of elementary abelian subgroups.  We are really
interested in \emph{prime} ideals, but let us first describe Quillen's
original theorem. Note that in a finitely generated graded-commutative
$k$-algebra such as $H^*(G,k)$ every prime ideal is the intersection
of the maximal ideals containing it, by a version of Hilbert's
Nullstellensatz.

Let $G$ be a finite group, and $k$ a field of characteristic $p$. If
$H$ is a subgroup of $G$, there is a \emph{restriction
  map}  $H^*(G,k)\to H^*(H,k)$ which is a ring
homomorphism. Writing \[\bar V_G\colonequals\MaxSpec H^*(G,k)\] this induces a map of
varieties $\bar V_H\to \bar V_G$.

\begin{theorem}
An element $u\in H^*(G,k)$ is nilpotent if and only if $\res_{G,E}(u)$ is
nilpotent for every elementary abelian $p$-subgroup $E\le G$. \qed
\end{theorem}

It thus makes sense to look at the product of the restriction maps
\[ 
H^*(G,k) \lto \prod_E H^*(E,k). 
\]
The theorem implies that the kernel of this map is nilpotent.  What is the image?\medskip

If an element $(u_E)$ is in the image then
\begin{enumerate}
\item[(C1)] for each conjugation $c_g\colon E' \to E$ in $G$, 
$c_g^*(u_E)=u_{E'}$ and
\item[(C2)] for each inclusion $i\colon E' \to E$ in $G$, $i^*(u_E)=u_{E'}$.
\end{enumerate}

Conversely, if $(u_E)$ satisfies these conditions, Quillen showed
that for some $t\ge 0$ the element $(u_E^{p^t})$ is in the image.

\begin{definition}
We define $\varprojlim_{E} H^*(E,k)$ to be the elements $(u_E)$ of the direct product $\prod_E H^*(E,k)$ satisfying
conditions (C1) and (C2) above.
\end{definition}

\begin{definition}
A homomorphism $\phi\colon R\to S$ of graded-commutative $k$-algebras 
is an \emph{$F$-isomorphism} or 
\emph{inseparable isogeny} if
\begin{enumerate}
\item[(F1)] the kernel of $\phi$ consists of nilpotent elements, and
\item[(F2)] for each $s\in S$ there exists $t\ge 0$ such that $s^{p^t}\in\Im \phi$.
\end{enumerate}
\end{definition}

We can now rephrase Quillen's theorem as follows:

\begin{theorem}
\pushQED{\qed}
The restriction maps induce an $F$-isomorphism
\[ 
H^*(G,k)\lto\varprojlim_{E} H^*(E,k)\,.\qedhere
\]
\end{theorem}

Now if $R\to S$ is an $F$-isomorphism of finitely generated 
graded-commutative $k$-algebras then the induced map $\MaxSpec S \to \MaxSpec R$
is a \emph{bijection}.

For the moment, let us assume that $k$ is algebraically
closed. Then this says that
$\varinjlim_{E}\bar V_E\to \bar V_G$, as a map of varieties, is \emph{bijective}
(but not necessarily invertible).

Here, $\varinjlim_{E}\bar V_E$ is the quotient of the disjoint union
of the $\bar V_E$ by the equivalence relation induced
by the conjugations and inclusions.

\subsection*{Prime ideals and subgroup theorems}

Now we discuss prime ideals in $H^*(G,k)$, and we drop the assumption
that $k$ is algebraically closed.  We write
\[V_G\colonequals\Spec H^*(G,k)\]
and for $X\in\bfK(\Inj kG)$ we set
\[V_G(X)\colonequals\supp_{H^*(G,k)}(X).\]

By Quillen's theorem, for each
$\fp \in V_G$ there exists an elementary abelian $p$-subgroup
$E\le G$ such that $\fp$ is in the image of $\res_{G,E}^*$.  We say
that $\fp$ \emph{originates} in such an $E$ if there
does not exist a proper subgroup $E'$ of $E$ such that $\fp$ is in the
image of $\res_{G,E'}^*$. In this language, \cite[Theorem~10.2]{Quillen:1971c} reads:

\begin{theorem}
\label{th:origin}
For each $\fp \in V_G$, the pairs $(E,\fq)$ where $\fq\in V_E$,  $\fp=\res_{G,E}^*(\fq)$ and
such that $\fp$ originates in $E$ are all $G$-conjugate. \qed
\end{theorem}

For any subgroup $H\le G$, formal properties of restriction and induction
prove the following.

\begin{lemma}\label{le:supp-res-ind}
Let $\fp \in V_G$ and set $U=(\res_{G,H}^*)^{-1}\{\fp \}$. If $X\in
\bfK(\Inj kG)$ and $Y\in \bfK(\Inj kH)$ then
\[\Gamma_\fp(X)\da_H\cong \Gamma_U(X\da_H)
  \qquad\text{and}\qquad\Gamma_\fp (Y\ua^G)\cong \Gamma_U (Y)\ua^G.\]
\end{lemma}

\begin{proof}
 The first isomorphism follows from Proposition~\ref{pr:change-cats-rings}.
  The second isomorphism follows from the first and Frobenius
  reciprocity, since
  \begin{align*}
    \Gamma_\fp (Y\ua^G)&\cong \Gamma_\fp (k) \otimes_k Y\ua^G\\
                       &\cong(\Gamma_\fp (k)\da_H \otimes_k Y)\ua^G\\
                       &\cong(\Gamma_U (k) \otimes_k Y)\ua^G\\
&\cong\Gamma_U(Y)\ua^G.\qedhere
  \end{align*}
\end{proof}

Using Mayer--Vietoris triangles one can actually show that
$\Gamma_U=\coprod_{\fq\in U}\Gamma_\fq$ since $U$ is discrete; see
Proposition~\ref{pr:loc-closed-MV}. But this is not needed for the
following.

\begin{proposition}\label{pr:V-res-ind}
  For $Y\in \bfK(\Inj kH)$ we have
  \[V_G(Y\ua^G)=\res_{G,H}^* V_H(Y).\]
\end{proposition}
\begin{proof}
It follows from
  Lemma~\ref{le:supp-res-ind} and faithfulness of induction that the condition $\fp\in V_G(Y\ua^G)$
  is equivalent to: there exists $\fq\in V_H$ such that
  $\res^*_{G,H}(\fq) = \fp$ and $\Gamma_\fq(Y)\neq 0$.
\end{proof}

The following proof of the subgroup theorem for a pair $H\le G$ uses
that $\bfK(\Inj kG)$ is stratified; so at this stage it works when $G$
is elementary abelian.

\begin{theorem}
\label{th:subgroup}
Let $H\le G$ and suppose that $G$ is an elementary abelian $p$-group. For  $X$ in $\bfK(\Inj
kG)$ we have
\[ 
 V_{H}(X{\da_{H}})=(\res_{G,H}^*)^{-1} V_G(X)\,. 
\]
\end{theorem}

The proof shows that the inclusion $\subseteq$  is purely
formal, while there are other tensor triangulated categories  where the
analogue of $\supseteq$ does not hold.

\begin{proof}
Let $\fq \in V_{H}$ and $\fp=\res_{G,H}^*(\fq)$. Then
Proposition~\ref{pr:V-res-ind} yields
\[ 
 V_G(\Gamma_\fq(k)\ua^G)=\{\fp \}= V_G(\Gamma_\fp(k)),
\] 
and the classification of localising subcategories for $kG$
(Theorem~\ref{th:strat-elem-abelian})  implies 
\[ 
\loc((\Gamma_\fq k)\ua^G)=\loc(\Gamma_\fp(k)). 
\]
Thus,
\begin{align*} 
\Gamma_\fq(X\da_{H})\ne 0 
&\iff \Gamma_\fq(k)\otimes_k X\da_{H}\ne 0 \\
&\iff (\Gamma_\fq(k)\otimes_k X\da_{H})\ua^G\ne 0 \\
&\iff \Gamma_\fq(k)\ua^G\otimes_k X \ne 0 \\
&\iff \Gamma_\fp(k)\otimes_k X\ne 0 \\
&\iff \Gamma_\fp (X) \ne 0,
\end{align*}
where the second equivalence uses that $(-)\ua^G$ is faithful,  the
third equivalence uses Frobenius reciprocity, and  the fourth
equivalence uses the above equality of localising subcategories.
\end{proof}

The following result is the analogue of
Theorem~\ref{th:filtration-elem-abelian} for bounded complexes of
finitely generated $kG$-modules.

\begin{theorem}\label{th:filtration-finite-group}
  There are elementary abelian subgroups $E_1,\ldots,E_t$ such that
  for each $X\in \bfK(\Inj kG)$
  \[\thick_\otimes(X)=\thick\left(\bigoplus_{i=1}^t X\da_{E_i}\ua^G\right).\]
\end{theorem}
\begin{proof}
The proof of Theorem~\ref{th:filtration-elem-abelian} carries over
without any changes.
\end{proof}

We continue with a variant of the subgroup theorem which is sufficient
for us; it uses Quillen's stratification of group cohomology, as
formulated in Theorem~\ref{th:origin}.

\begin{theorem}\label{th:subgroup-elem-abelian}
Let $E\le G$ be an elementary abelian $p$-subgroup. For $X,Y$ in
$\bfK(\Inj kG)$ we have
\[V_G(X)\subseteq V_G(Y)\qquad\implies\qquad V_E(X\da_E)\subseteq V_E(Y\da_E).\]
\end{theorem}
\begin{proof}
  Let $\fq\in V_E(X\da_E)$ and set $\fp=\res^*_{G,E}(\fq)$. Then
  $\fp\in V_G(X)$ by Theorem~\ref{th:subgroup}. Let $E_0\leq E$
  be a subgroup such that $\fp$ originates in $E_0$, and pick
  $\fq_0\in V_{E_0}$ such that $\fq=\res^*_{E,E_0}(\fq_0)$. Then
  $\fq_0\in V_{E_0}(X\da_{E_0})$, again by Theorem~\ref{th:subgroup}.

  We wish to show that $\fq\in V_E(Y\da_E)$ and may assume that
  $Y=\Gamma_\fp(Y)$, since
  $V_E((\Gamma_\fp(Y))\da_E)\subseteq V_E(Y\da_E)$. Choose an
  elementary abelian $p$-subgroup $E'\leq G$ such that
  $Y\da_{E'}\neq 0$; this exists by
  Theorem~\ref{th:filtration-finite-group}. Let
  $\fq'\in V_{E'}(Y\da_{E'})$. Then $\res^*_{G,E'}(\fq')=\fp$ by
  Theorem~\ref{th:subgroup}.  Choose a subgroup $E'_0\leq E'$ such that
  $\fp$ originates in $E'_0$, together with a prime
  $\fq'_0\in V_{E'_0}(Y\da_{E'_0})$ such that
  $\fq'=\res^*_{E',E'_0}(\fq'_0)$.  The pairs $(E_0,\fq_0)$ and
  $(E'_0,\fq'_0)$ are $G$-conjugate by Theorem~\ref{th:origin}. Thus
  $\fq_0\in V_{E_0}(Y\da_{E_0})$. Here, one uses that conjugation by
  an element $g\in G$ induces an automorphism of $\bfK(\Inj kG)$ which
  takes $Y$ to $Y^g$, and that multiplication with $g$ induces an
  isomorphism $Y\iso Y^g$.  With Theorem~\ref{th:subgroup} we
  conclude that $\fq\in V_{E}(Y\da_{E})$.
\end{proof}

\subsection*{Stratification for finite groups}

We are ready to establish the stratification of $\bfK(\Inj kG)$ for an
arbitrary finite group, which is due to Benson, Iyengar, and Krause
\cite{Benson/Iyengar/Krause:2011b}.  The proof uses the same strategy
as the one for the classification of thick tensor ideals of
$\bfD^\fin(kG)$ (Theorem~\ref{th:BCR}).

\begin{theorem}\label{th:strat-finite-group}
  For a finite group $G$ and a field $k$, the rigidly-compactly generated tensor triangulated category $\bfK(\Inj kG)$ is
  stratified via the canonical action of $H^*(G,k)$. In
  particular, cohomological support induces a lattice isomorphism
  \begin{equation*}
    \Loc_\otimes(\bfK(\Inj kG))\longiso\{\text{Subsets of }\Spec H^*(G,k)\}\,.
  \end{equation*}
\end{theorem}
\begin{proof}
Fix $X,Y$ in $\bfK(\Inj kG)$ and suppose that \[V_G(X)\subseteq V_G(Y)\]
  holds. For an elementary abelian $p$-subgroup $E\le G$,
 Theorem~\ref{th:subgroup-elem-abelian} implies
 \[V_E(X\da_E)\subseteq V_E(Y\da_E)\]
 and therefore \[\loc(X\da_E)\subseteq\loc(Y\da_E)\] by
 Theorem~\ref{th:strat-elem-abelian}. This implies
 \[\loc(X\da_E\ua^G)\subseteq \loc(Y\da_E\ua^G)\]
 and then Theorem~\ref{th:filtration-finite-group} implies
 \[\loc_\otimes(X)=\loc\left(\bigoplus_{i=1}^t X\da_{E_i}\ua^G\right)
   \subseteq \loc\left(\bigoplus_{i=1}^t
     Y\da_{E_i}\ua^G\right)=\loc_\otimes(Y).\] It remains to apply
 Proposition~\ref{pr:stratification-ring} which establishes the
 isomorphism between $\Loc_\otimes(\bfK(\Inj kG))$ and the subsets of
 $\Spec H^*(G,k)$.
\end{proof}

\begin{remark}
  The full subcategory of objects in $\bfK(\Inj kG)$ that are supported
  at the maximal ideal $H^+(G,k)$ of positive degree elements
  identifies with the derived category $\bfD(\Mod kG)$. On the other
  hand, the stable category $\StMod kG$ of the module category over
  $kG$ identifies with the full subcategory of acyclic complexes in
  $\bfK(\Inj kG)$, by sending a module to a \emph{complete resolution}, which
  is obtained by splicing together a projective resolution and an injective
  coresolution. Thus the tensor triangulated category $\StMod kG$ is
  stratified via
  \[\Proj H^*(G,k)=\Spec H^*(G,k)\smallsetminus\{H^+(G,k)\}.\]
\end{remark}

\end{document}